\documentclass{amsart}

\usepackage[T1]{fontenc}
\usepackage[utf8]{inputenc} 

\usepackage[all,cmtip]{xy}

\usepackage{enumerate}
\usepackage{enumitem}
\usepackage[colorlinks=true,hyperindex,linktocpage=true]{hyperref}
\usepackage[top=1in, left=1in, right=1in, bottom=1in]{geometry}
\usepackage{tikz-cd}

\usepackage{amsmath, amssymb, amsthm}
\usepackage{mathtools}
\usepackage{mathabx}
\usepackage{stmaryrd}
\usepackage{comment}

\theoremstyle{plain}
\newtheorem{thm}{Theorem}[section]
\newtheorem{prop}[thm]{Proposition}
\newtheorem{lem}[thm]{Lemma}
\newtheorem{cor}[thm]{Corollary}

\theoremstyle{definition}

\newtheorem{defn}[thm]{Definition}

\theoremstyle{remark}
\newtheorem{remark}[thm]{Remark}

\newcommand{\mr}[1]{\mathrm{#1}}
\newcommand{\citestacks}[1]{\cite[\href{https://stacks.math.columbia.edu/tag/#1}{Tag #1}]{stacks-project}}

\newcommand{\br}[1]{\left( #1 \right)}
\newcommand{\dbr}[1]{\left( \! \left( #1 \right) \!  \right)}

\newcommand{\dsq}[1]{\left\llbracket #1 \right\rrbracket}

\newcommand{\se}{\subseteq}

\newcommand{\ph}{\varphi}

\makeatletter
\newcommand{\xdashrightarrow}[2][]{\ext@arrow 0359\rightarrowfill@@{#1}{#2}}
\def\rightarrowfill@@{\arrowfill@@\relax\relbar\rightarrow}
\def\arrowfill@@#1#2#3#4{%
	$\m@th\thickmuskip0mu\medmuskip\thickmuskip\thinmuskip\thickmuskip
	\relax#4#1
	\xleaders\hbox{$#4#2$}\hfill
	#3$%
}
\makeatother

\newcommand{\Ec}{\mathcal{E}}
\newcommand{\Fc}{\mathcal{F}}
\newcommand{\Gc}{\mathcal{G}}

\newcommand{\Ic}{\mathcal{I}}
\newcommand{\Jc}{\mathcal{J}}
\newcommand{\Kc}{\mathcal{K}}
\newcommand{\Lc}{\mathcal{L}}

\newcommand{\Oc}{\mathcal{O}}

\newcommand{\Xc}{\mathcal{X}}

\newcommand{\OX}{\Oc_X}
\newcommand{\OS}{\Oc_S}
\newcommand{\OXS}{\Oc_{X_S}}
\newcommand{\ODS}[1]{\Oc_{D_{#1,S}}}
\newcommand{\ODnS}{\ODS{n}}

\newcommand{\A}{\mathbb{A}}

\newcommand{\D}{\mathbb{D}}
\newcommand{\F}{\mathbb{F}}

\newcommand{\N}{\mathbb{N}}
\newcommand{\Q}{\mathbb{Q}}

\newcommand{\Z}{\mathbb{Z}}

\renewcommand{\P}{\mathbb{P}}

\newcommand{\Fq}{{\F_q}}

\newcommand{\GaL}{\mathbb{G}_{a, \Lc}}

\newcommand{\GaR}{\mathbb{G}_{a, R}}

\newcommand{\Eb}{\mathbf{E}}

\newcommand{\Gb}{\mathbf{G}}
\newcommand{\Hb}{\mathbf{H}}

\newcommand{\af}{\mathfrak{a}}
\newcommand{\ff}{\mathfrak{f}}

\newcommand{\pf}{\mathfrak{p}}

\DeclareMathOperator{\id}{id}

\DeclareMathOperator{\coker}{coker}

\DeclareMathOperator{\Def}{Def}

\DeclareMathOperator{\End}{End}
\DeclareMathOperator{\Hom}{Hom}

\DeclareMathOperator{\Spec}{Spec}
\DeclareMathOperator{\Spf}{Spf}

\DeclareMathOperator{\Frob}{Frob}
\newcommand{\FrobS}{\Frob_S}

\newcommand{\EEnd}{\mathcal{E}nd}

\DeclareMathOperator{\Sym}{Sym}

\DeclareMathOperator{\GL}{GL}

\DeclareMathOperator{\Bun}{Bun}
\DeclareMathOperator{\Sht}{Sht}

\DeclareMathOperator{\Dmod}{Dr-Mod}
\newcommand{\Nilp}{\mathcal{N}ilp}

\newcommand{\Ell}{\Ec\ell\ell}
\DeclareMathOperator{\Gr}{Gr}
\DeclareMathOperator{\Dr}{Dr}

\newcommand{\sMatrix}[4]{\br{\begin{smallmatrix}#1 & #2\\#3 & #4 \end{smallmatrix}}}

\begin{document}

	\title{Compactification of Level Maps of Moduli Spaces of {Drinfeld~Shtukas}}
	\author[P.\,Bieker]{Patrick Bieker}
	\address{ Fakultät für Mathematik, Universität Bielefeld, Postfach 100 131, 33501 Bielefeld, Germany}
	\email{pbieker@math.uni-bielefeld.de}

		\begin{abstract}
			We define Drinfeld level structures for Drinfeld shtukas of any rank and show that their moduli spaces are regular and admit finite flat level maps.
			In particular, the moduli space of Drinfeld shtukas with Drinfeld $\Gamma_0(\mathfrak{p}^n)$-level structures provides a good integral model and a relative compactification of the moduli space of shtukas with naive  $\Gamma_0(\mathfrak{p}^n)$-level defined using shtukas for dilated group schemes.
		\end{abstract}

	\maketitle
	
	\section{Introduction}

Moduli spaces of shtukas are function field analogues of Shimura varieties and thus play a central role in the Langlands programme over function fields.
While a lot of progress has been made in understanding the geometry of moduli spaces of shtukas for general reductive groups with parahoric level, compare for example \cite{Rad2015}, \cite{Rad2017}, \cite{Breutmann2019}, \cite{Yun2019} and \cite{Zhu2014}, little is known for deeper level structures.

The goal of this work is to study deeper level structures in the case of Drinfeld shtukas (that means $\GL_{r}$-shtukas for some fixed $r\geq 1$ with two legs bounded by the minuscule cocharacters $(0, \ldots, 0, -1)$ and $(1, 0, \ldots, 0)$). 
More explicitly, we define for all $n \geq 0$ an analogue of $\Gamma_0(p^n)$-level structures on elliptic curves that behaves well at the place of bad reduction.
Let us explain the construction in more detail. 

Let $X$ be a smooth, projective and geometrically connected curve over a finite field $\Fq$.
Let us fix an $\Fq$-rational point $\infty$ of $X$ and let us denote $X' = X \setminus \{\infty\}$.
A \emph{Drinfeld shtuka} of rank $r$ (with pole fixed at $\infty$) over a scheme $S$ is given by the data
$\underline{\Ec} = (x, \Ec, \ph),$
	where
	\begin{itemize}
		\item $x \in X'(S)$ is the characteristic section (also called leg or paw),
		\item $\Ec$ is a rank $r$ vector bundle on $X_S$ and
		\item $\ph \colon \sigma^\ast \Ec|_{X_S \setminus (\Gamma_{x} \cup \Gamma_\infty)} \xrightarrow{\cong} \Ec|_{X_S \setminus (\Gamma_{x}\cup \Gamma_\infty)}$ is an isomorphism of $\OXS$-modules away from the graphs $\Gamma_{x}$ of $x$ and $\Gamma_{\infty}$ of $\infty$, such that $\ph$ extends to a map 
		$\ph \colon \sigma^\ast \Ec|_{X'} \to \Ec|_{X'} $
		with $\coker(\ph)$ supported on $\Gamma_{x}$ and invertible on its support, and
		$\ph^{-1}$ extends to a map 
		$\ph^{-1} \colon \Ec|_{X \setminus \Gamma_x} \to \sigma^{\ast} \Ec|_{X \setminus \Gamma_x}$
		with $\coker(\ph^{-1})$ supported on $\Gamma_\infty$ and invertible on its support.
	\end{itemize}
We denote by $\Sht_{r}$ the stack of Drinfeld shtukas of rank $r$, it is a Deligne-Mumford stack locally of finite type over $\Fq$. The projection to the characteristic section defines a map $\Sht_{r} \to X'$, which is smooth of relative dimension $(2r-2)$. 
Moduli spaces of Drinfeld shtukas are used with great success in \cite{Lafforgue2002} to establish the Langlands correspondence for $\GL_{r}$.

Drinfeld shtukas in this sense are a generalisation of Drinfeld modules (compare Proposition \ref{propFunDModSht} for a precise statement).
Moduli spaces of rank 2 Drinfeld modules can be thought of as a function field analogue of the moduli space of elliptic curves.
 
In the case of modular curves with $\Gamma_1$- and $\Gamma_0$-level structures, the reduction modulo $p$ was studied by \cite{Deligne1973a} using a normalisation procedure.
\cite{Katz1985} gave an explicit moduli description of an integral model using Drinfeld level structures.
This notion goes back to \cite{Drinfeld1976}, who first introduced this notion of level structure to study full level structures of Drinfeld modules.
The analogy to the modular curve suggests a strategy to construct good integral models in our shtuka setting, in other words, to define a good notion of level structure that also behaves as desired at a place of bad reduction: to use Drinfeld level structures. 

Let us fix a $\Fq$-rational point $0$ of $X$ and denote by $\Oc_0$ (respectively $\pf = \pf_0$) the completion of the local ring of $X$ and $0$ (respectively its maximal ideal).
In order to define Drinfeld level structures for shtukas, we associate to a Drinfeld shtuka its \emph{scheme of $\pf^n$-torsion points} $\underline{\Ec}[\pf^n]$. This was essentially constructed in \cite{Drinfeld1987a} and shares similar properties with the scheme of $\pf^n$-torsion points of a Drinfeld module (respectively the scheme of $p^n$-torsion points of an elliptic curve).  
It is an $\Oc_0/\pf^n$-module scheme which is finite locally free of rank $q^{nr}$ over $S$. 
Moreover, we show that \'etale locally on $S$ we find an embedding of $\underline{\Ec}[\pf^n]$ as a Cartier divisor into $\A^1_S$ by adapting a similar result for the $p^n$-torsion of one-dimensional $p$-divisible groups of \cite{Frimu2019} (compare Proposition \ref{propEmbItor} and  Remark \ref{remEmbPTorFrimu}).

This allows us to define Drinfeld $\Gamma_0(\pf^n)$-level structures on Drinfeld shtukas of rank 2 as follows.
\begin{defn}[compare Definition \ref{defnGamma0Sht} for general rank]
	A \emph{$\Gamma_0(\pf^n)$-level structure} on a rank 2 shtuka $\underline{\Ec}$ is a finite locally free closed subscheme $\Hb \se \underline{\Ec}[\pf^n]$ of rank $q^n$ that admits a generator fppf-locally on $S$, that means an $\Oc_0/\pf^n$-linear map
	$ \iota \colon (\pf^{-n}/\Oc_0) \to \underline{\Ec}[\pf^n](S)$
	such that after the choice of an embedding $\Ec[\pf^n] \hookrightarrow \A^1_S$ we have
	$$ \sum_{\alpha \in \pf^{-n}/\Oc_0} [\iota(\alpha)] = \Hb \qquad \text{and} \qquad \sum_{\alpha \in \pf^{-1}/\Oc_0} [\iota(\alpha)] \se \underline{\Ec}[\pf]$$
	as Cartier divisors in $\A^1_S$.
\end{defn}
Note the subtle difference compared to the definition of $\Gamma_0(p^n)$-level structures in \cite{Katz1985}: in the setting of elliptic curves the second condition is automatic. However, in our setting the second condition is in particular necessary to get well-defined level maps, see Remark \ref{remKMlevelCounterEx} for an explicit counterexample.

Adapting the theory of Drinfeld level structures for elliptic curves in \cite{Katz1985}, we obtain the following.
\begin{thm}[compare Theorem \ref{thmG0Reg}]
	Let $r \geq 1$ and $n \geq 0$ be integers.
	\begin{enumerate}
		\item The moduli stack $\Sht_{r, \Gamma_0(\pf^n)}$ of rank $r$ Drinfeld shtukas with Drinfeld $\Gamma_0(\pf^n)$-level structures is representable by a regular Deligne-Mumford stack locally of finite type over $\Fq$. 
		\item The level map $\Sht_{r, \Gamma_0(\pf^n)} \to \Sht_r$ is schematic, finite and flat. Moreover, it is finite \'etale away from $\pf$. 
	\end{enumerate}
\end{thm}
As in the case of elliptic curves in \cite[Chapter 5]{Katz1985}, we first show the corresponding results for $\Gamma_1(\pf^n)$-level structures. 
The main step in the proof of the $\Gamma_1(\pf^n)$ case is the study of the deformation theory at supersingular points, where we rely on results of \cite{Drinfeld1976}. 

Using the flatness of the moduli space, we construct a compatible system of level maps 
$$\Sht_{r, \Gamma_0(\pf^n)} \to \Sht_{r, \Gamma_0(\pf^m)}$$
that are finite and flat for all $m \leq n$.
The level maps allow us to interpret our construction in the following way in terms of the combinatorics of the Bruhat-Tits building of $\GL_{r}$ over the fraction field $K_0$ of $\Oc_0$.
Let us denote by  $\Sht_{r, \Omega}$ the moduli stack of shtukas for the Bruhat-Tits group scheme $\GL_{r, \Omega}$ bounded by $(0,\ldots,0,-1)$ and $(1, 0, \ldots,0)$ corresponding to the standard simplex $\Omega$ of sidelength $n$ in the standard apartment of the Bruhat-Tits building.
To such a bounded $\GL_{r, \Omega}$-shtuka we can associate a Drinfeld shtuka with Drinfeld $\Gamma_0(\pf^n)$-level structure, this is explained in more detail below.
Moreover, using the level maps, we get compatible system of maps $\Sht_{r, \Gamma_0(\pf^n)} \to  \Sht_{r,\ff}$ to the moduli space of shtukas for Bruhat-Tits group schemes for all facets $\ff \prec \Omega$ contained in the closure of $\Omega$.
\begin{thm}[compare Theorem \ref{thmDrinfeldClosureBT}]
	\label{thmDrinfeldClosureBTIntro}
	The map $\Sht_{r,\Omega} \to \varprojlim_{\ff \prec \Omega} \Sht_{r,\ff}$ is a quasi-compact open immersion and an isomorphism away from $0$. Its schematic image in the sense of \cite{Emerton2021} is 
	$\overline{\Sht}_{r, \Omega} = \Sht_{r, \Gamma_0(\pf^n)}$
	via the maps 
	$$\Sht_{r,\Omega} \hookrightarrow \Sht_{r, \Gamma_0(\pf^n)} \hookrightarrow  \varprojlim_{\ff \prec \Omega} \Sht_{r,\ff}$$
	constructed above.
	In the parahoric case $n = 1$, the map $\Sht_{r,\Omega} \to \Sht_{r, \Gamma_0(\pf)}$ is an isomorphism.
\end{thm}

Note that this also shows that $\overline{\Sht}_{r, \Omega} = \Sht_{r, \Gamma_0(\pf^n)}$ is the flat closure of the generic fibre inside $\varprojlim_{\ff \prec \Omega} \Sht_{r,\ff}$. 
Theorem \ref{thmDrinfeldClosureBTIntro} suggests that a candidate for a good integral model for the moduli spaces of shtukas for a general reductive group with deep Bruhat-Tits level structure (i.e. level structures generalising $\Gamma_0(\pf^n)$-level structures in the $\GL_r$-case) might be the closure of the moduli stack of shtukas for the Bruhat-Tits group scheme inside the limit of all moduli stacks with corresponding parahoric level.
We study this more general construction in \cite{Bieker2022a}.

We can also interpret this result as follows.  We denote for $n \in \N$ by $D_{n} = n[0] \se X$ the effective Cartier divisor in $X$ defined by (multiples of) the point $0$. Then $D_n = \Spec(\Oc_0/\pf^n)$. 
For a Drinfeld shtuka $\underline{\Ec} \in \Sht_r(S)$ over an ($\Fq$-)scheme $S$, we denote by $\underline{\Ec}|_{D_{n,S}}$ its pullback to $D_{n,S}$.  
A \emph{naive} $\Gamma_0(\pf^n)$-level structure on a rank 2 shtuka $\underline{\Ec}$ is given by a quotient of $\Oc_{D_{n,S}}$-modules $\Ec \twoheadrightarrow \Lc$ such that $\Lc$ is finite locally free of rank 1 on $\Oc_{D_{n,S}}$ and such that $\ph$ descends to a map $\sigma^* \Lc \to \Lc$.
Following \cite{Mayeux2020} we can view Drinfeld shtukas with naive $\Gamma_0(\pf^n)$-level structures as shtukas for the Bruhat-Tits group scheme $\GL_{r, \Omega}$ bounded by $\underline{\mu} = ((0,\ldots,0,-1),(1,0,\ldots,0))$. For a precise definition see Section \ref{secBT} below.

By the analogy with the modular curve, in the fibre of the moduli space of shtukas with $\Gamma_0(\pf^n)$-level structures over $0$ we should expect to find  $n+1$ components intersecting at supersingular points. However, it can be shown that in the non-parahoric case (in other words when $n \geq 2$), the fibre over 0 of the moduli space of Drinfeld shtukas with naive level structures as above only has two components which moreover do not intersect. In particular, its supersingular points are missing (compare Remark \ref{remNaiveLvlBad}).

Another way to phrase Theorem \ref{thmDrinfeldClosureBTIntro} is that we compactified the level map $ \Sht_{r, \Omega} \to \Sht_r$, which we saw in the example above is not proper in general, by factoring it in an open immersion with dense image followed by a finite (hence proper) and surjective map 
$$ \Sht_{r, \Omega} \hookrightarrow \Sht_{r, \Gamma_0(\pf^n)} \to \Sht_r.$$
In particular, in the rank 2 example $\Sht_{2, \Gamma_0(\pf^n)}$ acquires the supersingular points missing in $\Sht_{2, \Omega}$.

\subsection*{Organisation.}
This paper is organised as follows. In Section 2, we recall some facts on shtukas (in particular Drinfeld shtukas) and define naive $\Gamma_0(\pf^n)$-level structures. In Section 3, we explain the comparison with Drinfeld modules. This provides us with a way to associate to (global, local and finite) shtukas group schemes, which is what makes it possible to define Drinfeld level structures in the first place.
In particular, we construct the scheme of $\pf^n$-division points of a Drinfeld shtuka and study its properties.
In Sections 4 and 5, we define our Drinfeld ($\Gamma_1$- and $\Gamma_0$-type) level structures and prove the regularity of their moduli spaces. For this, we follow \cite{Katz1985}. In Section 6, we show that the Drinfeld level structures actually provide a good (relative) compactification of the moduli space with naive level structure  $ \Sht_{r, \Omega}$.

	\subsection*{Acknowledgements.}
	
	First of all I thank my advisor Timo Richarz for introducing me to this topic, his steady encouragement and his interest in my work. 
	I thank Paul Hamacher, Jo\~{a}o Louren\c{c}o, Lenny Taelman and Torsten Wedhorn for helpful conversations surrounding this work. 
	I thank Gebhard Böckle and Urs Hartl for their comments on a previous version of this paper.
	I thank Tasho Kaletha for sharing a preliminary version of \cite{Kaletha} with us.
	
	I thank the anonymous referee for their careful reading of the paper and their valuable comments and suggestions.
	
	This work was partially funded by the Deutsche Forschungsgemeinschaft (DFG, German Research Foundation) TRR 326 \textit{Geometry and Arithmetic of Uniformized Structures}, project number 444845124.

	\subsection*{Notation.}
	We fix the following notation. Let $\Fq$ be a finite field with $q$ elements, let $p$ be the characteristic of $\Fq$. All schemes will be $\Fq$-schemes unless otherwise specified.
Let $X$ be a smooth projective and geometrically connected curve over $\Fq$ with function field $K$.
For a closed point $x$ of $X$ we denote by $\Oc_{X,x}$ the local ring at $x$ and by $\Oc_x$ its completion. Moreover, we denote by $K_x$ the completion of $K$ at $x$.

We fix two distinct $\Fq$-rational points $\infty$ and $0$ of $X$ and denote by $\pf = \pf_0$ the maximal ideal in the complete local ring $\Oc_0$ at $0$.
Let us also fix a uniformiser $\varpi$ of $\pf$. 

We denote by $\sigma$ the (absolute) $q$-Frobenius endomorphism $\FrobS$ of some $\Fq$-scheme $S$, and also the map $\sigma = id_X \times \FrobS \colon X_S \to X_S$. It is always clear from context which map $\sigma$ is meant.

	\section{Moduli spaces of shtukas and naive $\Gamma_0(\pf^n)$-level structures}

Drinfeld shtukas were introduced in \cite{Drinfeld1987a} as \emph{elliptic sheaves} and were vastly generalised to arbitrary reductive groups or even general smooth affine group schemes by \cite{Varshavsky2004} and \cite{Rad2015}, respectively.

We introduce naive $\Gamma_0(\pf^n)$-level structures on Drinfeld shtukas, present how to encode these level structures in terms of Bruhat-Tits group schemes following \cite{Mayeux2020} and explain why the naive definition is  not appropriate for deeper level (that means for $n > 1$).

\subsection{Global shtukas.}
We recall the definitions of global shtukas and isogenies of Drinfeld shtukas.
We restrict ourselves to shtukas with two legs with one leg fixed at the point $\infty$.
	
\begin{defn}[\cite{Rad2015}]
	\label{defGShtuka}
	Let $G$ be a smooth affine group scheme on $X$. A \emph{global $G$-shtuka} over a scheme $S$ is given by the data
	\[
	\underline{\Ec}  = (x, \Ec, \ph \colon \sigma^\ast\Ec \dashrightarrow \Ec),
	\] 
	where 
	\begin{itemize}
		\item $x \in X'(S)$ is a section of $X' = X \setminus \{ \infty\}$,
		\item $\Ec$ is a $G$-bundle on $X_S$ and
		\item $\ph \colon \sigma^\ast \Ec|_{X_S \setminus (\Gamma_{x} \cup \Gamma_\infty)} \xrightarrow{\cong} \Ec|_{X_S \setminus (\Gamma_{x}\cup \Gamma_\infty)}$ is an isomorphism of $G$-bundles away from the graphs $\Gamma_{x}$ of $x$ and $\Gamma_{\infty}$ of $\infty$.
	\end{itemize}
	The point $x$ is called \emph{characteristic} or \emph{leg} of $\underline{\Ec}$.
	A map of $G$-shtukas is a tuple of maps of $G$-bundles compatible with the maps $\ph$ and $\ph'$.	
\end{defn}

Note that there are several ways to bound the zeros (and poles, respectively) of $G$-shtukas, and in general they are not equivalent (compare Remark \ref{remBounds}). 
We will mostly be interested in the case of Drinfeld shtukas, that means we consider $G = \GL_{r}$ (or corresponding Bruhat-Tits group schemes) and bounds given by the minuscule coweights $\underline{\mu} = ((0, \ldots, 0,-1),(1,0,\ldots,0))$. 
These admit the following explicit description.

\begin{defn}[\cite{Drinfeld1987a}]
	\label{defDSht}
	A \emph{Drinfeld shtuka} of rank $r$ over a scheme $S$ is given by the data
	$$\underline{\Ec} = (x, \Ec, \ph),$$
	where
	\begin{itemize}
		\item $x \in X'(S)$ is the characteristic section,
		\item $\Ec$ is a rank $r$ vector bundle on $X_S$ and
		\item $\ph \colon \sigma^\ast \Ec|_{X_S \setminus (\Gamma_{x} \cup \Gamma_\infty)} \xrightarrow{\cong} \Ec|_{X_S \setminus (\Gamma_{x}\cup \Gamma_\infty)}$ is an isomorphism of $\OXS$-modules away from the graphs $\Gamma_{x}$ of $x$ and $\Gamma_{\infty}$ of $\infty$, such that $\ph$ extends to a map 
		$\ph \colon \sigma^\ast \Ec|_{X'} \to \Ec|_{X'} $ 
		with $\coker(\ph)$ supported on $\Gamma_{x}$ and invertible on its support, and
		$\ph^{-1}$ extends to a map 
		$\ph^{-1} \colon \Ec|_{X \setminus \Gamma_x} \to \sigma^{\ast} \Ec|_{X \setminus \Gamma_x}$
		with $\coker(\ph^{-1})$ supported on $\infty$ and invertible on its support.
	\end{itemize}
	We denote by $\Sht_{r}$ the stack of Drinfeld shtukas of rank $r$. 
\end{defn} 

It is well known that $\Sht_{r}$ is a Deligne-Mumford stack locally of finite type over $\Fq$. It has a forgetful map $\Sht_{r} \to X'$ which is smooth of relative dimension $(2r-2)$, see \cite[Proposition 3.2 and 3.3]{Drinfeld1987a}.

In the context of Drinfeld shtukas, the characteristic section $x$ is often called the \emph{zero} of $\Ec$ while the second leg (that we fixed to be $\infty$) is the \emph{pole} of $\underline{\Ec}$.
By a slight abuse of notation we say that $\underline{\Ec}$ is \emph{in characteristic $\pf$} if its characteristic section factors through $0$. 
\begin{remark}
	\label{remdefDSht}
	Note that once the zero and the pole of the shtuka do not intersect, we can glue $\Ec$ and $\sigma^* \Ec$ along the isomorphism $\ph$ over $X_S \setminus (\Gamma_{x} \cup \Gamma_{x'})$ and obtain a vector bundle $\Ec'$ together with maps 
	$$ \ph' \colon \Ec \hookrightarrow \Ec' \hookleftarrow \sigma^\ast \Ec \colon \ph $$
	of $\OXS$-modules that satisfy the analogous conditions on the cokernels as in our definition of Drinfeld shtukas. 
	This notion is used in the original definition of Drinfeld shtukas in \cite{Drinfeld1987a} and does not require the two legs of the shtuka to be disjoint.
	We denote by $\Sht_{r, X^2} \to X^2$ the stack of Drinfeld shtukas in this sense. Then $\Sht_r = \Sht_{r, X^2} \times_{X^2}( \{\infty\} \times X')$.
\end{remark}

For $n \in \N$ we denote by $D_n = n[0] \se X$ the effective Cartier divisor in $X$. Note that $D_n = \Spec(\Oc_0/\pf^n)$.
\begin{defn}
	A map $f \colon \underline{\Ec}_1 \to \underline{\Ec}_2$ of Drinfeld shtukas is an \emph{isogeny} if $f$ is injective and $\coker(f)$ is finite locally free as $\OS$-module. Moreover, we say that $f$ is a $\pf^n$-isogeny, if the $\Oc_{X_S}$-module structure on $\coker(f)$ factors through $\Oc_{D_{n,s}}$, in other words, if $\coker(f)$ is $\pf^n$-torsion.
\end{defn}

In order to give a criterion which $\OS$-modules can arise as cokernels of $\pf^n$-isogenies, we use the following notion of a \emph{$\pf^n$-torsion shtuka}, which are an $\Oc_0/\pf^n$-linear analogue of the $\ph$-sheaves introduced by \cite{Drinfeld1987a}.
\begin{defn}
	A \emph{$\pf^n$-torsion shtuka} over $S$ is a pair $\underline{\Fc} = (\Fc, \ph)$ consisting of a quasi-coherent $\ODnS$-module $\Fc$ which is finite locally free as $\OS$-module and an $\ODnS$-module homomorphism $\ph \colon \sigma^* \Fc \to \Fc$. A map of $\pf^n$-torsion shtukas is a map of the underlying $\ODnS$-modules compatible with $\ph$. We say that a $\pf^n$-torsion shtuka is \emph{\'etale} if $\ph$ is an isomorphism.  
\end{defn}
In \cite{Hartl2019} Drinfeld's $\ph$-sheaves are also called finite shtukas. For our purposes however, the $\Oc_{D_n}$-module structure is central.

To a rank $r$ Drinfeld shtuka $\underline{\Ec} =(x, \Ec, \ph)$ over $S$ we associate its $\pf^n$-torsion shtuka defined as the pullback of $\Ec$ to the divisor $D_{n,S}$, which is more explicitly given by $\underline{\Ec}|_{D_{n,S}} = (\Ec|_{D_{n,S}}, \ph|_{D_{n,S}})$. Note that its underlying $\OS$-module has rank ${nr}$.
A second important class of examples of $\pf^n$-torsion shtukas are cokernels of $\pf^n$-isogenies of Drinfeld shtukas. Note that $\underline{\Ec}|_{D_{n,S}}$ is the cokernel of the $\pf^{n}$-isogeny $\underline{\Ec}(\pf^n) \hookrightarrow \underline{\Ec}$, where we denote by $\underline{\Ec}(\pf^n) = \underline{\Ec} \otimes \Oc(D_{n,S})$ the twist of $\underline{\Ec}$ by the divisor $D_n$.  

\subsection{Local shtukas.}
We can associate to Drinfeld shtukas its local counterparts called \emph{local shtukas} in the same way $p$-divisible groups are local analogues of abelian varieties. 
Local shtukas are introduced as Dieudonn\'e $\Fq\dsq{\varpi}$-modules in \cite{Hartl2005} as analogues of Dieudonn\'e modules of $p$-divisible groups and are studied and generalised for example by \cite{Hartl2011} and \cite{Rad2015}.

Let us denote by $\Fq\dsq{\zeta}$ the ring of formal power series in the formal variable $\zeta$ and by $\Nilp_{\Fq \dsq{\zeta}}$ the category of schemes $S$ over $\Fq\dsq{\zeta}$ such that $\zeta$ is locally nilpotent in $S$.
For a ring $R$, we denote by $R\dsq{\varpi}$ the ring of formal power series in the formal variable $\varpi$ and by $R\dbr{\varpi}$ the ring of formal Laurent series in $\varpi$ on $S$. 
Note that for $\Spec(R) \in \Nilp_{\F \dsq{\zeta}}$, we have $R\dbr{\varpi} = R\dsq{\varpi}\left[ \frac{1}{\varpi - \zeta}\right]$. 
We denote by $\sigma$ the endomorphism of $R \dsq{\varpi}$ (respectively $R \dbr{\varpi}$) that acts as the identity on $\varpi$ and as $b \mapsto b^q$ on $R$.


\begin{defn}
	Let $S = \Spec(R) \in \Nilp_{\Fq \dsq{\zeta}}$. A \emph{local shtuka} $\underline{\Gc} = (\Gc, \ph)$ of rank $r$ over $S$ is a locally free sheaf of $R\dbr{\varpi}$-modules $\Gc$ of rank $r$ together with an isomorphism $\ph \colon \sigma^*\Gc[\frac{1}{\varpi-\zeta}] \to \Gc[\frac{1}{\varpi-\zeta}]$ of $R\dbr{\varpi}$-modules. The local shtuka $\underline{\Gc}$ is called \emph{effective} if $\ph$ comes from a map $\tilde \ph \colon \sigma^*\Gc \hookrightarrow \Gc$ of $R\dsq{\varpi}$-modules and \emph{\'etale} if additionally $\tilde \ph$ is an isomorphism. 
	A \emph{quasi-isogeny} $f \colon \underline{\Gc} \to \underline{\Gc'}$ between local shtukas is an isomorphism $\Gc\left[ \frac{1}{\varpi} \right] \tilde{\to} \Gc'\left[ \frac{1}{\varpi} \right]$ of the underlying $R\dbr{\varpi}$-modules, which is compatible with $\ph$ and $\ph'$. 
\end{defn}

We say a local shtuka $\underline{ \Gc} = (\Gc, \ph)$ is \emph{bounded by $(1,0,\ldots,0)$} if it is effective, $\coker(\ph)$ is locally free of rank 1 as an $R$-module and $(\varpi-\zeta)$ annihilates $\coker(\ph)$. Similarly, we say $\underline{\Gc}$ is bounded by $(0, \ldots,0,-1)$ if $\ph^{-1}$ is bounded by $(1,0,\ldots,0)$ in the above sense. More precisely, $\underline{\Gc}$ is bounded by $(0, \ldots,0,-1)$ if $\ph^{-1}$ induces a map $\Gc \hookrightarrow \sigma^* \Gc$ with a cokernel which is locally free of rank 1 as $R$-module and which is annihilated by $(\varpi-\zeta)$.

\begin{remark}
	\label{remBounds}
	There are several ways to define bounds for local shtukas in general, cf. \cite[Definition 3.5 and Lemma 4.3]{Hartl2011} and \cite[Definition 4.8]{Rad2015}. For the Drinfeld case the bound in the sense of \cite{Rad2015} is also given more explicitly in \cite[Section 7.2]{Breutmann2019}. 
	Note that the straightforward generalisation of our definition above does not produce the correct notion for coweights $(d, 0, \ldots,0)$ with $d > 1$ by Example 8.4 in the arXiv version of \cite{Hartl2019}. In particular \cite[Example 4.5]{Hartl2011} and \cite[Example 2.1.8]{Zhu2016} seem to be problematic.
\end{remark} 

The Newton stratification for local shtukas is defined in \cite{Hartl2011} as an analogue of the Newton stratification for $F$-isocrystals in \cite{Rapoport1996}. 
\begin{defn}
	\label{defNewtonPoint}
	The \emph{Newton point} of of a local shtuka $\underline{\Gc}$ of rank $r$ over an algebraically closed field $\ell$ is $(u_1, \ldots, u_r) \in \Q^r$ with $u_1 \geq \ldots \geq u_r$ and the $u_i$ are the slopes associated to the corresponding isoshtuka $\underline{\Gc}\left[ \varpi^{-1} \right]$ by the Dieudonn\'e-Manin classification in the function field case \cite[Theorem 2.4.5]{Laumon1996}.
\end{defn}
We denote by $B(\GL_r)$ the \emph{Kottwitz set} of isomorphism classes of isoshtukas over an algebraically closed field $\ell$, in other words, the set of $\sigma$-conjugacy classes of invertible $(r \times r)$-matrices over $\ell \dbr{\varpi}$. The set $B(\GL_r)$ does not depend on the choice of $\ell$. 
Recall that the Newton map $\nu_{\GL_r} \colon B(\GL_r) \to \Q^r$ is already injective (this fails for general reductive groups). The Bruhat order on the space of cocharacters $X_\ast(T) \otimes_\Z \Q \cong \Q^r$ induces a partial order on $B(\GL_r)$. 
It is more explicitly given by 
$$ (u_1, \ldots, u_r) \leq (u_1', \ldots, u_r') \qquad \text{if} \qquad \sum_{j = 1}^i u_j \leq \sum_{j = 1}^i u_j' $$
for all $1 \leq i \leq r$ with equality in the case $i = r$. Moreover, for a dominant cocharacter $\mu$, in other words, $\mu = (\mu_1, \ldots, \mu_r) \in \Z^r$ with $\mu_1 \geq \ldots \geq \mu_r$,  we denote by $B(\GL_r, \mu) = \{[b] \in B(\GL_r) \colon \nu_{\GL_r} ([b]) \leq \mu\}$. 

For a local shtuka $\underline{\Gc}$ over a scheme $S = \Spec(R)$ and a geometric point $s$ of $S$ we denote by $[\underline{\Gc}_s]$ the associated point in $B(\GL_r)$ after pullback to $s$. 
Note that if  $\underline{\Gc}$ is bounded by $\mu = (1,0,\ldots,0)$, then $[\underline{\Gc}_s]$ is contained in $B(\GL_r, \mu)$ for all $s \in S$. 
The Newton point induces a stratification in the following way.
\begin{prop}[{\cite[Theorem 7.3]{Hartl2011}}, compare also {\cite[Theorem 3.6]{Rapoport1996}}]
	Let $S = \Spec(R)$ be an affine $\Fq$-scheme and $\underline{\Gc}$ be a local shtuka over $S$ and $b \in B(\GL_r)$. Then the set $\{ s \in S \colon [\underline{\Gc}_s] \leq b\}$ is a Zariski-closed subset of $S$. Furthermore, $\{ s \in S \colon [\underline{\Gc}_s] = b\}$ is an open subset of the former.
\end{prop}
We denote by $S_{\leq b}$ the closed subscheme of $S$ given by the reduced subscheme on $\{ s \in S \colon [\underline{\Gc}_s] \leq b\}$ and similarly $S_{b}$ the corresponding open subscheme of $S_{\leq b}$. Then $S_{b}$ is a locally closed subscheme of $S$. 

\subsection{Global-to-local functor and a Serre-Tate theorem.}
We explain how to associate local shtukas to global shtukas. We follow the general construction of \cite{Rad2015}. This is a generalisation of the construction of \cite[Section 8]{Bornhofen2011} for abelian sheaves and Anderson motives. 

We follow the notation of \cite[Section 5.2.]{Rad2015}.
Let $y$ be a closed point of $X$, which we assume for simplicity to be defined over $\Fq$. This is the only case we use later. For the general construction we refer to \cite{Rad2015}. 
Let $\Oc_y$ be the completed local ring at $y$. 
The choice of a uniformiser $\varpi_y$ at $y$ defines an isomorphism $\Oc_y \cong \Fq\dsq{\varpi_y}$. Let $x \in X(\Spec(R))$ be a section that factors through $\Spf(\Oc_y)$. Then $\varpi_y$ is nilpotent in $R$.
Let $\D_y = \Spec(\Oc_y)$ and $\hat{\D}_y = \Spf(\Oc_y)$. We denote by $\hat{\D}_{y,R} $ the $\varpi_y$-adic completion of $\D_y \times_\Fq \Spec(R)$. 

By \cite[Lemma 5.3.]{Rad2015}, the section $x$ induces a canonical isomorphism of the formal completion of $X_R$ along the graph $\Gamma_x$ of $x$ with $\hat{\D}_{y,R}$.
By construction, the formal completion along $\Gamma_x$ has structure sheaf $R\dsq{\varpi_y - \zeta}$, where $\zeta$ is the image of $\varpi_y$ in $R$. As $\zeta$ is nilpotent in $R$, $R\dsq{\varpi_y - \zeta}$ and $R\dsq{\varpi_y}$ are isomorphic.

We fix a pair $\underline{y} = (y_1, y_2)$ of ($\Fq$-rational) closed points of $X$ with $y_1 \neq y_2$.
Let $\Oc_{\underline{y}}$ be the completion of the local ring of $X^2$ at $\underline{y}$. 
We denote by $\Sht_r^{\underline{y}} = \Sht_{r, X^2} \times_{X^2} \Spf(\Oc_{\underline{y}})$ the substack of $\Sht_{r, X^2}$ such that the legs factor through $\Spf(\Oc_{y_1})$ and $\Spf(\Oc_{y_2})$, respectively.
In particular, for points of $\Sht_r^{\underline{y}}$ the graphs of its legs are disjoint. 
Let $\underline{ \Ec} =(x',x, \Ec, \ph) \in \Sht_r^{\underline{y}}(R)$.
The local shtuka associated to $\underline{\Ec}$ at $y_i$ is then its pullback to $\hat{\D}_{y_i,R} $ for $i = 1,2$. 

\begin{defn}
	\label{defnGlobalLocalFun}
	The \emph{global-to-local functor} associates to a global shtuka $\underline{\Ec} \in  \Sht_r^{\underline{y}}(R)$ a pair of local shtukas (at $y_1$ and $y_2$, respectively) given by 
	$$ \widehat{\underline\Ec_{y_i}} := (\Ec|_{\hat{\D}_{y_i,R}}, \tilde{\ph}) \qquad \text{and} \qquad \widehat{\underline\Ec_{\underline{y}}} =  (\widehat{\underline\Ec_{y_1}}, \widehat{\underline\Ec_{y_2}}).$$
	Then, $\widehat{\underline\Ec_{y_i}}$ is called the \emph{local shtuka} of $\underline{\Ec}$ at $y_i$. 
\end{defn}

By definition of $\Sht_{r, X^2}$, the local shtuka at $y_2$ is bounded by $(1, 0, \ldots, 0)$ as the condition that $(\varpi_{y_2}-\zeta_{y_2})$ annihilates the cokernel in the local case directly corresponds to the fact that the cokernel is supported on the graph in the global case. Similarly, the local shtuka at $y_1$ is bounded by $(0, \ldots,0,-1)$.

The global-to-local functor also gives rise to a Serre-Tate theorem relating the deformation theory of global shtukas with the deformation theory of their associated local shtukas. 
Let $S = \Spec(R) \in \Nilp_{\Oc_{\underline{y}}}$ and let $i \colon \overline{S} = \Spec(R/I) \hookrightarrow S$ be a closed subscheme defined by a nilpotent ideal $I$. 
Let $\underline{\bar \Ec} \in \Sht_r^{\underline y}(\overline{S})$. The category $\Def_{\underline{\bar\Ec}}(S)$ is the category of deformations of $\underline{\Ec}$ to $S$, i.e. the category of pairs $(\underline{\Ec}, f \colon i^*\underline{\Ec} \tilde \rightarrow \underline{\bar{\Ec}} )$ where $\Ec \in \Sht_r^{\underline{y}}(S)$ and $f$ is an isomorphism of shtukas over $\overline{S}$. 
Similarly, for a local shtuka $\underline{\bar \Gc}$ bounded by $(1,0,\ldots,0)$ we define $\Def^{\leq(1,0,\ldots,0)}_{\underline{\bar\Gc}}(S)$ as the category of deformations of  $\underline{\bar \Gc}$ to $S$, i.e. the category of pairs $(\underline{\Gc}, g \colon i^*\underline{\Gc} \tilde \rightarrow \underline{\bar{\Gc}} )$ where $\Gc$ is a local shtuka on $S$ bounded by $(1,0, \ldots 0)$ and $g$ is an isomorphism of local shtukas over $\overline{S}$. Similarly, we define deformations of local shtukas bounded by $(0,\ldots,0,-1)$. 
\begin{prop}[Serre-Tate Theorem for shtukas, {\cite[Theorem 5.10]{Rad2015}}]
	\label{propSerreTateSht}
	Let $\underline{\bar \Ec} \in \Sht_r^{\underline{y}}(\overline{S})$. Then the functor 
	$$\widehat{(-)_{\underline{y}}} \colon \Def_{\underline{\bar\Ec}}(S)
	\to  \Def^{\leq(0,\ldots,0,-1)}_{\widehat{\underline{\overline{\Ec}}_{y_1}}}(S) \times \Def^{\leq(1,0,\ldots,0)}_{\widehat{\underline{\overline{\Ec}}_{y_2}}}(S), \qquad (\underline{\Ec}, f) \mapsto \prod_{i = 1,2}(\widehat{\underline\Ec_{y_i}}, \widehat{f_{y_i}}) $$
	induced by the global-to-local functor is an equivalence of categories.
\end{prop}
\begin{proof}
As before, this follows directly from the unbounded case in \cite[Theorem 5.10.]{Rad2015} as a global $\GL_r$-shtuka is bounded by $(0,\ldots,0,-1), (1,0,\ldots,0)$ if and only if the associated local shtukas are.
\end{proof}

The Newton stratification induces also a stratification on the special fibre of the stack of (global) Drinfeld shtukas in the sense of \cite[Section 4]{Breutmann2019}.
We continue to restrict ourselves to the case of Drinfeld shtukas with one leg fixed at $\infty$ as this is the only case of interest to us in the following. The following has obvious analogues also for $\Sht_{r, X^2}$. 
For a closed point $y$ of $X$ with residue field $\F_y$ different from $\infty$ we set $\Sht_{r, \F_y} = \Sht_r \times_{X', y} \F_y$. 
%

\begin{defn}[compare {\cite[Proposition 4.1.4]{Neupert2016}}, {\cite[Definition 4.12]{Breutmann2019}} for the general definition]
	Let $b_\infty \in B(\GL_{r, K_\infty}, (0,\ldots,0,-1))$.
	The locally closed and reduced substack of $\Sht_{r}$ where the associated local shtuka at $\infty$ has Newton point $b_\infty$ is called the \emph{Newton stratum} associated to $b_\infty$ and is denoted by $\Sht_{r,b_\infty}$.
	In a similar fashion, for a closed point $y$ of $X$ and a pair $$\underline{b}_{\underline{y}} =(b_{\infty}, b_{y}) \in B(\GL_{r, K_{\infty}}, (0,\ldots,0,-1)) \times B(\GL_{r, K_{y}}, (1, 0, \ldots, 0))$$
	the locally closed and reduced substack of $\Sht_{r, \F_{y}}$ 
	where the associated local shtuka at $\infty$ has Newton point $b_{\infty}$ and the associated local shtuka at $y$ has Newton point $b_{y}$ is denoted by $\Sht_{r,\underline{b}_{\underline{y}}}$.
\end{defn}

	\subsection{Isogenies of Drinfeld shtukas.}
We study isogenies of Drinfeld shtukas in more detail.
We consider the following moduli problem of Drinfeld shtukas with chains of isogenies.
\begin{defn}
	Let $m, r_1, \ldots, r_m \geq 1$ be positive integers such that $\sum_{j=1}^m r_j \leq r$.  
	A \emph{chain of $\pf^n$-isogenies of type $(r_1, \ldots, r_m)$} on a Drinfeld shtuka $\underline\Ec$ over a scheme $S$ is a flag of quotients of $\pf^n$-torsion shtukas
	$$ \underline{\Ec}|_{D_{n,S}} = \underline{\Fc}_{m+1}  \twoheadrightarrow \underline\Fc_m \twoheadrightarrow \ldots \twoheadrightarrow \underline\Fc_1 \twoheadrightarrow 0 $$
	over $S$ such that $\Fc_i$ has rank $n \cdot (r_1 + \ldots + r_i)$ as $\OS$-module.
	We denote the stack of Drinfeld shtukas with chains of $\pf^n$-isogenies of type $(r_1, \ldots, r_m)$ by $\Sht_{r,(r_1, \ldots, r_m)-\pf^{n}\text{-chain}}$.
\end{defn} 

We show below that a chain of $\pf^n$-isogenies of type $(r_1, \ldots, r_m)$ in the sense of the definition is the same as giving a chain of actual $\pf^n$-isogenies of Drinfeld shtukas
$$ \underline\Ec(\pf^n) = \underline\Ec_{m+1} \xrightarrow{f_{m+1}} \underline\Ec_m \xrightarrow{f_{m}}  \ldots \xrightarrow{f_2} \underline\Ec_1 \xrightarrow{f_{1}} \underline\Ec_{0} = \underline\Ec$$
such that 
the composition $f_{1} \circ \ldots \circ f_{m+1}$ is the inclusion $\underline\Ec(\pf^n) \to \underline\Ec$.

\begin{prop}
	\label{propKerIsog}
	Let $\underline{\Ec} \in \Sht_r(S)$ be a Drinfeld shtuka. 
	Every quotient $\underline{\Ec}|_{D_{n,S}} \twoheadrightarrow \underline{\Fc}$, where $\underline{\Fc}$ is a $\pf^n$-torsion shtuka, is the cokernel of a $\pf^n$-isogeny.
	
	Moreover, for two $\pf^n$-isogenies $f_1 \colon \underline{\Ec}_1 \hookrightarrow \underline{\Ec}$ and $f_2 \colon \underline{\Ec}_2 \hookrightarrow \underline{\Ec}$ such that the cokernels factor as  
	successive quotients $\underline{\Ec}|_{D_{n,S}} \twoheadrightarrow \coker(f_1) \twoheadrightarrow \coker(f_2)$, there exists a unique $\pf^n$-isogeny $f \colon \underline{\Ec_1} \hookrightarrow \underline{\Ec_2}$ such that $f_1 = f_2 \circ f$ and $\coker(f) \cong \ker(\coker(f_1) \twoheadrightarrow \coker(f_2))$. 
\end{prop}
\begin{proof}
	Let $\underline{\Fc}$ be a $\pf^n$-torsion shtuka as in the statement of the proposition. 
	Let us denote by $\Ec' = \ker(\Ec \twoheadrightarrow \Fc)$. As a first step, we want to show that $\Ec'$ is finite locally free of rank $r$ on $X_S$. In order to do so, we may by reduction to the universal case assume that $S = \Spec R$ is affine and Noetherian. As $\Ec' \hookrightarrow \Ec$ is an isomorphism away from $\pf$, it then suffices by fpqc-descent to show that the completion at $0$ is finite locally free of rank $r$. 
	As we assumed $R$ to be Noetherian, completion at $0$ is exact. The completion $\widehat {\Ec'_{0}} = \Ec' \otimes \Oc_0$ of $\Ec'$ at $0$ is hence given by the kernel of $\widehat{\Ec_0} \twoheadrightarrow \Fc$. The assertion now follows from \cite[Lemma 2.2.8]{Genestier1996}.
	
	By the right exactness of the tensor product, the cokernel of the induced map $\sigma^* \Ec' \to \sigma^* \Ec$ is given by $\sigma^* \Fc$. By \cite[Lemma 2.2]{Hartl2017a}, the map $\sigma^* \Ec' \to \sigma^* \Ec$ is thus injective, and $\sigma^* \Ec' = \ker(\sigma^* \Ec \twoheadrightarrow \sigma^* \Fc)$. In particular, we obtain an induced map $\tilde \ph \colon \sigma^* \Ec' \dashrightarrow \Ec'$ defined away from $\Gamma_x$ and $\Gamma_\infty$. As the map $\Ec' \to \Ec$ is an isomorphism away from $0$, locally around $\infty$ we obtain a map $\Ec' \to \sigma^* \Ec' $ with cokernel supported at $\infty$ and of rank 1 as $\OS$-module.
	It follows also that $\ph'|_{X'_S} \colon \sigma^* \Ec'|_{X'_S} \to \Ec'|_{X'_S}$ is a well-defined and injective map (as $\ph$ is). Note that $( \Ec'|_{{X'_S}}, \ph'|_{{X'_S}})$ is the associated $A$-motive in the sense of \cite{Hartl2017a} where $A = \Gamma(X \setminus \{\infty\}, \OX)$. By \cite[Proposition 2.3]{Hartl2017a}, the $A$-motive $( \Ec'|_{X'_S}, \ph'|_{X'_S})$ is effective, this means that $\coker(\ph'|_{X'_S})$ is annihilated by $\Jc^n$ for some positive integer $n$, where $\Jc$ is the quasi-coherent sheaf of ideals defining $\Gamma_x \se X_S$. Using \cite[Proposition 5.8]{Hartl2017a} we obtain that $\coker(\ph'|_{X'_S})$ has rank 1 as $\OS$-module.  
	Thus, $\coker(\ph'|_{X'_S})$ is already annihilated by $\Jc$, which means that $(\Ec', \ph')$ defines a point of $\Sht_r(S)$. 
	
	For the second part let $f_1 \colon \underline{\Ec}_1 \hookrightarrow \underline{\Ec}$ and $f_2 \colon \underline{\Ec}_2 \hookrightarrow \underline{\Ec}$ be two $\pf^n$-isogenies as in the assertion.
	It follows essentially by assumption that there is a unique injective homomorphism of shtukas $f \colon \underline{\Ec}_2 \to \underline{\Ec}_1$ such that $f_2 = f_1 \circ f$. It remains to check that $f$ is a $\pf^n$-isogeny. We have the short exact sequence of $R$-modules
	$$ 0 \to \Ec_1/f(\Ec_2) = \coker(f) \hookrightarrow \Ec/f_2(\Ec_2) = \coker(f_2) \twoheadrightarrow \Ec/f_1(\Ec_1) = \coker(f_1) \to 0,$$
	where the first map is $f_1$ and the second map is well-defined by assumption. As both $f_1$ and $f_2$ are isogenies, their cokernels are finite locally free $R$-modules. It follows that $\coker(f)$ is finite locally free as well, and thus $f$ is an isogeny. That it is a $\pf^n$-isogeny is also clear.
\end{proof}

\begin{remark}
	\begin{enumerate}
		\item Note that we only use that $0$ is $\Fq$-rational in order to apply \cite[Lemma 2.2.8]{Genestier1996}. It was pointed out to us by the anonymous referee that we can remove this hypothesis here as follows.
		Let $\F$ be the residue field of $0$. Then the argument of \cite[Lemma 2.2.8]{Genestier1996} goes through if we assume that $\Fc$ is finite locally free as an $R \otimes_\Fq \F$-module (and not only as an $R$-module). 
		
		Moreover, this condition is automatic if we assume that the leg $x$ of the shtuka $\Ec$ factors through $\Spf(\Oc_0)$. Namely, the completion of $X_R$ along $0$ in this case is given by the formal spectrum of 
		$$R \hat \otimes_\Fq \Oc_0 \cong (R \otimes_\Fq \F)\dsq{\varpi} \cong \prod_{1 \leq i \leq [\F \colon \Fq]} R \dsq{\varpi},$$
		compare \cite[Section 8]{Bornhofen2011} or \cite[Remark 5.2]{Rad2015}. In this case, we get $\Fc \cong \bigoplus_i \Fc_i$, where $\Fc_i = \Fc/\af_i \Fc$ and $\af_i \se R \hat \otimes_\Fq \Oc_0$ is the ideal cutting out the $i$-th factor. Then if $\Fc$ is finite locally free as an $R$-module, so are all the $\Fc_i$. 
		In particular, $\Fc$ is finite locally free as $R \otimes_\Fq \F$-module. 
		
		\item Using the comparison \cite[Theorem 5.8]{Hartl2017a} with isogenies of Drinfeld modules, we get as immediate corollaries that any finite locally free closed submodule scheme with strict $\Fq$-action of the $\pf^n$-torsion of a Drinfeld module is the kernel of an isogeny and a factorisation property as in the second part of the proposition. 
		Both of these facts seem to be only proven in the literature when the base is a field in \cite[2, Lemma 3.1 and Lemma 3.2]{Lehmkuhl2009}.
		\item This also shows that giving a point of $\Sht_{r, (r_1, \ldots, r_m)-\pf^{n}\text{-chain}}$ is the same as giving a chain of actual $\pf^n$-isogenies of Drinfeld shtukas
		$$ \underline\Ec(\pf^n) = \underline\Ec_{m+1} \xrightarrow{f_{m+1}} \underline\Ec_m \xrightarrow{f_{m}}  \ldots \xrightarrow{f_2} \underline\Ec_1 \xrightarrow{f_{1}} \underline\Ec_{0} = \underline\Ec$$
		such that $\coker(f_i)$ has rank $n \cdot r_i$ and such that
		the composition $f_{m+1} \circ \ldots \circ f_1$ is the inclusion $\underline\Ec(\pf^n) \to \underline\Ec$. 
	\end{enumerate}
\end{remark}

	\subsection{Naive $\Gamma_0(\pf^n)$-level structures and shtukas for Bruhat-Tits group schemes.}
\label{secBT}

We introduce naive $\Gamma_0(\pf^n)$-level structures on Drinfeld shtukas and explain how to interpret them as shtukas for certain Bruhat-Tits group schemes. 
These naive level structures seem inadequate in the non-parahoric case (that means when $n > 1$), as their moduli spaces are missing points in the fibre over $0$. In other words, the level map to $\Sht_r$ is not proper, compare Remark \ref{remNaiveLvlBad} below. 
The interpretation of naive level structures in terms of Bruhat-Tits group schemes allows us to give a candidate for a compactification of the level map: We can take the closure of the stack of shtukas with naive level in the product of the stacks of Drinfeld shtukas with corresponding parahoric level.

\begin{defn}
	\label{defnG0Class}
	A \emph{naive $\Gamma_0(\pf^n)$-level structure} on a Drinfeld shtuka $\underline{\Ec} = (\Ec, \ph) \in \Sht_{r}(S)$ of rank $r$ is a flag of quotients of $\pf^n$-torsion finite shtukas
	$$ \underline\Ec|_{D_{n,S}} = \underline\Lc_r \twoheadrightarrow \underline\Lc_{r-1} \twoheadrightarrow \ldots \twoheadrightarrow  \underline\Lc_{1} \twoheadrightarrow \underline\Lc_0 =  0$$ 
	such that $\Lc_i$ is finite locally free of rank $i$ as $\ODnS$-module (and hence of rank $in$ as $\OS$-module). 
\end{defn}

\begin{remark}
	By Proposition \ref{propKerIsog}, a naive $\Gamma_0(\pf^n)$-level structure is equivalently given as a chain of isogenies
	$$ \underline{\Ec}(\pf^n) = \underline\Ec_r \stackrel{f_r}{\rightarrow} \underline\Ec_{r-1} \stackrel{f_{r-1}}{\rightarrow} \underline\Ec_{r-2} \rightarrow \ldots \stackrel{f_1}{\rightarrow} \underline\Ec_0 = \underline{\Ec}$$
	such that $\coker(f_i)$ is finite locally free of rank 1 as $\ODnS$-module for all $1 \leq i \leq r$.
\end{remark}

We interpret naive $\Gamma_0(\pf^n)$-level structures on Drinfeld shtukas as shtukas for certain Bruhat-Tits group schemes in the following sense.

\begin{defn}
	\label{defBTS}
	A \emph{Bruhat-Tits group scheme} on $X$ is a smooth affine group scheme $G \to X$ such that
	\begin{enumerate}
		\item all fibres of $G$ are connected,
		\item the generic fibre of $G$ is a reductive group over $K$ and
		\item  \label{defBTScomp}  for every closed point $x$ of $X$ the base change $G_{\Oc_x} = G \times_X \Spec(\Oc_x)$ is a Bruhat-Tits group scheme in the sense that there is a non-empty bounded subset $\Omega$ in some apartment in the Bruhat-Tits building of $G_{K_{x}}$ such that $G(\Oc_x) \se G(K_x)$ is the connected fixator of $\Omega$ in the sense of \cite[(4.6.26)]{Bruhat1984}.
	\end{enumerate}
	A Bruhat-Tits group scheme is \emph{parahoric}, if the subgroups $G(\Oc_x) \se G(K_x)$ in (\ref{defBTScomp}) are parahoric for all places of $X$.
\end{defn}

\begin{remark}
	\label{remG0BT}
	Of particular relevance to our situation is the case where $\Omega$ is the stabiliser of a regular $(r-1)$-simplex $\Omega$ in the standard apartment of the (reduced) Bruhat-Tits building of $\GL_{r, K_0}$ with side-length $n$.
	We denote by $\GL_{r, \Omega} \to X$ the corresponding Bruhat-Tits group scheme that is isomorphic to $\GL_r$ away from $0$ and such that $\GL_{r, \Omega}(\Oc_0) \se \GL_r(K_0)$ is the connected stabiliser of $\Omega$.
	
	We can more explicitly describe this subgroup by 
	$\GL_{r, \Omega}(\Oc_0) = \{ M \in \GL_r(\Oc_0) \colon M \mod \pf^n \in B(\Oc_0/\pf^n)\}$. 
	By \cite[Lemma 3.1 and Theorem 3.2]{Mayeux2020}, the group scheme $\GL_{r, \Omega}$ can thus also be interpreted as the N\'eron blowup of $\GL_r$ in its subgroup $B$ of upper triangular matrices along the divisor $D_n$ in the sense of \cite[Section 3.1]{Mayeux2020}. 
\end{remark}

By \cite[Theorem 4.8]{Mayeux2020}, giving a $\GL_{r, \Omega}$-torsor on $X$ is equivalent to giving a $\GL_r$-torsor $\Ec$ on $X$ together with a reduction of $\Ec$ to an $B$-torsor over the divisor $D_{n}$ of $X$. More explicitly, a point of $\Bun_{\GL_{r, \Omega}}(S)$ is given by a rank $r$ vector bundle $\Ec$ on $X_S$ together  a flag of quotients of $\Ec|_{D_{n,S}}$ as in the definition of naive $\Gamma_0(\pf^n)$-level structures. In this sense, a naive $\Gamma_0(\pf^n)$-level structure on a Drinfeld shtuka $\underline{\Ec}$ defines a $(B, D_n)$-level structure on $\underline{\Ec}$ in the sense of \cite[Section 4.2.2]{Mayeux2020}.
 
A $\GL_{r, \Omega}$-shtuka is called \emph{bounded by $\underline{\mu} = ((0, \ldots, 0,-1), (1,0,\ldots,0))$} if its underlying $\GL_{r}$-shtuka $(x, \Ec, \ph)$ is bounded by  $(0, \ldots, 0,-1), (1,0,\ldots,0)$, and if the flag of quotients given by the $(B, D_n)$-structure on the underlying vector bundle $\Ec$ is $\ph$-stable.
In other words, the $\GL_{r, \Omega}$-shtukas bounded by $\underline{\mu}$ in this sense are exactly the Drinfeld shtukas with naive $\Gamma_0(\pf^n)$-level structures in the sense above.
We denote this stack of bounded $\GL_{r, \Omega}$-shtukas (or equivalently the stack of Drinfeld shtukas with naive $\Gamma_0(\pf^n)$-level structures) by 
 $\Sht_{r, \Omega}$.

For a facet $\ff$ in the Bruhat-Tits building of $\GL_{r, K_0}$ we write $\ff \prec \Omega$ if $\ff$ is contained in the closure of $\Omega$.
In a similar fashion to the construction above, for such a facet $\ff \prec \Omega$ we write $\GL_{r, \ff}$ for the corresponding Bruhat-Tits group scheme and $\Sht_{r, \ff}$ for the stack of $\GL_{r, \ff}$-shtukas bounded by $\underline{\mu}$ in the sense above. 
By Bruhat-Tits theory, for any facet $\ff$ contained in $\Omega$ there is a map of group schemes $\GL_{r,\Omega} \to \GL_{r,\ff}$ that is the identity away from $0$. By \cite[Corollary 3.16]{Breutmann2019}, we get maps  $\Sht_{r, \Omega} \to \Sht_{r,\ff}$. 

In particular, in the case $n = 1$ the set $\Omega$ is just given by the base alcove (corresponding to the standard Iwahori subgroup of matrices that are upper triangular mod $\pf$). 
Hence, for an alcove $\ff$ its corresponding moduli space of shtukas $\Sht_{r,\ff}$ parametrises chains of isogenies of Drinfeld shtukas as in Definition \ref{defnG0Class}. 
For a facet $\ff'$ of the alcove $\ff$ the map $\Sht_{r,\ff} \to \Sht_{r, \ff'}$ is then given by projection to some subchain of isogenies, depending on the position of $\ff'$. 
In particular, when $\ff'$ is a vertex, $\Sht_{r,\ff'}$ parametrises single Drinfeld shtukas and when $\ff'$ is an edge,  $\Sht_{r, \ff'}$ parametrises pairs of Drinfeld shtukas with a certain isogeny between them. 

In order to describe the maps $\Sht_{r, \Omega} \to \Sht_{r, \ff}$ for facets $\ff \prec \Omega$ more explicitly, we label the vertices in $\Omega$ by tuples $\underline{m} = (m_1,\ldots,m_{r-1})$ such that $n \geq m_1 \geq \ldots \geq m_{r-1} \geq 0$, edges are between vertices $\underline{m}$ and $\underline{m}'$ if and only if $0 \leq m_i - m_i' \leq 1$ for all $i$ or $0 \leq m_i' - m_i \leq 1$ for all $i$. 
The base alcove corresponds to the simplex defined by the vertices $(0,\ldots,0), (1,0, \ldots,0), (1,1,0, \ldots,0), \ldots, (1,1, \ldots,1)$. The vertex $(0,\ldots,0)$ corresponds to the constant group scheme $\GL_r$.

Note that every alcove $\ff \prec \Omega$ has a unique base point $\underline{m}$ (such that $m_i \leq x_i$ for all $i$ and points $\underline{x} \in \ff$), and an orientation that we encode by an element $\tau \in \Sym_{r-1}$ of the symmetric group on $(r-1)$ elements.
The orientation $\tau$ is chosen such that the vertices of $\ff$ are given by $\underline{m} + \tau(\mathbf{1}_{r-1}^{(i)})$ for $0 \leq i \leq r-1$, where $\mathbf{1}_{r-1}^{(i)} = (1, \ldots, 1, 0, \ldots, 0) \in \Z^{r-1}$ has exactly $i$ many entries equal to 1.
For a given pair $(\underline{m}, \tau)$ there clearly exists a unique alcove $\ff_{\underline{m}, \tau}$ in the standard apartment of the Bruhat-Tits building with base point $\underline{m}$ and orientation $\tau$. 

Starting from a $ \GL_{r, \Omega}$-shtuka $(\underline{\Ec}, (\underline{\Lc}_i)) \in \Sht_{r, \Omega}(S)$, we construct a Drinfeld shtuka $\underline{\Ec}_{\underline{m}}$ for a vertex $\underline{m} \prec \Omega$ as follows.
Assume that $S = \Spec(R)$ is affine and that all $\Lc_i$ are finite free as $R[\varpi]/(\varpi^n)$-modules. In this case, we can choose a basis $(e_1, \ldots, e_{r-1})$ of $\Lc_{r-1} = \left( R[\varpi]/(\varpi^n) \right)^{r-1}$ such that $(e_1, \ldots, e_i)$ is a basis for $\Lc_i$ for all $1 \leq i \leq r-1$. 
We consider the quotient 
$$\Lc_{r-1} \twoheadrightarrow \Lc_{\underline{m}} := R[\varpi]/(\varpi^{m_1}) e_1 \oplus \ldots \oplus R[\varpi]/(\varpi^{m_{r-1}}) e_{r-1}.$$
As all the $\Lc_i$ are $\ph$-stable quotients of $\Lc_{r-1}$, the matrix representation of $\ph$ with respect to $(e_1, \ldots, e_{r-1})$ is upper-triangular.
This shows that also $\Lc_{\underline{m}}$ is $\ph$-stable as $m_1 \geq \ldots \geq m_{r-1}$ by assumption.
By a similar argument, $\Lc_{\underline{m}}$ does not depend on the choice of basis (any base change matrix is again upper-triangular). Thus, we can glue to obtain a $\ph$-stable quotient $\Lc_{\underline{m}}$ also in the general case. We then associate to the vertex $\underline{m}$ the Drinfeld shtuka corresponding to the kernel $ \underline{\Ec}_{\underline{m}} = \ker(\underline{\Ec} \twoheadrightarrow \underline\Lc_{\underline{m}})$ by Proposition \ref{propKerIsog}. Moreover, by the second part of Proposition \ref{propKerIsog}, there are also canonical isogenies associated to the edges in the Bruhat-Tits building.

Using this construction, for an alcove $\ff_{\underline{m}, \tau} \prec \Omega$ the level map $\Sht_{r, \Omega} \to \Sht_{r, \ff_{\underline{m}}}$ associates to $(\underline{\Ec}, (\underline{\Lc}_i))$ the chain of isogenies 
$$ \underline{\Ec}_{\underline{m}}(\pf^n) \hookrightarrow  \underline{\Ec}_{\underline{m}+\tau(\mathbf{1}^{(r-1)}_{r-1})} \hookrightarrow \ldots \hookrightarrow \underline{\Ec}_{\underline{m} + \tau(\mathbf{1}^{(1)}_{r-1})} \hookrightarrow \underline{\Ec}_{\underline{m}}.$$
This means that the induced map $\Sht_{r, \Omega} \to \varprojlim_{\ff \prec \Omega} \Sht_{r, \ff}$ associates to a Drinfeld shtuka with naive $\Gamma_0(\pf^n)$-level structure a diagram $(\underline{\Ec}_{\underline{m}})_{\underline{m}}$ with the canonical isogenies as constructed above.

\begin{remark}
	\label{remNaiveLvlBad}
	For parahoric level (in our case that means $n \leq 1$) \cite[Theorem 3.20]{Breutmann2019} shows that the level maps are proper and surjective. 
	An explicit calculation for deeper level (that is for $n > 1$) shows that this is false already in the $\GL_2$-case over $X = \P^1$ in general.
	Namely, we study the fibre over $0$ using the local model of \cite[Theorem 4.4.6]{Rad2017}. In our case, the local model is given as the closure of the Schubert variety $\Gr_{\GL_n}^{\leq (1, 0)}$ over $\Spec(K_0)$ inside the affine Grassmannian $\Gr_{\GL_{2, \Omega}}$ over $\Spec(\Oc_0)$ for the corresponding Bruhat-Tits group scheme. 
	For $n = 1$ we get the familiar local picture of two copies of $\P^1$ intersecting transversally at supersingular points.
	
	However, for any $n > 1$ the local model fails to be proper by \cite[Theorem A]{Richarz2016}. More precisely, we can calculate the schematic closure explicitly to see that the special fibre of the corresponding local model only consists of two copies of $\A^1$ that do not intersect. 
	This means in particular that $\Sht_{r, \Omega}$ is missing the supersingular points in the special fibre.
	Moreover, from the comparison with the modular curve, we might expect to find $n+1$-components two of which are reduced by \cite[Theorem 13.4.7]{Katz1985}.
	The two components we see using the naive level structure correspond to the two reduced components, but we do not get the non-reduced ones.
	
	The goal of this paper is to explain one way to remedy this.
	We show that we can explicitly describe the schematic image of the  map $\Sht_{r, \Omega} \to \varprojlim_{\ff \prec \Omega} \Sht_{r, \ff}$
	in terms of \emph{Drinfeld level structures} and that this provides a natural compactification of the level map.
	It turns out (compare Section \ref{secComp} below) that requiring the quotients in the definition of naive level structures to be locally free as $\ODnS$-modules is too restrictive and we rather should allow in the special fibre also degenerations to certain $\pf^n$-torsion finite shtukas which are not locally free as $\ODnS$-modules.
	In particular, we will see that supersingular points admit a Drinfeld $\Gamma_0(\pf^n)$-level structure that does not come from a naive $\Gamma_0(\pf^n)$-level structure.
\end{remark}

%
%

	\section{Group schemes attached to Drinfeld shtukas}

In order to define Drinfeld level structures for Drinfeld shtukas, we explain how to construct a (finite locally free $\Oc_0/\pf^n$-module) scheme of $\pf^n$-torsion points $\underline{\Ec}$ of a Drinfeld shtuka. 
This scheme of $\pf^n$-torsion points serves as an analogue of the scheme of $p^n$-torsion points of an elliptic curve and behaves similarly in many ways.
In order to study properties of $\underline{\Ec}[\pf^n]$ we use an explicit comparison of Drinfeld shtukas and Drinfeld modules.

\subsection{Comparison with Drinfeld modules.}

We recall some facts about Drinfeld modules and show how to construct Drinfeld shtukas from them.
Let $A = \Gamma(X \setminus \{\infty\}, \OX)$. The point $0$ of $X$ then corresponds to a maximal ideal of $A$, which by a slight abuse of notation we denote by $\pf$ as well. 

Roughly speaking, a Drinfeld $A$-module is an $A$-module structure on a (geometric) line bundle. 
Drinfeld modules were first introduced in \cite{Drinfeld1976} in order to construct a Langlands correspondence in the cohomology of their moduli spaces. In this sense, Drinfeld modules (of rank 2) are function field analogues of elliptic curves in the number field case. For a more detailed treatment also compare \cite{Lehmkuhl2009}, \cite{Blum1997} or \cite{Laumon1996}.

Let $\Lc$ be an invertible sheaf on $S$. The corresponding geometric line bundle is denoted by $\GaL = \underline{\Spec}_S(\Sym(\Lc^{-1}))$. 
If $S = \Spec(R)$ is an affine scheme such that $\Lc$ is trivial, the corresponding line bundle is given by $\GaR = \Spec R[t]$. 
Locally, the ring of endomorphisms of a line bundle is then given by the skew-polynomial ring $R\{\tau\}$ with the commutation relation $\tau c = c^q \tau$ for $c \in R$.

\begin{defn}
	\label{defnDModule}
	A \emph{Drinfeld $A$-module} $\Eb = (\GaL, e)$ of rank $r$ over a scheme $S$ consists of
	an additive group scheme $\GaL$ and a ring homomorphism 
	$e \colon A \to \End(\GaL), a \mapsto e_a$
	such that $e_a$ is finite for all $a \neq 0 \in A$ of degree $-r\mr{deg}(\infty) v(a)$, where $v$ is the normalised valuation on $K$ corresponding to $\infty$.  
	The composition $\partial \circ e$ with the differential induces a map $S \to \Spec(A)$ called the \emph{characteristic} of $\Eb$. 
\end{defn}

We denote by $\Dmod_{r}$ the moduli stack of Drinfeld modules of rank $r$. 
It is a Deligne-Mumford stack of finite type over $\Fq$, which is smooth of relative dimension $r-1$ over $X' = \Spec(A)$. 

When $S = \Spec \ell$ is the spectrum of a field (or more generally when the line bundle $\Lc$ is trivial), a Drinfeld module as a ring homomorphism $e \colon A \to \ell \{ \tau \}$. As for Drinfeld shtukas, in a slight abuse of notation, we say $\Eb$ has characteristic $\pf$ if the the characteristic of $\Eb$ factors through $0$, or in other words, if the kernel of the induced map $A \to \OS(S)$ is $\pf$. 
We say that a Drinfeld module over a field $\ell$ in characteristic $\pf$ has height $h$, if the smallest non-vanishing coefficient in $e_{\varpi} \in \ell \{ \tau \}$ has degree $h$, where $\varpi \in \pf$ is a uniformiser.  

There are several ways to associate vector bundles to Drinfeld modules, for example the so-called \emph{elliptic sheaves} due to \cite{Drinfeld1977a}, for a more detailed treatment also compare \cite{Blum1997}, \cite{Hartl2005} or \cite{Wiedmann04}, or $t$-motives \cite{Anderson1986} and their generalisations, see for example \cite{Hartl2017a}. However, a precise comparison to Drinfeld shtukas, which is certainly well-known to the experts, does not seem to be part of the literature yet. 
We explain how to construct Drinfeld shtukas from Drinfeld modules.

Recall that an \emph{elliptic sheaf} $\underline{\Ec}$ over $S$ of rank $r$ is given by the data $(x ,( \Ec_i)_{i \in \Z}, (s_i)_{i \in \Z}, (t_i)_{i \in \Z})$ where $x \colon S \to X' = X \setminus \{ \infty\}$ is a map of schemes,  $\Ec_i$ is a rank $r$ vector bundle on $X \times S$ for every $i \in \Z$ and $s_i \colon \Ec_i \to \Ec_{i+1}$ and $t_i \colon \sigma^* \Ec_i \to \Ec_{i+1}$ are injective maps that satisfy some further properties. In particular, $\coker(s_i)$ and $\coker(t_i)$ are supported on $\infty$ and $\Gamma_{x}$, respectively and invertible as $\OS$-modules on their support. We denote by $\Ell_r$ the moduli stack of elliptic sheaves. We have a well-defined map
$\Ell_r \to \Sht_r$ given by the projection
$$(x, (\Ec_i), (s_i), (t_i))  \mapsto (x, \Ec_0, s^{-1}_{0}|_{X \setminus ( \Gamma_x \cup \Gamma_\infty)} \circ t_{0}|_{X \setminus ( \Gamma_x \cup \Gamma_\infty)}),$$
or by Remark \ref{remdefDSht} equivalently by projection to $(x, \Ec_0, \Ec_1, s_0, t_0)$.
We use this second perspective for the remainder of this section as it more convenient in this context.
We define a functor $\Z \times \Dmod_r \to \Sht_r$ by composing the equivalence $\Z \times \Dmod_r \to \Ell_r$ of \cite{Drinfeld1977a} with this projection.

\begin{lem}
	\label{lemEllToShtFFs}
	The projection $\Ell_r \to \Sht_r$ is fully faithful.
\end{lem}
\begin{proof}
	Let $\underline{\Ec}_\bullet = (x, (\Ec_i), (s_i), (t_i))$ and $\underline{\tilde\Ec}_\bullet = (x, (\tilde\Ec_i), (\tilde s_i), (\tilde t_i))$ be two elliptic sheaves over $S$. Assume that we have a map of the corresponding shtukas, in other words a pair of maps $f_{0} \colon \Ec_{0} \to \tilde\Ec_{0}$ and $f_{1} \colon \Ec_{1} \to \tilde\Ec_{1}$ that commute with $s_{0}$ and $t_{0}$ in the obvious way. 
	By \cite[Corollary 5.4]{Wiedmann04} we may then glue $f_1$ and $\sigma^* f_1$ to get a map $f_2 \colon \Ec_2 \to \tilde{\Ec}_2$ again commuting with  $s_1$ and $t_1$. Such a map is necessarily unique. 
	We continue inductively to define maps in higher degrees. The maps in degrees smaller than $0$ can be constructed as twists.
\end{proof}

Let us denote by $b_\infty = (-1/r, \ldots, -1/r) \in B(\GL_{r, K_\infty}, (0,\ldots,0,-1))$ the basic Newton polygon. Recall that we defined $\Sht_{r, b_\infty}$ to be the (reduced) locus in $\Sht_r$ where the local shtuka at $\infty$ has Newton polygon $b_\infty$. Note that  $\Sht_{r,b_\infty}$ is a closed substack of $\Sht_r$ as $b_\infty$ is basic.

\begin{prop} 
	\label{propFunDModSht}
	The functor $\Z \times \Dmod_{r} \to \Sht_r$ is schematic and a closed immersion which factors through an isomorphism 
	$$ \Z \times \Dmod_{r} \xrightarrow{\cong} \Sht_{r,b_\infty}.$$
\end{prop}

\begin{proof}
	
	As a first step we show that the locus where a Drinfeld shtuka can be extended to an elliptic sheaf is closed.
	Let $\underline{\Ec} = (x, \Ec_{-1}, \Ec_{0}, s_{-1}, t_{-1}) \in \Sht_r(S)$. 
	As the zero and pole of $\underline{\Ec}$ do not intersect,  we can repeatedly glue $\Ec_i$ and $\sigma^*\Ec_i$ to obtain a commutative diagram
	\begin{center}
		\begin{tikzcd}
			\Ec_{-1} \arrow[r, "\ph = s_{-1}"]
			& \Ec_{0} \arrow[r, "s_{0}"] 
			& \Ec_1 \arrow[r, "s_{1}"] 
			& \ldots \\
			\sigma^*\Ec_{-1} \arrow[r] \arrow[ur, "\ph = t_{-1}"] 
			& \sigma^*\Ec_{0} \arrow[r] \arrow[ur, "t_{0}"] &  \sigma^*\Ec_{1} \arrow[r] \arrow[ur, "t_{1}"] &\ldots
		\end{tikzcd}
	\end{center} 
	If the diagram comes from an elliptic sheaf, we have by definition that $\Ec_{0} \hookrightarrow \Ec_{r}$ identifies $\Ec_{0}$ with $\Ec_{r}(-\infty)$. In a similar fashion we get $\Ec_1 = \Ec_{r+1}(-\infty)$.
	 
	We claim that these two conditions are already sufficient for the diagram to come from an elliptic sheaf. 
	By construction, the cokernel of $s_i$ is supported on $\Gamma_\infty$ and the cokernel of $t_i$ is supported on $\Gamma_x$, and both are invertible on their respective supports. 
	We first check that $s_{0} = s_{r} \otimes \id_{\Oc_X(-\infty)}$. As all $s_i$ are isomorphisms away from $\infty$ and the question is fpqc-local on $S$, it suffices to consider the completion at $\infty$ and we may assume that $S = \Spec(R)$ is affine and all $\Ec_i$ are free $R\dsq{\varpi_\infty}$-modules of rank $r$. Thus, the $s_i$ are identified with endomorphisms of $R\dsq{\varpi_\infty}^r$ such that both $s_{r-1} \circ \ldots \circ s_0$ and $s_{r} \circ \ldots \circ s_{1}$ are multiplication by $\varpi_\infty$ by assumption, where $\varpi_\infty$ is a uniformiser at $\infty$.
	But as multiplication by $\varpi_\infty$ is injective and lies in the centre of the endomorphism ring, this implies that $s_{0} = s_{r}$ as desired.
	Moreover, the $s_i$ induce isomorphism $\coker(t_{i-1}) \xrightarrow{\cong} \coker(t_i)$ for all $i \geq 1$, hence $t_{r} = t_{-1} \otimes \id_{\OX(\infty)}$. Hence, we get inductively that  $\Ec_{i+r} = \Ec_i \otimes \OX(\infty)$ for all $i \geq 1$. The data for indices $i \leq 0$ is then obtained by twisting. This shows the claim.
	
	It remains to check that the conditions of the claim are closed conditions. 
	In order to see that the locus where $\Ec_0 = \Ec_r (-\infty)$ is closed, we argue as follows.
	As $\Ec_r/\Ec_0$ is supported on $\Gamma_\infty$, the uniformiser $\varpi_\infty$ at $\infty$ acts on  $\Ec_r/\Ec_0$ and we have $\Ec_0 = \Ec_r (-\infty)$  if and only if $\varpi_\infty = 0$ in $\End_{\OS}(\Ec_r/\Ec_0)$.
	Hence, the locus where $\Ec_0 = \Ec_r (-\infty)$ is represented by the vanishing locus $V(\Ic)$ of the quasi-coherent ideal 
	$ \Ic = \mr{image}(\EEnd_{\OS}(\Ec_r/\Ec_0)^\vee \xrightarrow{\varpi_\infty^\vee} \OS)$.
	In a similar fashion, the locus where  $\Ec_{1} = \Ec_{r+1} (-\infty)$ is representable by a closed subscheme of $S$ given by the vanishing locus of a quasi-coherent sheaf of ideals $\Ic'$ in $\OS$.
	Thus, the locus where $\underline{\Ec}$ defines a (necessarily unique) elliptic sheaf is representable by the closed subscheme $S' = V(\Ic + \Ic')$ of $S$.
	In particular, $\Z \times \Dmod_{r} \to \Sht_r$ is schematic and a closed immersion.

	Moreover, it is clear that both stacks have the same geometric points, as one can easily see by the classification of bounded local shtukas over algebraically closed fields that a Drinfeld shtuka over an algebraically closed field $\ell$ comes from a Drinfeld module if and only if the local shtuka at $\infty$ is $$\left( \ell\dsq{\varpi_\infty}^r, \sigma \cdot \begin{pmatrix}
	0 & &  & \varpi_\infty^{-1} \\
	1 & & & \\
	& \ddots & & \\
	& & 1 & 0
	\end{pmatrix} \right) $$
	The second part of the assertion follows as $\Dmod_{r}$ and the Newton stratum $\Sht_{r,b_\infty}$ both are reduced. 
\end{proof}

\begin{remark}
	\label{remComparisonGeneral}
	\begin{enumerate}
		\item For the general case (when $\infty$ is not defined over $\Fq$) giving a correct definition of the truncation is more subtle, as the pole might not be supported at $\infty$ but at some Frobenius twist of $\infty$. In order to remedy this, one should include a Frobenius twist in the action of $\Z$ by shifts and still obtain a well-defined closed immersion
		$\Z \times \Dmod_{r} \to \Sht_r$.
		\item 
		\label{remProjDmod}
		It follows that we get an essentially surjective functor $\Sht_{r,b_\infty} \to \Dmod_{r}$ that agrees with the construction from \cite[Proposition 7.8]{Breutmann2019} (up to forgetting the level structure).
		
		\item 
		\label{remLocShtDMod}
		By \cite{Hartl2017a} and \cite{Hartl2019}, the comparison is compatible with local and finite objects.
		More precisely, the local equivalence of \cite[Theorem 8.3]{Hartl2019} and \cite[Theorem 7.6]{Hartl2017a} identifies the $\pf$-divisible module associated to a Drinfeld module $\Eb$ over $S \in \Nilp_{\F \dsq{\varpi}}$ the local shtuka at $0$ of any Drinfeld shtuka associated to $\Eb$ by the comparison.
		We call this local shtuka the \emph{local shtuka at $0$} (or the \emph{local shtuka at $\pf$}) of the Drinfeld module $\Eb$.
		By \cite[Proposition 1.7]{Drinfeld1976}, the Newton polygon of the local shtuka at $\pf$ associated to a Drinfeld module of height $h$ over an algebraically closed field in characteristic $\pf$ is given by $(1/h, \ldots, 1/h, 0, \ldots,0)$.
	\end{enumerate}
\end{remark}

\begin{remark}
	Let $\underline{\Ec}$ be an Drinfeld shtuka over an algebraically closed field $\ell$ such that its characteristic factors through $0$. As $\ell$ is algebraically closed, the local shtuka of $\underline{\Ec}$ splits canonically as the product of an \'etale local shtuka and a topologically nilpotent one by \cite[Proposition 2.9]{Hartl2019} (or by \cite[Corollary 10.17]{Hartl2019} equivalently its corresponding $\pf$-divisible module splits canonically into an \'etale and a one-dimensional formal part).
	By \cite[Corollary 2.9]{Rad2015}, the \'etale local shtuka is isomorphic to $\left( \ell\dsq{\varpi_\infty}^r, \sigma \right)$, while the formal part is classified up to isomorphism by its height by \cite[Proposition 1.7]{Drinfeld1976}. 
	In particular, the local shtukas at $0$ of Drinfeld shtukas are classified up to isomorphism (and not only up to isogeny) by their Newton polygons. 
	\label{remClassNewtonIso}
\end{remark}

\subsection{Strong stratification property of the Newton stratification.}
We show the strong stratification property of the Newton stratification and use the comparison of Drinfeld modules and Drinfeld shtukas to show the non-emptiness of Newton strata. 
We deduce that the local shtuka of a Drinfeld shtuka is isomorphic to the local shtuka of an appropriate Drinfeld module. 

\begin{prop}
	\label{propSSShtEx}
	Let $\underline{b}_{\underline{0}} = ((-1/r, \ldots, -1/r), (1/r, \ldots, 1/r)) \in B(\GL_{r, K_\infty}) \times B(\GL_{r, K_0})$. Then the basic Newton stratum $\Sht_{r,\underline{b}_{\underline{0}}} \se \Sht_{r, \F_0}$ is non-empty.
\end{prop}
\begin{proof}
	By Remark \ref{remComparisonGeneral} and Proposition \ref{propFunDModSht} a basic Drinfeld module in characteristic $\pf$, that is, a Drinfeld module of both rank and height $r$, defines a point in $\Sht_{r,\underline{b}_{\underline{{0}}}}$. But basic Drinfeld modules in characteristic $\pf$ exist by \cite[Proposition 7.4.1]{Kondo2020}.
\end{proof}

As a next step, we study closure relations among the Newton strata. 
The result may be well-known to experts. The author was unable to track down a precise reference.
The corresponding statement for Shimura varieties in the Siegel case is due to  \cite{Oort2001} and has been generalised to the PEL case by \cite{Hamacher2015}.

For a pair of closed points $\underline{y} = (y_1, y_2)$ of $X$ we define a partial order on $B(\GL_{r, K_{y_1}}) \times B(\GL_{r, K_{y_2}})$ (and also on $B(\GL_{r, K_{y_1}}, \mu_1) \times B(\GL_{r, K_{y_2}}, \mu_2)$ for a pair of cocharacters $\mu_1, \mu_2$ of $\GL_r$) by $b_{\underline{y}} = (b_{y_1}, b_{y_2}) \leq b'_{\underline{y}} = (b'_{y_1}, b'_{y_2})$ if $b_{y_1} \leq b'_{y_1}$ and $b_{y_2} \leq b'_{y_2}$.
Let us fix the cocharacters $\mu_1 = (0, \ldots,0,-1)$ and $\mu_2 = (1, 0, \ldots, 0)$ of $\GL_r$.

\begin{thm}
	\label{thmNewtonStratStrongStrat}
	The Newton stratification on $\Sht_{r, \F_{0}}$ satisfies the strong stratification property. In other words, 
	for all
	$\underline b \in B(\GL_{r, \infty}, \mu_1) \times B(\GL_{r, 0},\mu_2)$ we have
	$$\overline{\Sht_{r,\underline b}} = \bigcup_{{\underline b'} \leq \underline b} \Sht_{r,\underline b'} = \Sht_{r,\leq {\underline b}}.$$
	
	Moreover, all the Newton strata
	$\Sht_{r,\underline{b}}$ for 
	$\underline b \in B(\GL_{r, \infty}, \mu_1) \times B(\GL_{r, 0}, \mu_2)$ are non-empty.
\end{thm}
Viehmann in \cite[Remark 5.6]{Viehmann2020} remarks that the assertion should follow by a similar argument as in \cite{Hamacher2015} for Shimura varieties of PEL type.

\begin{proof}
	Let $\underline{b}^0 $ correspond to the unique basic point 
	$$ ((1/r, \ldots, 1/r), (-1/r, \ldots, -1/r)) \in B(\GL_{r, K_{ \infty}}, \mu_1) \times B(\GL_{r, K_{0}}, \mu_2).$$
	By Proposition \ref{propSSShtEx} the Newton stratum $\Sht_{r,\underline{b}^0 }$ is non-empty.
	The non-emptiness of the other strata will follow from the closure relations.
	
	Now, let $\underline b  \in B(\GL_{r, K_{ \infty}}) \times B(\GL_{r, K_{0}})$ and assume that $\Sht_{r,\underline{b} }$ is non-empty. We fix a point $s \in \Sht_{r,\underline{b} }$ and let $R$ be its universal deformation ring. Then $s$ lies in the closure of some $\Sht_{r,\underline{b'} }$ for ${\underline{b} } \leq {\underline{b}' }$ if and only if the same is true in the Newton stratification on $\Spec R$.
	By the Serre-Tate Theorem (Proposition \ref{propSerreTateSht}) the universal deformation ring factors as $\Spec R = \Spec R_\infty \times \Spec R_0$, where $R_\ast$ is the universal deformation ring of the corresponding local shtuka at $\ast = \infty,0$. Under this isomorphism we have $\Spec(R)_{\underline{b} } = \Spec(R_1)_{{b}_{{ \infty}}} \times \Spec(R_2)_{{b}_{{0}}}$, where we denote by $\Spec(R_\ast)_{{b}_{\ast}}$ the corresponding Newton strata in $\Spec R_\ast$ for $\ast = \infty,0$.
	On $\Spec R_\ast$ the closure properties hold by \cite[Theorem 2, Lemma 21 (2)]{Viehmann2011}, and thus they hold on $\Spec R$. This proves the assertion.
\end{proof}

\begin{remark}
	\label{remStratDmod}
	In a similar fashion, there is a Newton stratification on the moduli space of Drinfeld modules in characteristic $\pf$ defined via the local shtukas as defined in Remark \ref{remComparisonGeneral} (\ref{remLocShtDMod}).
	The Newton stratifications are clearly compatible with the projection $\Sht_{r,b_\infty, \F_0} \to \Dmod_{r, \F_0}$ in the fibre over $0$ from Remark \ref{remComparisonGeneral} (\ref{remProjDmod}).
	Thus, the Newton stratification on Drinfeld modules retains the strong stratification property as in the theorem above.
	In particular, there exist Drinfeld modules of arbitrary height in characteristic $\pf$. 
\end{remark}

\begin{cor}
	\label{corCompLocSht}
	Let $\underline{\Ec}$ be a Drinfeld shtuka over a complete local noetherian ring $R$ with algebraically closed residue field $\ell$ such that the characteristic of $\underline{\Ec}_\ell$ factors through $0$. 
	Then there exists a Drinfeld module $\Eb$ over $R$ such that the local shtukas at $0$ of $\underline{\Ec}$ and $\Eb$ are isomorphic.
\end{cor}

\begin{proof}
	The case where $R = \ell$ is a field directly follows from the non-emptiness of Newton strata of Theorem \ref{thmNewtonStratStrongStrat} and Remark \ref{remStratDmod}.
	The case that $R$ is local artinian then follows from the Serre-Tate Theorem \ref{propSerreTateSht}, and the general case that $R$ is a complete noetherian local ring with algebraically closed residue field follows from the fact that $\Dmod_r$ is of finite type over $\Fq$ and \cite[Proposition 3.16]{Hartl2011}. 
\end{proof}

\subsection{The $\pf^n$-torsion scheme of a Drinfeld shtuka.}

We briefly explain how to construct a scheme of $\pf^n$-torsion points of a shtuka, which will play the role of the $p^n$-torsion points of an elliptic curve.
The construction goes back to \cite{Drinfeld1987a}. Let $\underline{\Fc}$ be a $\ph$-sheaf (for example a $\pf^n$-torsion shtuka) over $S$. We set 
$$ \Dr_{q}(\underline \Fc) = \underline{\Spec}\left(\Sym^\bullet \Fc \right)/\Ic, $$
where $\Ic$ is the ideal locally generated by the sections $v^{\otimes q} - \ph(\sigma^* v)$. 
It induces a contravariant functor from the category of $\ph$-sheaves to the category of finite locally free group schemes with $\Fq$-action over $S$. 
Assume $S = \Spec R$ is affine, $\Fc = R^r$ is trivial and $\ph$ is given by the matrix $(a_{ij})$. Then
$$ \Dr_{q}(\underline \Fc) = \Spec \left( R[Y_1, \ldots, Y_r]/\left(Y_1^{q} - \sum_{i=1}^r a_{i1} Y_i, \ldots, Y_r^{q} - \sum_{i=1}^r a_{ir} Y_i\right)\right).$$
\begin{prop}[{\cite[Proposition 2.1]{Drinfeld1987a}}, {\cite[Theorem 2]{Abrashkin2006}} and {\cite[Theorem 5.2]{Hartl2019}}]
	\label{propFinShtEq}
	Let $\underline \Fc = (\Fc, \ph)$ be a finite shtuka of rank $r$ on $S$. 
	Then the group scheme $\Dr_{q}(\underline \Fc)$ is finite locally free of rank $q^r$ over $S$, \'etale over $S$ if and only if $\ph$ is an isomorphism, and radicial over $S$ if and only if $\ph$ is locally nilpotent on $S$. 
	Moreover, the functor $\Dr_q$ is $\Fq$-linear and exact. Its essential image is characterised by the property that the $\Fq$-action is strict in the sense of \cite{Faltings2002}.
\end{prop} 
Note that the notion of a strict $\Fq$-action is a condition on the $\Fq$-action on the co-Lie complex of a certain deformation of the group scheme. We do not need the exact definition here and refer to \cite{Faltings2002} or \cite{Hartl2019} for more details. 
In our setting the strictness of the $\Fq$-action will usually be automatic. 

\begin{defn}
Let $\underline{\Ec}$ be a rank $r$ Drinfeld shtuka over $S$. 
We denote by 
$$ \underline{\Ec}[\pf^n] = \Dr_q(\underline{\Ec}|_{D_{n,S}})$$
the \emph{scheme of $\pf^n$-division points} of $\underline{\Ec}$.
\end{defn}
The previous proposition implies that $\underline{\Ec}[\pf^n]$ is a finite locally free $S$-group scheme of rank $q^{nr}$ with strict $\Fq$-action. The $\Oc_0/\pf^n$-module structure on $\underline{\Ec}|_{D_{n,S}}$ gives rise to a canonical $\Oc_0/\pf^n$-module structure on $\underline{\Ec}[\pf^n]$. 
The finite shtuka equivalence in particular induces an equivalence of quotients of $\underline{\Ec}|_{D_{n,S}}$ as $\pf^n$-torsion shtukas and finite locally free closed $\Oc_0/\pf^n$-module subschemes with strict $\Fq$-action of $\underline{\Ec}[\pf^n]$.

By comparison with Drinfeld modules, we get the following explicit description of the $\pf^n$-torsion in characteristic $\pf$.

\begin{cor}
	\label{corCompItorUniv}
		Let $\underline{\Ec}$ be a Drinfeld shtuka over a complete local noetherian ring $R$ with algebraically closed residue field $\ell$ such that the characteristic of $\underline{\Ec}_\ell$ factors through $0$.  
		Then there exists a Drinfeld module $\Eb$ over $R$ such that
	$$ \underline{\Ec}[\pf^n] \cong \Eb[\pf^n]$$
	as $\Oc_0/\pf^n$-module schemes over $R$ for all $n \in \N$.
\end{cor}

\begin{proof}
	This follows directly from the corresponding assertion for the local shtukas in Corollary \ref{corCompLocSht}.
\end{proof}

\begin{prop}
	\label{propDescPTor}
	The scheme of $\pf^n$-division points of a shtuka $\underline\Ec$ of rank $r$ over an algebraically closed $\ell$ is given by the $\Oc_0/\pf^n $-module scheme
	$$ \underline\Ec[\pf^n] = \alpha_{q^{nh}} \times (\pf^{-n}/\Oc_0)^{r-h},$$
	where the operation of $\varpi$ on $\alpha_{q^{nh}}$ is given by $t \mapsto t^{q^h}$, and where $h$ is the height of $\underline\Ec$ (we use the convention $h = 0$ when the characteristic of $\underline{\Ec}$ is away from $0$). 
\end{prop}
\begin{proof}
	 We first consider the case that $\pf$ is away from the characteristic of $\underline{\Ec}$. Then, $ \underline{\Ec}[\pf^{n}]$ is a finite \'etale scheme by the finite shtuka equivalence. It follows that \'etale locally on $S$ the $\Oc_0/\pf^n$-module scheme $\underline{\Ec}[\pf^{n}]$ is constant. Over geometric points, we have that $\underline{\Ec}[\pf^{n}] \cong (\pf^{-n}/\Oc_0)^{r}$ as the corresponding \'etale local shtuka is trivial by \cite[Corollary 2.9]{Rad2015}.
	In characteristic $\pf$, by the previous Corollary \ref{corCompItorUniv} it suffices to check the assertion for Drinfeld modules, which then follows from \cite[3, Proposition 1.5]{Lehmkuhl2009} and \cite[2, Corollary 2.4]{Lehmkuhl2009}.
\end{proof}

Even more generally, we can embed the scheme of $\pf$-torsion points as a closed subscheme of a smooth curve.  
However, this smooth curve will not be a Drinfeld module in general. 

\begin{prop}
	\label{propEmbItor}
	Let $\underline{\Ec}$ be a Drinfeld shtuka over 
	$S$. Then \'etale-locally on $S$, the scheme of $\pf^{n}$-division points of $\underline{\Ec}$ can be embedded as a closed subscheme of a smooth curve over $S$.
	More precisely, we can \'etale-locally on $S$ embed $\underline{\Ec}[\pf^{n}]$ as a closed subscheme of $\A^1_S$.
\end{prop}
\begin{remark}
	\label{remEmbPTorFrimu}
	For one-dimensional $p$-divisible groups a similar statement is discussed in \cite[Lemma 5.2.1]{Frimu2019}, building on arguments from \cite{Harris2001}. However, \cite{Frimu2019} claims that an embedding even exists Zariski-locally on $S$, this seems to be problematic to us for the following reason. 
	Let us assume that $S$ is the spectrum of a finite field. We assume that the \'etale part of $\underline\Ec[\pf^n]$ is non-trivial and constant over $S$. Then the number of rational points of $\underline\Ec[\pf^n]$ tends to infinity as $n \to \infty$.  However, the number of $S$-rational points on $\A^1_S$ is bounded. In particular, it cannot be possible to embed $\underline\Ec[\pf^n]$ into $\A^1_S$ for all $n \in \N$.
	
	For the proof of the proposition, we essentially adapt the proof of \cite[Lemma 5.2.1]{Frimu2019}, but we allow finite extensions on the residue fields in order to circumvent the issue discussed above.
	So we only get an \'etale local statement. 
\end{remark}
\begin{proof}
	We adapt the proof of \cite[Lemma 5.2.1]{Frimu2019}.
	We first consider the case that  $S = \Spec \ell$ is the spectrum of an algebraically closed field.
	In this case, the assertion follows from the explicit description of $\underline{\Ec}[\pf^n]$ in Proposition \ref{propDescPTor}.

	For the general case, we may by reduction to the universal case assume that $S$ is locally of finite type over $\Fq$. As the statement is local on $S$, we may further assume that $S= \Spec(R)$ is affine and of finite type over $\Fq$. Then $\underline\Ec[\pf^n] = \Spec(B)$ is affine as well. 
	We fix a closed point $s \in S$. By the argument above, there exists a finite extension $\F$ of the residue field $k(s)$ (which is finite by assumption) at $s$ such that there exists a closed immersion $\underline\Ec[\pf^n]_\F \hookrightarrow \A^1_\F$ over $\F$, in other words a surjection $\F[t] \twoheadrightarrow B \otimes \F$.
	By \citestacks{00UD} there exists an \'etale neighbourhood $\Spec R' \to \Spec R$ of $s$ and a point $s'$ over $s$ such that the extension of residue fields $k(s) \to k(s')$ is given by $k(s) \to \F$.
	We can thus lift the surjection to a map $R'[t] \to B \otimes R'$ by choosing a lift of the image of $t$.
	By Nakayama's lemma this is a surjection over some $R'_f$, where $f \in R'$ is not contained in the maximal ideal at $s'$.
	In other words, $\underline\Ec[\pf^n]_\F \hookrightarrow \A^1_\F$ extends to $\underline\Ec[\pf^n]_{R'_f} \hookrightarrow \A^1_{R'_f}$ over the \'etale neighbourhood $\Spec({R'_f})$ of $s$.
\end{proof}

We conclude this section by collecting some consequences on isogenies of Drinfeld shtukas.
Using the finite shtuka equivalence we see that a chain of $\pf^n$-isogenies of type $(r_1, \ldots, r_m)$ on a Drinfeld shtuka $\underline\Ec$ is equivalent to the data of a flag
$$ 0 \se \Hb_1 \se \Hb_2 \se \ldots \se \Hb_m \se \underline\Ec[\pf^{n}] $$
of finite locally free submodule schemes $\Hb_i \se \underline\Ec[\pf^{n}]$ of rank $q^{n \cdot (r_1 + \ldots + r_i)}$ over $S$ with strict $\Fq$-action.
In particular, $\Hb_i/\Hb_{i-1}$ has rank $q^{n r_i}$ and has an induced strict $\Fq$-action.

\begin{prop}
	\label{propIIsoRepSht}
	The stack $\Sht_{r, (r_1, \ldots, r_m)-\pf^{n}\text{-chain}} $ is a Deligne-Mumford stack locally of finite type over $\Fq$. The forgetful map to $\Sht_r$ given by projection to $\underline \Ec$ is schematic and finite. Moreover, the forgetful map $\Sht_{r, (r_1, \ldots, r_m)-\pf^{n}\text{-chain}} \to \Sht_r$ is finite \'etale away from $0$.
\end{prop}
\begin{proof}
	Let $\underline{\Ec} \in \Sht_r(S)$.
	The functor on $S$-schemes
	$$ T \mapsto \left\{ \begin{array}{l}
		\text{flags of quotients } \underline{\Ec}|_{D_{n,T}} \twoheadrightarrow \underline\Fc_m \twoheadrightarrow \ldots \twoheadrightarrow \underline\Fc_1 \twoheadrightarrow 0 \text{ of $\pf^n$-torsion} \\
		\text{finite shtukas such that $\Fc_i$ has rank $ n(r_1 + \ldots + r_i)$ as $\Oc_T$-module}
	\end{array}\right\}$$
	is representable by the closed subscheme of a certain flag variety of quotients of $\underline{\Ec}|_{D_{n,S}}$ (as $\OS$-modules) where both the map $\sigma^*\underline{\Ec}|_{D_{n,S}} \to \underline{\Ec}|_{D_{n,S}}$ and the $\Oc_0/\pf^n$-module structure descend to all the $\Fc_i$.
	As the flag variety is projective, we see that
	$\Sht_{r,  (r_1, \ldots, r_m)-\pf^{n}\text{-chain}} \to \Sht_r$ is schematic and projective.
	
	In order to show finiteness of the map we proceed as in the proof of \cite[Proposition 6.5.1]{Katz1985}. By \cite[Corollaire 18.12.4]{EGA4IV} it suffices to show that the map has finite fibres. 
	Let $\ell$ be an algebraically closed field and let $\underline{\Ec}$ be a rank $r$ shtuka over $\ell$. 
	It suffices to show that $\underline{\Ec}[\pf^n]$ only has finitely many submodule schemes. We know by Proposition \ref{propDescPTor} that for some $h \geq 0$ we have
	$$ \underline\Ec[\pf^n] \cong \alpha_{ q^{nh}} \times \left( \pf^{-n}/\Oc_0 \right)^{r-h}.$$
	As $\ell$ is in particular perfect, any $\Oc_0/\pf^n$-submodule scheme $\Hb \se \underline\Ec[\pf^n]$ factors as $\Hb \cong \Hb^{\text{conn}} \times \Hb^{\text{\'et}}$ but for both factors (which are necessarily submodule schemes of $\alpha_{ q^{nh}}$ and $\left( \pf^{-n}/\Oc_0 \right)^{r-h}$, respectively) there are only finitely many possibilities.
	
	The \'etaleness away from $0$ follows for example from \cite[Lemma 3.3 a)]{Varshavsky2004}.
\end{proof}

	
	\section{Drinfeld $\Gamma_1(\pf^n)$-level structures on shtukas}

In this section, we introduce $\Gamma_1$-type (Drinfeld-) level structures on Drinfeld shtukas adapting similar constructions for Drinfeld modules and elliptic curves. 
We show that the moduli space of Drinfeld shtukas with these level structures is regular following the arguments of \cite{Katz1985}.
For Drinfeld modules, full Drinfeld level structures were studied extensively starting with \cite{Drinfeld1976}, compare for example also \cite{Lehmkuhl2009}. For other kinds of level structures some results are known, \cite{Shastry2007} considers $\Gamma_1(\pf)$-level structures in the rank 2 case and \cite{Kondo2020} study level structures for arbitrary torsion modules and higher rank Drinfeld modules.

We propose a slightly different generalisation of an $M$-level structure on Drinfeld shtukas for a $\pf$-torsion $\Oc_0$-module $M$. In this notation $(\pf^{-n}/\Oc_0)^r$-structures are full level structures and in the rank 2 case $(\pf^{-n}/\Oc_0)$-structures are $\Gamma_1(\pf^n)$-level structures.
For us, it does not seem to be a priori clear that our definition and the analogue of \cite{Kondo2020} agree, even for full level, as is claimed in \cite[(4.1.2.)]{Kondo2020}. For full level structures on Drinfeld modules, this follows from a deep result on the deformation theory of \cite[Proposition 3.3]{Lehmkuhl2009}. We show that the two definitions agree in general in a similar fashion. 
One could also directly adapt the definition of \cite{Katz1985}, as does for example \cite{Taelman2006}. However, it seems to us that this definition does not give the correct moduli space, see Remark \ref{remKMlevelCounterEx}.

Moreover, we define analogues of balanced level structures of \cite{Katz1985} and use this notion of balanced level structure to give a definition of $\Gamma_1(\pf^n)$-level structure for Drinfeld shtukas of arbitrary rank and arbitrary $n \in \N$ in Definition \ref{defnBalIStructSht}.

\subsection{$M$-Structures on Drinfeld shtukas.}

In order to define Drinfeld level structures we use the notion of \emph{full sets of sections}, compare \cite[Section 1.8]{Katz1985}. When working with closed subschemes of a smooth curve, this can be expressed in terms of Cartier divisors by \cite[Theorem 1.10.1]{Katz1985}. Katz and Mazur hoped that the notion of "full sets of sections" might be useful to define level structures for higher dimensional abelian varieties. However, this notion gives rise to a moduli problem which is not even flat over $\Z$ in general (compare \cite{Chai1990}). 
Nevertheless, these issues do not appear in our setting, as Proposition \ref{propEmbItor} allows us to locally work with Cartier divisors in $\A^1$. Note that in a similar fashion Drinfeld level structures are well-behaved when working with one-dimensional $p$-divisible groups, as do \cite{Harris2001} and \cite{Scholze2013b}.

Let $M$ be a finitely generated $\pf^n$-torsion $\Oc_0$-module.
In order to define $M$-structures on Drinfeld shtukas, we would like for an $\Oc_0$-module homomorphism $\iota \colon M \to \underline{\Ec}[\pf^n](S)$ to induce a 
(unique) scheme \emph{generated} by $\iota$, similar to the Cartier divisor generated by an $\Gamma_1(p^n)$-Drinfeld level structure on elliptic curves.
In other words, we are looking for a unique finite locally free scheme over $S$ such that the image of $\iota$ forms a full set of sections for $S$ in the sense of \cite[Section 1.8]{Katz1985}.
This notion is defined as follows. Let $Z$ be a finite locally free $S$-scheme of rank $N$. A set of sections $P_1, \ldots, P_N \in Z(S)$ is called \emph{full set of sections} of $Z$ if for every affine $S$-scheme $\Spec(R) \to S$ and every $f \in \Gamma(Z_R, \Oc_{Z_R})$ we have $\mr{Norm}(f) = \prod_{i = 1}^N f(P_i)$. By \cite[Theorem 1.10.1]{Katz1985}, when $Z \hookrightarrow C$ is embedded as a relative effective Cartier divisor in a smooth curve $C$ over $S$, the set $P_1, \ldots, P_N \in Z(S)$ is a full set of sections of $Z$ if and only if $Z = \sum_{i =1}^N [P_i]$ as Cartier divisors in $C$.

Recall that in general, that given a set of sections $P_1, \ldots, P_{N'} \in Z(S)$ we can neither expect that a finite locally free subscheme $Z'$ of $Z$ of rank $N'$ such that $P_1, \ldots, P_{N'}$ forms a full set of sections of $Z'$ exists nor that it is unique when it exists (compare \cite[Remark 1.10.4]{Katz1985}). However, Proposition \ref{propEmbItor} allows us to construct such a unique scheme in the cases we are interested in.

\begin{lem} 
	\label{lemIstrctDefSch}
	Let $\underline{\Ec}$ be a Drinfeld shtuka over $S$ and let $M$ be a $\pf^n$-torsion module. Let $\iota \colon M \to \underline{\Ec}[\pf^n](S)$ be an $\Oc_0$-linear map.
	\begin{enumerate}
		\item Assume there exists a closed immersion $\underline{\Ec}[\pf^n] \hookrightarrow C$ into a smooth curve $C$ over $S$. Then there exists a unique finite locally free closed subscheme $\Hb$ of $C$ such that the image of $\iota$ (in $C(S)$) forms a full set of sections for $\Hb$. 
		\item There exists at most one finite locally free closed subscheme $\Hb \se \underline{\Ec}[\pf^n]$ such that the image of $\iota$ forms a full set of sections for $\Hb$. 
		\item There exists a (by the previous point necessarily unique) finite locally free closed subscheme $\Hb$ of $\underline{\Ec}[\pf^n]$ such that $\iota$ gives a full set of sections for $\Hb$ if and only if the following equivalent conditions are satisfied:
		\begin{enumerate}
			\item For all \'etale maps $U \to S$ and all closed immersions $\underline{\Ec}[\pf^n]_U \hookrightarrow C$ into a smooth curve over $U$, the Cartier divisor defined by the image of $\iota$ in $C(U)$ is a subscheme of $ \underline{\Ec}[\pf^n]_U$.
			\item There exists an \'etale cover $\{ U_i\}_{i \in I}$ of $S$ and for each $i \in I$ a smooth curve $C_i$ over $U_i$ together with a closed immersion $\underline{\Ec}[\pf^n]_{U_i} \hookrightarrow C_i$ such that the Cartier divisor defined by the image of $\iota$ in $C_i(U_i)$ is a subscheme of $ \underline{\Ec}[\pf^n]_{U_i}$ for all $i \in I$.
		\end{enumerate}
		The existence of such an $\Hb$ is a closed condition on $S$, defined locally on $S$ by finitely many equations.
	\end{enumerate}
\end{lem}

\begin{proof}
	\begin{enumerate}
		\item This is \cite[Theorem 1.10.1]{Katz1985}. The scheme $\Hb$ is the Cartier divisor $ \sum_{\alpha \in M} [\iota(\alpha)]$.
		\item This is clear from the previous point, as \'etale-locally on $S$, $\underline{\Ec}[\pf^n]$ admits an embedding into a smooth curve over $S$ by Proposition \ref{propEmbItor}.
		\item It is clear that the existence of an $\Hb$ implies condition (a), and that condition (a) implies (b) using Proposition \ref{propEmbItor}. Let us now assume that condition (b) is satisfied.
		We denote by $\Hb_i$ the Cartier divisor in $C_i$ defined by $\iota$. We can glue the $\Hb_i$ to form a finite locally free scheme $\Hb$ over $S$ by the uniqueness in the previous point. It is clear that $\iota$ forms a full set of sections for $\Hb$, this can be checked \'etale-locally on $S$.
		
		In order to check that the locus of existence of $\Hb$ is closed in $S$, we may choose an \'etale cover $\{U_i\}$ of $S$ together with embeddings of $\underline{\Ec}[\pf^n]$ into a smooth curve over $U_i$. The assertion on $U_i$ follows from \cite[Key Lemma 1.3.4]{Katz1985}.
	\end{enumerate}
\end{proof}

We can now give our definition of $M$-structures for shtukas.
 
\begin{defn}
	\label{defnIstructSht}
	Let $\underline{\Ec}  \in \Sht_r(S)$ be a rank $r$ shtuka over $S$. Let $M$ be a finitely generated $\Oc_0/\pf^n$-module. A \emph{$M$-structure} on $\underline{\Ec}$ is an $\Oc_0$-module homomorphism $\iota \colon M \to \underline{\Ec}[\pf^{n}](S)$ such that there exists a finite locally free subscheme $\Hb$ of $\underline{\Ec}[\pf]$ of rank $|M[\pf]|$ such that the image of the restriction $\iota|_{M[\pf]}$ of $\iota$ to the $\pf$-torsion forms a full set of sections for $\Hb$ in the sense of \cite[Section 1.8]{Katz1985}. 
\end{defn}

\begin{remark}
	Note that in the theory of Drinfeld modules the modules $M$ and $\underline{\Ec}[\pf]$ would usually be considered as $\pf$-torsion $A$-modules, where $A = \Gamma(X \setminus \{\infty\}, \OX)$. As $A/\pf^n \cong \Oc_0/\pf^n$ this does not give a different notion of level structures. 
	In our context, working with $\Oc_0$- instead of $A$-modules seems more natural and should stress that the level structure only depends on local data of $\underline{\Ec}$ at $0$ and in particular not on the choice of $\infty$.
\end{remark}

\begin{remark}
	\label{remKMlevelCounterEx}
	Our definition is an analogue of Drinfeld's original definition of full level structures in \cite{Drinfeld1976}.
	The definition in \cite{Katz1985} for elliptic curves is slightly different.
	The direct analogue of their definition asserts that there is a finite locally free  subgroup scheme $\Hb \se \underline{\Ec}[\pf^n]$ of rank $|M|$ such that $\iota$ is a full set of sections for $\Hb$ (instead of the corresponding assertion only for the $\pf$-torsion).
	We show that this is implied by our definition in Proposition \ref{propSchIStruct}.	
	For full level structures on Drinfeld modules, \cite[3, Proposition 3.1]{Lehmkuhl2009} and \cite{Wiedmann2010} show that the two notions are equivalent. 
	However, this is not true in general as we can see in the following example. 		
	Consider $X = \P^1_\Fq$, and $S = D_2 = 2[0] = \Spec(\Fq[\zeta]/(\zeta^2))$ viewed as an $X$-scheme via the canonical inclusion. Then the map 
	$$ \ph \colon \sigma^*\OXS^2 \xdashrightarrow{\sMatrix{0}{\varpi-\zeta}{1}{-\zeta}} \OXS^2 $$
	defines a rank 2 Drinfeld shtuka over $S$. Its schemes of $\pf$- and $\pf^2$-torsion points  are given by
	$$ \underline{\Ec}[\pf] = \Spec \left( R[t]/(t^{q^2} + \zeta t^{q} + \zeta t)\right) \qquad \text{and} \qquad \underline{\Ec}[\pf^2] = \Spec \left( R[t]/(t^{q^4} + \zeta t^{q^3} + \zeta t^{q^2})\right), \quad \text{respectively}. $$
	Then the constant zero map $\iota \colon \pf^{-2}/\Oc_0 \to R, \varpi^{-2} \mapsto 0$ defines the subscheme of $\Spec(R[t]/(t^{q^2})) \se \underline{\Ec}[\pf^2]$.
	However, the restriction of $\iota$ to $\pf^{-1}/\Oc_0$ induces the subscheme $\Spec(R[t]/(t^{q}))$  of $\underline{\Ec}[\pf^2]$, which is not a subscheme  of $\underline{\Ec}[\pf]$.
	Thus, $\iota$ is not an $\pf^{-2}/\Oc_0$-structure in the sense of our definition.
	Hence, the definition of \cite{Katz1985} does not yield well-defined level maps in our setting and thus does not seem to be adequate here.
	
	We also do not require the subscheme defined by $\iota$ to be a subgroup scheme as we show this is already automatic below in Proposition \ref{propSchIStruct}. Moreover, we show that it is even automatically an $\Oc_0$-module subscheme.
\end{remark}

\begin{prop}
	\label{propResIstructSht}
	Let $\underline{\Ec} \in \Sht_r(S)$ be a Drinfeld shtuka over $S$ and let $M$ be a finitely generated $\Oc_0/\pf^n$-module. Let $\iota \colon M \to \underline{\Ec}[\pf^{n}](S)$ be an $M$-structure on $\underline{\Ec}$ and let $M' \se M$ be a submodule. The restriction of $\iota$ to $M'$ defines an $M'$-structure on $\underline{\Ec}$.
\end{prop}
\begin{proof}
	\'Etale locally, the Cartier divisor defined by the restriction of $\iota$ to $M'[\pf]$ is a closed subscheme of the Cartier divisor defined by $\iota|_{M[\pf]}$, which in turn is a closed subscheme of $\underline{\Ec}[\pf]$ by assumption. The assertion follows from Lemma \ref{lemIstrctDefSch}.
\end{proof}

In the \'etale case, we have the following descriptions of $M$-structures.
\begin{prop}
	\label{propLevelEtaleSht}
	Let $\underline{\Ec} \in \Sht_r(S)$ be a Drinfeld shtuka over $S$ and let $M$ be a finitely generated $\Oc_0/\pf^n$-module.
	Let $\iota \colon M \to \underline{\Ec}[\pf^{n}](S)$ be an $\Oc_0$-module homomorphism. The following are equivalent:
	\begin{enumerate}
		\item 
		\label{propLevelEtShtField}
		For every geometric point $\Spec \ell \to S$, the induced homomorphism 
		$$\iota_{\ell} \colon M \to \underline{\Ec}[\pf^{n}](\ell)$$
		is injective.
		\item
		\label{propLevelEtShtCartier} 
		The map $\iota$ defines a locally free closed subscheme of $\underline{\Ec}[\pf^{n}]$ which is finite \'etale over $S$.
		\item
		\label{propLevelEtShtAmod}			 
		The map $\iota$ defines a closed immersion of $\Oc_0/\pf^{n}$-module $S$-schemes $$M_S \hookrightarrow \underline{\Ec}[\pf^{n}]$$
		and $\iota$ is a full set of sections for the image of this map (as subscheme of $\underline{\Ec}[\pf^{n}]$).
	\end{enumerate}
	If the equivalent conditions (\ref{propLevelEtShtField})-(\ref{propLevelEtShtAmod}) are satisfied, $\iota$ is an $M$-structure on $\underline{\Ec}$. 
	Moreover, when $S$ is connected, these conditions are equivalent to saying that $M \to \Hb(S)$ is an isomorphism of $\Oc_0$-modules for some constant closed finite locally free $\Oc_0$-module subscheme $\Hb$ of $\underline{\Ec}$.
	
	Moreover, these conditions are automatically satisfied when the characteristic of $\underline{\Ec}$ is away from $0$ and $\iota$ is an $M$-structure on $\underline{\Ec}$.
\end{prop}	
\begin{proof}
	We adapt the proof of an analogous assertion for elliptic curves in \cite[Lemma 1.4.4]{Katz1985}.
	
	\noindent
	(\ref{propLevelEtShtCartier}) $\Leftrightarrow$ (\ref{propLevelEtShtAmod}): This follows directly from the set-theoretic analogue in \cite[Proposition 1.8.3]{Katz1985}.
	
	\noindent
	(\ref{propLevelEtShtField}) $\Leftrightarrow$ (\ref{propLevelEtShtAmod}): 
	The map $\iota$ defines a map of $\Oc_0$-module schemes $M_S \rightarrow \underline{\Ec}[\pf^{n}]$.
	We may work \'etale-locally on $S$ and assume that we can embed $\underline{\Ec}[\pf^{n}]$ into a smooth curve $C$ over $S$. Let us denote by $D$ the Cartier divisor in $C$ defined by $\iota$. We can check that the natural map $M_S \to D$ is an isomorphism on geometric points as in the proof of \cite[Lemma 1.4.4]{Katz1985}, and this is clearly satisfied if and only if $\iota$ is injective on geometric points.
	
	Let us now assume that (\ref{propLevelEtShtField})-(\ref{propLevelEtShtAmod}) are satisfied.
	By (\ref{propLevelEtShtAmod}) the restriction of $\iota$ to $M[\pf]$ defines a subscheme of $\underline{\Ec}[\pf]$. Thus, $\iota$ is an $M$-structure on $\underline{\Ec}$. 
	
	Now assume that $\pf$  is away from the characteristic.
	Let $\iota$ be an $M$-structure on $\underline{\Ec}$. We check that condition (\ref{propLevelEtShtField}) is satisfied. Let  $\Spec \ell \to S$ be a geometric point of $S$. By Proposition \ref{propDescPTor}, we have an $\Oc_0$-linear isomorphism
	$ \underline\Ec[\pf^{n}](\ell) \cong  (\pf^{-n}/\Oc_0)^r.$
	Now $\iota_\ell$ is injective if and only if the restriction $\iota_\ell|_{M[\pf]}$ is injective, as multiplying a non-trivial element $m$ in the kernel of $\iota_\ell$ by the maximal power of $\varpi$ that does not kill $m$ produces a non-trivial element in the kernel of $\iota_\ell|_{M[\pf]}$. The injectivity of  $\iota_\ell|_{M[\pf]}$ follows by assumption and the implication (\ref{propLevelEtShtCartier}) $\Rightarrow$ (\ref{propLevelEtShtField}) in the $n = 1$ case.	
\end{proof}
\begin{remark}
	If the characteristic is away from $0$, Drinfeld level structures agree with corresponding classical level structures.  
	Let $\underline{\Ec} \in \Sht_r(S)$ be such that the characteristic is away from $0$. By the previous Proposition \ref{propLevelEtaleSht}, a full level structure is an isomorphism of \'etale $\Oc_0/\pf^{n}$-module schemes over $S$
	$$ \left( \pf^{-n}/\Oc_0 \right)^{r}_S \xrightarrow{\cong} \underline{\Ec}[\pf^{n}]$$
	by. By \cite[Proposition 2.2]{Drinfeld1987a}, this is the same as giving a trivialisation of $\underline \Ec|_{D_{n,S}}$.
\end{remark}

\subsection{Regularity of the moduli stack of shtukas with $M$-structures.}

We show the main result on $M$-structures: the corresponding moduli problem gives rise to a Deligne-Mumford stack, which we show to be regular following the corresponding result on elliptic curves in \cite{Katz1985}.
\begin{prop}
	\label{propIstructShtRep}
	Let $\underline{ \Ec}$ be a Drinfeld shtuka over $S$ and let $M$ be a finitely generated $\Oc_0/\pf^n$-module.
	The functor on $S$-schemes 
	$$ T \mapsto  \{ M\text{-structures on $\underline{ \Ec}_T$} \}$$
	is representable by a finite $S$-scheme locally of finite presentation. Moreover, it is finite \'etale over $S$ if $\pf$ is away from the characteristic of $\underline{\Ec}$ and can in this case \'etale locally on $S$ be represented by the constant $S$-scheme 
	$$S \times \{ \text{injective $\Oc_0$-module homomorphisms } M \hookrightarrow (\pf^{-n}/\Oc_0)^r \}.$$
\end{prop}
\begin{proof}
	We proceed as in the proof of the corresponding assertions for elliptic curves in \cite[Proposition 1.6.2, Proposition 1.6.4 and Corollary 3.7.2]{Katz1985}.
	By the classification of finitely generated modules over principal ideal domains, there exists an isomorphism of $\Oc_0/\pf^n$-modules $M \cong (\Oc_0/\pf^{n_1}) \oplus \ldots \oplus (\Oc_0/\pf^{n_m})$ for some $m \geq 0$ and integers $1 \leq n_i \leq n$ for $1 \leq i \leq m$. Using this isomorphism, we find for a scheme $T$ over $S$ that 
	$$ \Hom_{\Oc_0}(M, \underline{\Ec}[\pf^n](T)) = \prod_i \Hom_{\Oc_0/\pf^{n_i}}(\Oc_0/\pf^{n_i}, \underline{\Ec}[\pf^{n_i}](T)) = \prod_i \underline{\Ec}[\pf^{n_i}](T).$$
	The functor of $M$-structures on $\underline{\Ec}$ is clearly represented by the closed subscheme of $\prod_i \underline{\Ec}[\pf^{n_i}](T)$ over which the universal homomorphism defines an $M$-structure on $\underline{\Ec}$. This is a closed subscheme defined locally by finitely many equations by Lemma  \ref{lemIstrctDefSch}.
	
	Now assume that the characteristic of $\underline{\Ec}$ is away from $0$. By the above, it suffices to show that the scheme is formally \'etale. 
	Therefore, let $T_0 \se T$ be a closed subscheme defined by a locally nilpotent sheaf of ideals. Let $\iota_0$ be an $M$-structure on $\underline{\Ec}_{T_0}$. Then $\iota_0$ factors through $\underline{\Ec}[\pf^{n}](T_0)$. As $\pf$ is away from the characteristic of $\underline{\Ec}$, the scheme $\underline{\Ec}[\pf^{n}]$ is \'etale over $S$. 
	We construct a map $\iota \colon M \to \underline{\Ec}[\pf^{n}](T)$ where we associate to $m \in M$ the unique lift of $\iota_0(m)$ to $\underline{\Ec}[\pf^{n}](T)$. As lifts are unique, the $\Oc_0$-linearity of $\iota_0$ implies that $\iota$ is an $\Oc_0$-linear map as well.
	We check that $\iota$ defines an $M$-structure on $\underline{\Ec}$. Using Proposition \ref{propLevelEtaleSht} this can be done on geometric points, but the geometric points of $T_0$ and $T$ agree. 
	
	For the second claim, we may assume that $\underline{\Ec}[\pf^n] \cong (\pf^{-n}/\Oc_0)^{r}_S$ by Proposition \ref{propDescPTor}. So, the claim follows from Proposition \ref{propLevelEtaleSht} (\ref{propLevelEtShtAmod}).
\end{proof}

We denote by $\Sht_{r, M}$ the stack parametrising shtukas of rank $r$ with an $M$-structure as defined above.
\begin{thm}
	\label{thmIstructShtReg}
	Let $M$ be a submodule of $(\pf^{-n}/\Oc_0)^r$.
	The stack $\Sht_{r, M}$ is a regular Deligne-Mumford stack locally of finite type over $\Fq$. Moreover, the forgetful map $\Sht_{r,M} \to \Sht_r$ is schematic and finite flat. It is finite \'etale away from $0$. 
\end{thm}
\begin{proof}
	By Proposition \ref{propIstructShtRep}, the forgetful map to $\Sht_r$ is schematic and finite. In particular, $\Sht_{r, M}$ is a Deligne-Mumford stack locally of finite type over $\Fq$ (since $\Sht_r$ is a DM-stack locally of finite type over $\Fq$). Also by Proposition \ref{propIstructShtRep}, the forgetful map is finite \'etale away from $0$.
	
	We proceed as in the proof of \cite[Theorems 5.1.1 and 5.2.1]{Katz1985}. Again since $\Sht_r$ is a smooth DM-stack of dimension $(2r-1)$ over $\Fq$, we find an \'etale presentation $S \to \Sht_r$ by a $(2r-1)$-dimensional smooth scheme $S$ over $\Fq$. We denote by $T = S \times_{\Sht_{r}} \Sht_{r, M}$. 
	
	We denote by $U \se S$ the set of points in $s \in S$ such that the local rings at all points in $T$ over $s$ are regular. Then $U$ is open in $S$, as its complement is the image under a finite (hence closed) map of the non-regular locus in $T$, which is closed in $T$ as $T$ is locally of finite type over a perfect field.
	
	In order to show that $U = S$, it suffices to show that all closed points of $S$ are contained in $U$.
	As the map $T \to S$  is \'etale away from $0$, clearly all points away from $0$ are contained in $U$. It remains to check that all closed points in the fibre over $0$ are in $U$.
	By passing to the completion of the strict henselisation we are reduced to showing that for all $\ell$-valued points $\bar{s}$ of $S$ in the fibre over $0$, where $\ell$ is an algebraic closure of $\Fq$, the complete local rings at all $\ell$-valued points of $T$ over $\bar{s}$ are regular. 
	Note that the completion of the strict henselisation at a closed point $s \in S$ over $0$ is then given by $\widehat{\Oc_{S, s}^{\mr{sh}}} = \widehat{\Oc}_{S_{ \ell\dsq{\varpi}},\overline{s}}$, where $S_{\ell\dsq{\varpi}}$ denotes the base change $S \times_X \Spec(\ell\dsq{\varpi})$.
	
	Let us fix some $\ell$-valued point $\bar{s}$ of $S$ over $0$ and let $\underline{\Ec}_0 \in \Sht_r(\ell)$ be the corresponding shtuka. 
	By \citestacks{07N9}, the disjoint union of the spectra of the completions of all local rings at $\ell$-valued points of $T_{\ell\dsq{\varpi}}$ over $\overline{s}$ is given by the scheme 
	$$ T' = T_{\ell\dsq{\varpi}} \times_{S_{\ell\dsq{\varpi}}} \Spec( \widehat{\Oc}_{S_{ \ell\dsq{\varpi}},\overline{s}}). $$
	As $S$ is \'etale over $\Sht_r$, we get by \cite[Proposition 3.3]{Drinfeld1976} that $ \widehat{\Oc}_{S_{ \ell\dsq{\varpi}},\overline{s}} \cong \ell \dsq{\varpi, T_1, \ldots, T_{2r-2}}$, 
	which identifies the pullback of the universal shtuka to $\Spec( \widehat{\Oc}_{S_{ \ell\dsq{\varpi}}})$ with the universal deformation $\underline{\Ec}$ of $\underline{\Ec}_0$. In particular, 
	$T' = (\Sht_{r, M})_{\underline{\Ec}} := \Sht_{r, M} \times_{\Sht_{r}}  \Spec(\widehat{\Oc}_{S_{ \ell\dsq{\varpi}},\overline{s}})$, where we interpret $\underline{\Ec}$ as the corresponding $\Spec(\widehat{\Oc}_{S_{ \ell\dsq{\varpi}},\overline{s}})$-valued point of $\Sht_r$. 
	
	Thus, by construction, $T'$ only depends on the scheme of $\pf^n$-torsion points of $\underline{\Ec}$ (and thus on its local shtuka at $0$ by \cite[Theorem 7.6]{Hartl2017a}), which in turn by the Serre-Tate Theorem (Proposition \ref{propSerreTateSht}) only depends on the local shtuka at $0$ of $\underline{\Ec}_0$, which are classified up to isomorphism by their Newton polygons by Remark \ref{remClassNewtonIso}. In particular, $\bar{s}$ is contained in $U$ if and only if $U$ contains all points in the fibre over $0$ in the same Newton stratum. By Theorem \ref{thmNewtonStratStrongStrat} it thus suffices to show that $U$ contains a basic point.
	
	Let thus $\bar{s}$ be a basic point in characteristic $\pf$ (recall that such a point exists by Proposition \ref{propSSShtEx}) corresponding to a shtuka $\underline{\Ec}_0$. 
	By Proposition \ref{propDescPTor}, we get that $\underline{\Ec}_0[\pf^n](\ell) = \{0\}$ for all $n \in \N$, so in particular the only possible $M$-structure is the zero map (which is readily checked to be an $M$-structure as $M$ is a submodule of $(\pf^{-n}/\Oc_0)^r$), so there is exactly one point lying over $\bar{s}$. 
	This means that $T'$ is the spectrum of the complete local ring pro-representing the deformation functor of the unique point lying over $\bar{s}$.
	Note that since $\underline{\Ec}$ is basic, the associated divisible module at $0$ is formal (in the sense of \cite[§1]{Drinfeld1976} or \cite[Definition 1.1]{Hartl2019}). As $M$-structures only depend on the local shtukas, the Serre-Tate Theorem is also compatible with $M$-structures and we are thus reduced to showing that the deformation functor of formal modules with $M$-structures is representable by a $r$-dimensional regular local ring. 
	
	This can be shown as in \cite[Proposition 4.3]{Drinfeld1976}. 
	We sketch the argument.
	We write $M = (\pf^{-n_1}/\Oc_0) \oplus \ldots \oplus (\pf^{-n_{r'}}/\Oc_0)$. Note that by assumption $r' \leq r$. 
 	The lemma in the proof of \cite[Proposition 4.3]{Drinfeld1976} shows that the deformation functor on formal modules with $(\pf^{-1}/\Oc_0)^{r'}$-structure is pro-represented by a complete regular local ring $R_1$ finite flat over $R_0 = \ell \dsq{\varpi, T_1, \ldots, T_{r-1}}$ whose maximal ideal is generated by $\iota(e_1), \ldots \iota(e_{r'}), T_{r'}, \ldots, T_{r-1}$, where $e_1, \ldots, e_{r'}$ is a basis of $(\pf^{-1}/\Oc_0)^{r'}$ and $\iota$ is the universal $(\pf^{-1}/\Oc_0)^{r'}$-structure. This settles the case that $M$ is $\pf$-torsion. 
 	As a next step we show the claim for $M[\pf^m]$ by induction on $m$. 
 	Let us assume that a complete regular local ring $R_m$ finite flat over $R_{m-1}$ pro-represents the deformation functor of formal modules with $M[\pf^m]$-level and that $\iota_m(\varpi^{-\min\{n_1, m\}}), \ldots, \iota_m(\varpi^{-\min\{n_{r'}, m\}}), T_{r'}, \ldots, T_r$ forms a system of local parameters for $R_m$, where $\iota_m$ is the universal $M[\pf^m]$-level structure. Let us denote by $i_1, \ldots, i_j$ the indices such that $n_i > m$. Then
 	$$ R_{m+1} = R_m \dsq{\tilde{T}_{i_1}, \ldots, \tilde{T}_{i_j}}/(e_\varpi(\tilde{T}_{i_1}) - \iota_m(\varpi^{-\min\{n_{i_1}, m\}}), \ldots, e_\varpi(\tilde{T}_{i_j}) - \iota_m(\varpi^{-\min\{n_{i_j}, m\}})).$$
 	This also shows that $R_{m+1}$ is regular and finite flat over $R_m$. Moreover, it has a system of local parameters as desired.
\end{proof}

The regularity allows us to collect some consequences.
We show that an $M$-structure defines an $\Oc_0$-module subscheme $\Hb \se \underline{\Ec}[\pf^n]$.
\begin{prop}
	\label{propSchIStruct}
	Let  $\underline{\Ec}  \in \Sht_r(S)$ and let $\iota \colon M \to \underline{\Ec}[\pf^{n}](S)$ be an $M$-structure on $\underline{\Ec}$ for some submodule $M \se (\pf^{-n}/\Oc_0)^r$.
	\begin{enumerate}
		\item 	
		There exists a (necessarily unique) finite locally free closed subscheme $\Hb \se \underline{\Ec}[\pf^{n}]$ of rank $|M|$ such that the image of $\iota$ forms a full set of sections for $\Hb$. 
	
		Moreover, for each submodule $M'$ of $M$ the restriction of $\iota$ to $M'$ defines an $M'$-structure on $\underline{\Ec}$ and in particular there exists a finite locally free closed subscheme $\Hb_{M'}$ of $\underline{\Ec}$ of rank $M'$ such that $\iota|_{M'}$ forms a full set of sections for $\Hb_{M'}$.
		\label{propSchIStructEx}
		\item 
		$\Hb$ is an $\Oc_0$-module subscheme of $\underline{\Ec}[\pf^{n}]$ such that the $\Fq$-module structure on $\Hb$ is strict.
		\label{propSchIStructMod}
	\end{enumerate}
\end{prop}
	We call $\Hb$ the \emph{subscheme generated by $\iota$} and the map $\iota$ a \emph{$M$-generator of $\Hb$}. To be more precise, when we say that a map $\iota$ is an $M$-generator of a finite locally free closed subscheme $\Hb \se \underline{\Ec}[\pf^{n}]$ of rank  $|M|$, we really mean both that $\iota$ gives a full set of sections of $\Hb$ \emph{and} that $\iota$ is an $M$-structure on $\underline{\Ec}$ (recall that the first condition does not imply the second one, compare Remark \ref{remKMlevelCounterEx}).

For full level structures on Drinfeld modules the assertion is essentially shown in \cite[3, Proposition 3.3]{Lehmkuhl2009}. The proposition also implies that for general $M$-structures on Drinfeld modules our definition agrees with the one given in \cite{Kondo2020}.
\begin{proof}
	That the restriction to $M'$ defines an $M'$-structure is Proposition \ref{propResIstructSht} and thus, the second statement in (\ref{propSchIStructEx}) follows from the first. 
	In order to show both the first part of (\ref{propSchIStructEx}) and (\ref{propSchIStructMod}), we may assume by reduction to the universal case that $S$ is locally Noetherian and flat over $X'$. Both assertions are true away from $0$ by Proposition \ref{propLevelEtaleSht}. 
	It thus remains to show that the conditions are closed in both cases.
	For the first part of (\ref{propSchIStructEx}) this follows from Lemma \ref{lemIstrctDefSch}.
	
	For (\ref{propSchIStructMod})  we may additionally assume that $S = \Spec(R)$ is affine and that we can embed $\underline{\Ec}[\pf^{n}]$ in $\A^1_R$ as the assertion is \'etale local on $S$ (for the strictness of the $\Fq$-action this is \cite[Lemma 4.4]{Hartl2017a}).
	The locus where the group structure on $\underline{\Ec}$ restricts to a group structure on $\Hb$ is closed by the argument of \cite[Corollary 1.3.7]{Katz1985}. 
	
	By the discussion above, we can write $\underline{\Ec}[\pf^{n}] = \Spec(R[t]/(f))$ for some monic polynomial $f \in R[t]$ and $\Hb = \Spec(R[t]/(g))$ for some monic polynomial $g \in R[t]$ such that $f \in (g)$.
	Then the $\Oc_0/\pf^n$-module structure restricts to $\Hb$ if and only if for each $a \in \Oc_0/\pf^n$ the map $e_a$ induces a map 
		\begin{center}
			\begin{tikzcd}
				R[t]/(f) \arrow[d, two heads] \arrow[r, "e_a"] & R[t]/(f) \arrow[d, two heads] \\
				R[t]/(g) \arrow[r, dashed] & R[t]/(g),
			\end{tikzcd}
		\end{center}
	in other words, that $e_a(g)(t) = g(e_a^\sharp(t)) \equiv 0 \mod g$, where $e_a^\sharp(t) \in R[t]$ is a polynomial defining the map $e_a$. Thus, the locus where $\Hb$ is an $\Oc_0/\pf^n$-module subscheme is the closed subscheme of $\Spec(R)$ where all the coefficients of the remainders of $g(e_a^\sharp(t))$ modulo $g$ vanish. Note that this is clearly independent of the choice of $e_a^\sharp$.
		
	It remains to show that the locus where the $\Fq$-action is strict is closed. 
	As $\underline{\Ec}[\pf^{n}]$ carries a strict $\Fq$-action by construction, we have a lift of the $\Fq$-action to $\underline{\Ec}[\pf^{n}]^\flat = \Spec(R[t]/(tf))$ by \cite[Remark b) after Definition 1]{Faltings2002}, compare also \cite[Lemma 4.4 in the arXiv version]{Hartl2019}. By the same argument as for the $\Oc_0/\pf^n$-module structure, the $\Fq$-action restricts to a map on the deformation $\Hb^\flat = \Spec(R[t]/(tg)) \se \underline{\Ec}[\pf^{n}]^\flat$. That it induces the correct operation on the co-Lie complex of $(\Hb, \Hb^\flat)$ is again a closed condition.
\end{proof}

We now define $M$-cyclic isogenies.
\begin{defn}
	\label{defnIcycSht}
	Let $\underline{\Ec} \in \Sht_r(S)$ and let $M$ be a finitely generated $\pf^n$-torsion module. 
	\begin{enumerate}
		\item A \emph{$M$-generator} of a finite locally free subscheme $\Hb \se \underline{ \Ec}[\pf^n]$ is an $M$-structure $\iota$ on $\underline{\Ec}$ such that the subscheme of $\underline{\Ec}[\pf^{n}]$ defined by $\iota$ is $\Hb$.
		\item A finite locally free subscheme $\Hb \se \underline{ \Ec}[\pf^n]$ is called \emph{$M$-cyclic}  if there is an fppf cover $S' \to S$ such that $\Hb_{S'}$ admits an $M$-generator.
		\item A $\pf^n$-isogeny of Drinfeld shtukas $f \colon \underline{ \Ec} \to \underline{ \Ec}'$ is called \emph{$M$-cyclic} if $\Dr_q(\coker(f))$ is $M$-cyclic.
	\end{enumerate}
\end{defn}
Note that an $M$-cyclic subscheme necessarily has rank $|M|$. 
We also use the term \emph{$\pf^{n}$-cyclic} as abbreviation for $(\pf^{-n}/\Oc_0)$-cyclic submodule schemes or isogenies, respectively.

\begin{lem}
	\label{lemIcycOrdI}
	Let $M$ be a submodule of $(\pf^{-n}/\Oc_0)^r$.
	Every $M$-cyclic subscheme $\Hb \se \underline{ \Ec}[\pf^n]$ is an $\Oc_0$-module subscheme with strict $\Fq$-action.
\end{lem}
\begin{proof}
	All of the assertions can be checked fppf-locally on the base, where they follow from Proposition \ref{propSchIStruct}.
\end{proof}

We collect two representability results.
\begin{prop}
	Let $\underline{\Ec} \in \Sht_r(S)$ and assume that its characteristic is away from $0$. Let $M$ be a finitely generated $\Oc_0/\pf^n$-module. Then the  functor on $S$-schemes 
	$$ T \mapsto \{ M \text{-cyclic subgroups of } \underline\Ec[\pf^n]_T \}.$$
	is representable by a finite \'etale $S$-scheme. 
	Moreover, \'etale locally on $S$ the functor is representable by the constant $S$-scheme 
	$$ S \times \{ \text{submodules of } (\pf^{-n}/\Oc_0)^r \text{ isomorphic to } M\}. $$
\end{prop}
\begin{proof}
We proceed as in the proof of \cite[Theorem 3.7.1]{Katz1985}. By descent for finite locally free schemes and closed immersions, and the fact that cyclicity is local for the fppf-topology by definition, the functor is a fppf (and hence an \'etale) sheaf. By \'etale descent, it thus suffices to show the second statement. 
We may assume that $\underline\Ec[\pf^n] \cong (\pf^{-n}/\Oc_0)^r_S$ by Proposition \ref{propDescPTor}. By the argument in the proof of \cite[Theorem 3.7.1]{Katz1985}, over a connected base $T$ any finite locally free closed subgroup scheme of a constant scheme is itself constant. The claim follows from the explicit description in Proposition \ref{propLevelEtaleSht}.
\end{proof}

\begin{prop}
	\label{propIGenSht}
	Let $\underline{\Ec} \in \Sht_r(S)$, let $M$ be a finitely generated $\Oc_0/\pf^n$-module and let $\Hb \se \underline{ \Ec}[\pf^{n}]$ be a finite locally free closed $\Oc_0$-module subscheme of rank $|M|$. We consider its functor of generators, i.e. the functor on $S$-schemes 
	$$ T \mapsto \{ M\text{-generators of } \Hb_T \text{ in the sense of Definition \ref{defnIcycSht}} \}.$$
	It is representable by a finite scheme of finite presentation over $S$ denoted by $\Hb^\times$. Moreover, it is finite \'etale when $\Hb$ is \'etale (in particular, when the characteristic of $\underline{\Ec}$ is away from $0$).
\end{prop}

\begin{proof}
	We adapt the proof of \cite[Proposition 1.6.5]{Katz1985}. The functor clearly is representable by the closed subscheme of $\Hom(M, \Hb)$ over which the universal homomorphism is an $M$-structure on $\underline{\Ec}$ (which is a closed condition locally defined by finitely many equations by Lemma \ref{lemIstrctDefSch}) and over which the subscheme defined by the universal homomorphism is $\Hb$, which is also a closed condition given by finitely many equations by Lemma \ref{lemIstrctDefSch} and \cite[Corollary 1.3.5]{Katz1985}.
	
	If $\Hb$ is \'etale, we see as in the proof of Proposition \ref{propIstructShtRep}
	 that $\Hb^\times$ is formally \'etale.
\end{proof}

\subsection{Balanced level structures for shtukas.}

\begin{defn}
	\label{defnBalIStructSht}
	Let $m \in \N$ and let $r_1, \ldots, r_m$ be positive integers such that $\sum_{i = 1}^m r_i \leq r$.  
	A \emph{balanced $\pf^{n}$-level structure of type $(r_1, \ldots, r_m)$} on a Drinfeld shtuka $\underline\Ec$ over $S$ is a chain of isogenies 
	$$ \underline\Ec(\pf^n) = \underline\Ec_{m+1} \xrightarrow{f_{m+1}} \underline\Ec_m \xrightarrow{f_{m}}  \ldots \xrightarrow{f_2} \underline\Ec_1 \xrightarrow{f_{1}} \underline\Ec_{0} = \underline\Ec$$
	such that the composition $f_{m+1} \circ \ldots \circ f_1$ is the inclusion $\underline\Ec(\pf^n) \to \underline\Ec$, 
	together with $(\pf^{-n}/\Oc_0)^{r_i}$-generators of $\Dr_q(\coker(f_i)) \se \underline\Ec_i[\pf^{n}]$ for all $1 \leq i \leq m$ in the sense of Definition \ref{defnIcycSht}.
	We denote by $\Sht_{r, \pf^{n}\text{-bal-}(r_1, \ldots, r_m)}$ the stack parametrising Drinfeld shtukas together with a balanced $\pf^{n}$-level structure of type $(r_1, \ldots, r_m)$.
	
	A \emph{$\Gamma_1(\pf^n)$-level structure} on a Drinfeld shtuka of rank $r$ is a balanced $\pf^{n}$-level structure of type $\mathbf{1}_r = (1, \ldots, 1) \in \Z^r$. We denote by $\Sht_{r, \Gamma_1(\pf^n)} = \Sht_{r, \pf^{n}\text{-bal-}\mathbf{1}_r}$ the stack of Drinfeld shtukas with a $\Gamma_1(\pf^n)$-level structure.
\end{defn}

As for Drinfeld shtukas with chains of isogenies, we see that a balanced $\pf^{n}$-level structure of type $(r_1, \ldots, r_m)$ on a Drinfeld shtuka $\underline\Ec$ is equivalent to the data of a flag
$$ 0 \se \Hb_1 \se \Hb_2 \se \ldots \se \Hb_m \se \underline\Ec[\pf^{n}] $$
of finite locally free submodule schemes $\Hb_i \se \underline\Ec[\pf^{n}]$ of rank $n \cdot (r_1 + \ldots + r_i)$ with strict $\Fq$-action together with $(\pf^{-n}/\Oc_0)^{r_i}$-generators of $\Hb_i/\Hb_{i-1}$ for all $1 \leq i \leq m$.

\begin{lem}
	\label{lemShtBalStructRep}
		The stack $\Sht_{r, \pf^{n}-(r_1, \ldots, r_m)-\mr{bal}}$ is representable by a Deligne-Mumford stack locally of finite type over $\Fq$. The projection  $\Sht_{r, \pf^{n}-(r_1, \ldots, r_m)-\mr{bal}} \to \Sht_r$ is schematic and finite. Moreover, it is finite \'etale away from $0$. 
\end{lem}
\begin{proof}
		We have a well-defined map of stacks
		$$ \Sht_{r, \pf^{n}-(r_1, \ldots, r_m)-\mr{bal}} \to \Sht_{r,(r_1, \ldots, r_m)-\pf^{n}\text{-chain}}.$$
		This map is schematic, finite and moreover  finite  \'etale away from $0$ by Proposition \ref{propIGenSht}.
		The assertions then follow from Proposition \ref{propIIsoRepSht}.
\end{proof}

\begin{prop}
	\label{propShtBalStructReg}
	The Deligne-Mumford stack $\Sht_{r, \pf^{n}-(r_1, \ldots, r_m)-\mr{bal}}$ is regular.
\end{prop}
\begin{proof}
	As in the proof of Theorem \ref{thmIstructShtReg} it suffices to check that the deformation functor of the $\pf$-divisible module of a basic point over $0$ with balanced $\pf^{n}$-level structure of type $(r_1, \ldots, r_m)$ is pro-representable by a regular local ring. 
	By \cite[Proposition 5.2.2]{Katz1985} it suffices to show that the maximal ideal is generated by $r$ elements.
	
	Let $(\Gb_0, (\Hb_{0,i}, \iota_{0,i})_{1 \leq i \leq m})$ be the $\pf$-divisible module of a basic Drinfeld shtuka of rank $r$ over an algebraically closed field $\ell$ in the fibre over $0$ together with a balanced $\pf^{n}$-level structure of type $(r_1, \ldots, r_m)$ on $\Gb_0$. Note that $\Gb_0$ is automatically formal and the level structure is unique, and all the $\iota_{0,i}$ are the zero map. Then by the Serre-Tate theorem the deformation functor of $(\Gb_0, (\Hb_{0,i}, \iota_{0,i})_{1 \leq i \leq m})$ is representable by a complete local ring denoted by $B$.
	Let $(\Gb, ((\Hb_i), (\iota_i))_{1 \leq i \leq m})$ be the universal deformation over $B$. 	
	For every $1 \leq i \leq m$ we choose a basis $e^{(i)}_{1}, \ldots, e^{(i)}_{r_i}$ of $(\pf^{-n}/\Oc_0)^{r_i}$.
	We claim that the maximal ideal of $B$ is generated by $\iota_i(e^{(i)}_j)$ for $1 \leq i \leq m$ and $1 \leq j \leq r_i$ and $T_{r_1+ \ldots + r_m}, \ldots, T_{r-1}$. We proceed as in the proof of \cite[Theorem 5.3.2, (5.3.5)]{Katz1985}.
	We need to check that for every Artin local $\ell\dsq{\varpi}$-algebra $R$ such that $T_{r_1+ \ldots + r_m}, \ldots, T_{r-1}$ vanish in $R$ every deformation $(\tilde\Gb, ((\tilde\Hb_i), (\tilde \iota_i))_{1 \leq i \leq m})$ on $R$ such that all $\tilde{\iota}_i$ are the constant zero maps is itself constant.
	
	Using \cite[Lemma 1.11.3]{Katz1985} we see inductively that the zero map is an $(\pf^{-n}/\Oc_0)^{r_1 + \ldots + r_i}$-structure on $\Hb_i$ for all $1 \leq i \leq m$. In particular, the zero map is an $(\pf^{-n}/\Oc_0)^{r_1 + \ldots + r_i}$-structure on $\tilde{\Gb}$ and thus $\Gb$ is constant by the proof of Theorem \ref{thmIstructShtReg}.
\end{proof}

We collect some consequences. We start by constructing balanced level structures from $(\pf^{-n}/\Oc_0)^{r'}$-structures.
Let  $m \in \N$ and let $r_1, \ldots, r_m$ be positive integers such that $r' = \sum_{i = 1}^m r_i \leq r$. 
Let $(\underline \Ec,  \iota)$ be a Drinfeld shtuka together with an $(\pf^{-n}/\Oc_0)^{r'}$-structure on $S$.
For $1 \leq i \leq m$  the restriction of $\iota$ restricted to the first $r_1 + \ldots + r_i$ components is an $(\pf^{-n}/\Oc_0)^{r_1 + \ldots + r_i}$-structure by Proposition \ref{propResIstructSht} and thus defines an $\Oc_0/\pf^{n}$-submodule scheme $\Hb_i$ of $\underline{\Ec}[\pf^{n}]$ by Proposition \ref{propSchIStruct}.  We denote by $\iota_i$ the induced map $(\pf^{-n}/\Oc_0)^{r_i} \to \Hb_i/\Hb_{i-1} (S)$.

\begin{prop}
	\label{propResBalSht}
	Let $(\underline{\Ec}, \iota)$ be a rank $r$ Drinfeld shtuka together with an $(\pf^{-n}/\Oc_0)^{r'}$-structure over $S$. 
	Using the notation as above, the flag of finite locally free closed submodule schemes $0 = \Hb_0 \se \Hb_1 \se \ldots \se \Hb_m \se \underline\Ec[\pf^{n}]$ together with the maps $(\iota_i)_{1 \leq i \leq m}$ defines a balanced $\pf^{n}$-level structure of type $(r_1, \ldots, r_m)$ on $\underline\Ec$. 
\end{prop}

\begin{proof}
	We follow the proof of \cite[Theorem 5.5.2]{Katz1985}.
	By reduction to the universal case and Theorem \ref{thmIstructShtReg} we may assume that $S$ is flat and affine over $X'$. 
	The assertion is clear when the characteristic of $\underline\Ec$ is away from $0$.
	The condition that $\iota_i$ generates $\Hb_i/\Hb_{i-1}$ is closed and thus the assertion follows by flatness.
\end{proof}

The proposition can also be applied in the following more general situation. Let $\underline{\Ec}$ be a Drinfeld shtuka together with a balanced $\pf^{n}$-level structure of type $(r_1, \ldots, r_m)$ denoted by $((\Hb_i), (\iota_i))$. Let $1 \leq m' \leq m$ and for each $1 \leq j \leq m'$ let $i_j$  and $r'_{i_1 + \ldots + i_{j-1} +1}, \ldots, r'_{i_1 + \ldots + i_{j} }$ be positive integers such that $\sum_{i = 1}^{i_j} r'_{{i_1 + \ldots + i_{j-1} +i}} = r_j$.
By applying Proposition \ref{propResBalSht} to each $\iota_j$ for $1 \leq j \leq m'$, we obtain a well-defined  balanced $\pf^{n}$-level structure of type $(r'_1,\ldots, r'_{i_1 + \ldots + i_{m'}})$ on $\underline\Ec$. This construction thus induces a map of stacks
\begin{equation}
	\label{eqnLvlMapResBalSht}
	\Sht_{r, \pf^{n}-(r_1, \ldots, r_m)-\mr{bal}} \to \Sht_{r, \pf^{n}-(r'_1, \ldots, r'_{i_1 + \ldots + i_{m'}})-\mr{bal}}.
\end{equation}

\begin{cor}
	\label{corLvlMapResBalSht}
	The map (\ref{eqnLvlMapResBalSht}) is finite locally free of constant rank. In particular, fppf-locally on the base, any balanced $\pf^{n}$-level structure of type $(r'_1,\ldots, r'_{i_1 + \ldots + i_{m'}})$ on $\underline{\Ec}$ can be extended to a balanced $\pf^{n}$-level structure of type $(r_1, \ldots, r_m)$ on $\underline\Ec$.
\end{cor}
\begin{proof}
	We follow the proof of  \cite[Corollaries 5.5.3 \& 5.5.4]{Katz1985}
	As both the stacks  $\Sht_{r, \pf^{n}-(r_1, \ldots, r_m)-\mr{bal}}$ 
	and $\Sht_{r, \pf^{n}-(r'_1, \ldots, r'_{i_1 + \ldots + i_{m'}})-\mr{bal}}$ are regular $r$-dimensional and finite over $\Sht_r$, the map (\ref{eqnLvlMapResBalSht}) is necessarily finite flat. 
	The degree can be computed in the \'etale case, where it is clearly constant.
	The second assertion follows immediately.
\end{proof}

\begin{cor}
	\label{corBalStructCycSht} 
	Let $\underline \Ec$ be a Drinfeld shtuka over $S$ together with a balanced $\pf^{n}$-level structure of type $(r_1, \ldots, r_m)$ denoted by  $(\Hb_i, \iota_i)_{1 \leq i \leq m}$. Then $\Hb_i$ is $(\pf^{-n}/\Oc_0)^{r_1 + \ldots + r_i}$-cyclic.
\end{cor}
\begin{proof}
	We use induction on $i$. For $i=1$ the assertion is clear by definition. Let now $i > 1$ and let us assume that $\Hb_{i-1}$ is $(\pf^{-n}/\Oc_0)^{r_1 + \ldots + r_{i-1}}$-cyclic. As the question is fppf-local on $S$, we may assume that  $\Hb_{i-1}$ admits a generator over $S$. Then $0 \se \Hb_{i-1} \se \Hb_i \se \underline{\Ec}[\pf^n]$ together with the generators of $\Hb_{i-1}$ and $\Hb_i/\Hb_{i-1}$ defines a balanced $\pf^n$-level structure of type $(r_1 + \ldots + r_{i-1}, r_i)$ on $\underline{\Ec}$. By Corollary \ref{corLvlMapResBalSht}, it can be completed fppf-locally to an $(\pf^{-n}/\Oc_0)^{r_1 + \ldots + r_i}$-structure. But this exactly means that $\Hb_i$ admits a generator fppf-locally on $S$.
\end{proof}

	\section{Drinfeld $\Gamma_0(\pf^n)$-level structures on shtukas}

	We are now in a position to discuss $\Gamma_0(\pf^{n})$-level structures on Drinfeld shtukas. 
	We closely follow the exposition of \cite[Chapter 6]{Katz1985} for elliptic curves and adapt the arguments to suit our situation.

	\subsection{Main theorem on $\pf^n$-cyclic submodule schemes.}
	
	The goal of this section is to show the following analogue of \cite[Theorems  6.1.1 and 6.4.1]{Katz1985}.
	
	\begin{thm}[Main Theorem on $\pf^{n}$-Cyclic Modules]
		\label{thmMainThmCyc}
		Let $\underline{\Ec} \in\Sht_r(S)$ be a Drinfeld shtuka of rank $r$ over a scheme $S$. Let $\Hb \se \underline\Ec[\pf^n]$ be a finite locally free $\Oc_0/\pf^n$-submodule scheme of rank ${q}^n$ over $S$. 
		\begin{enumerate}
			\item 
			\label{thmMainThmCycGen}
			Suppose that $\Hb$ is $\pf^n$-cyclic and admits a generator $\iota$. 
			Let $D \se \Hb$ be the finite locally free subscheme of $\Hb$ of rank ${q}^{n-1}({q}-1)$ defined by the restriction of $\iota$ to $ (\pf^{-n}/\Oc_0)^\times = (\pf^{-n}/\Oc_0) \setminus (\pf^{-(n-1)}/\Oc_0)$. Then $D = \Hb^\times$ as subschemes of $\Hb$.			
			\item 
			\label{thmMainThmCycGenFree}
			$\Hb$ is $\pf^{n}$-cyclic if and only if its scheme of generators $\Hb^\times$ is finite locally free over $S$ of rank $${q}^{n-1}({q}-1).$$			
			\item 
			\label{thmMainThmCycClosed}
			Cyclicity of $\Hb$ is a closed condition, in the sense that there is a closed subscheme $W \se S$ locally of finite presentation over $S$ such that for any $T \to S$ the pullback $\Hb_T$ is $\pf^{n}$-cyclic if and only if $T \to S$ factors through $W$.
		\end{enumerate} 
	\end{thm} 
	
	\begin{proof}
		Assertion (\ref{thmMainThmCycClosed}) follows from (\ref{thmMainThmCycGenFree}) by the flattening stratification as in the proof of \cite[Theorem 6.4.1]{Katz1985}. 
		We sketch the argument. 
		As a first step we note that in the case where $S = \Spec(k)$ is the spectrum of a field and $\Hb$ is not cyclic, we have $\Hb^\times = \emptyset$. Namely, any generator of $\Hb$ can be defined over a finite extension of $k$, but by assumption $\Hb$ does not admit a generator after any finite extension of $k$. Hence, $\Hb^\times$ does not have any field valued points and is thus the empty scheme.
		
		As the question is Zariski-local on $S$ and both $\Hb$ and $\underline{\Ec}[\pf^n]$ are of finite presentation over $S$, we may assume that $S = \Spec(R)$ is affine and Noetherian. By the above argument and (\ref{thmMainThmCycGenFree}), the maximal rank of $\Hb^\times$ over $S$ is ${q}^{n-1}({q}-1)$. By the flattening stratification, the locus where $\Hb^\times$ has rank ${q}^{n-1}({q}-1)$ and hence $\Hb$ is cyclic by (\ref{thmMainThmCycGenFree}) is closed.
		
		It is also clear that (\ref{thmMainThmCycGenFree}) follows from (\ref{thmMainThmCycGen}) as in \cite[Theorem 6.1.1]{Katz1985}. Namely, if $\Hb^\times$ is finite locally free of rank ${q}^{n-1}({q}-1)$, the diagonal map $\Hb^\times \to \Hb^\times_{\Hb^\times}$ is a section of $\Hb^\times$ after base change along $\Hb^\times \to S$. 
		Hence, $\Hb$ admits a generator after the fppf base change to $\Hb^\times$ and is thus $\pf^n$-cyclic. Conversely, assume that $\Hb$ is $\pf^n$-cyclic. The question is fppf-local on $S$, we may thus assume $\Hb$ admits a generator. The assertion in this case follows from (\ref{thmMainThmCycGen}). 
		
		It thus remains to show (\ref{thmMainThmCycGen}). We adapt the proof of \cite[Theorem 6.1.1]{Katz1985}. The assertion is certainly true when the characteristic of $\underline\Ec$ is away from $\pf$ by Proposition \ref{propLevelEtaleSht}.  
		
		As a first step we show that $D \se \Hb^\times$. It is clear by definition that $D \se \Hb$.  By reduction to the universal case and using Theorem \ref{thmIstructShtReg}, it suffices to consider the case when $S$ is Noetherian and flat over $X'$ and as the question is local on $S$, we may further assume that $S = \Spec(R)$ is affine. It follows that $D$ is then also flat over $X'$. In order to show that $D \se \Hb^\times$, we show that the tautological section of $\Hb$ over $D$ induced by the inclusion $D \hookrightarrow \Hb$ is a generator of $\Hb_D$. This is certainly true away from $\pf$. 
		The claim follows from the flatness of $D$ over $X'$ and the fact that the locus where $D \to \Hb_D$ is a generator is closed in $D$ by Lemma \ref{lemIstrctDefSch}. 
		Hence, $D \hookrightarrow \Hb$ factors over $\Hb^\times$ and the induced map $D \se \Hb^\times$ is necessarily a closed immersion. 
		
		In order to show that the closed immersion $D \hookrightarrow \Hb^\times$ is an isomorphism, we introduce the following two auxiliary moduli problems.
		\begin{align*}
			\Xc_1(S) &= \left\langle (\underline\Ec, \Hb, \iota_1, P) \colon  \begin{array}{l}
				\underline\Ec \in\Sht_{r}(S),  \
				\Hb \se \underline\Ec[\pf^{n}] \text{ a $\pf^{n}$-cyclic submodule scheme, }  \iota_1 \in \Hb^\times(S),\\
				 P \in D(S) \text{ such that } 
				\{ \alpha P \colon \alpha \in (\Oc_0/\pf^{n})^\times\} \text{ is a full set of sections for } D
			\end{array}
			\right\rangle \\
			\Xc_2(S) &= \left\langle (\underline\Ec, \Hb, \iota_1, \iota_2) \colon  \begin{array}{l}
				\underline\Ec \in\Sht_{r}(S), \
				\Hb \se \underline\Ec[\pf^{n}] \text{ a $\pf^{n}$-cyclic submodule scheme, }
				\iota_1, \iota_2 \in \Hb^\times(S)  
			\end{array}
			\right\rangle
		\end{align*}
		It is clear that $\Xc_1$ and $\Xc_2$ are stacks, and both map to $\Sht_{r, (\pf^{-n}/\Oc_0)}$ by forgetting $P$ and $\iota_2$, respectively. This maps are clearly schematic and finite as they are representable by (a closed subscheme of) the finite schemes $D$ and $\Hb^\times$, respectively.
		Note that both $\Xc_1$ and $\Xc_2$ have a unique point lying over a supersingular point of $\Sht_r$ over an algebraically closed field in characteristic $\pf$.
		
		Since $D \se \Hb^\times$ there is a natural map $\Xc_1 \to \Xc_2$ over $\Sht_{r, (\pf^{-n}/\Oc_0)}$ which is an isomorphism away from $\pf$ as noted above.
		We show that the map is an isomorphism. By an argument as in the proof of Theorem \ref{thmIstructShtReg} (compare also \cite[Theorem 6.2.1]{Katz1985}), it suffices to check it is an isomorphism at the completed local rings at the unique points lying over supersingular points over algebraically closed fields in characteristic $\pf$.
		
		
		Let $\underline\Ec_0 \in\Sht_r(\ell)$ be a supersingular rank $r$ Drinfeld shtuka over some algebraically closed field $\ell$ in characteristic $\pf$ and let $\underline\Ec$ be its universal formal deformation over $\tilde B = \ell\dsq{\varpi, T_1, \ldots, T_{r-1}, T_r, \ldots, T_{2r-2}}$. We denote by $B = \ell\dsq{\varpi, T_1, \ldots, T_{r-1}}$. Note that $B$ pro-represents the deformation functor of the local shtuka of $\underline \Ec_0$ at $\pf$.  
		Recall that $(\Sht_{r, (\pf^{-n}/\Oc_0)})_{\underline\Ec} = \Spec(\tilde B_0)$ is an affine scheme, which is finite over $ \tilde B$ and that $\tilde B_0$ is a complete regular noetherian ring by Theorem \ref{thmIstructShtReg}.  
		Then $\Xc_{1, \underline\Ec} = \Spec(\tilde B_1)$ and $\Xc_{2, \underline\Ec} = \Spec(\tilde B_2)$ are affine schemes finite over $\tilde B_0$ (and therefore also over $\tilde B$). 
		Thus, $\tilde B_1$ and $\tilde B_2$ are complete, local and noetherian rings. We have to check that the map $\tilde B_2 \to \tilde B_1$ is an isomorphism. 
		By the Serre-Tate theorem (which is clearly compatible with all the relevant level structures as they only depend on the $\pf^n$-torsion), we can write $\tilde B_i = B_i \hat \otimes_\ell \ell\dsq{T_r, \ldots T_{2r-2}}$ for some complete, local and noetherian rings $B_i$ finite over $B$. Moreover, $B_0$ is regular. It clearly suffices to check that $B_2 \to B_1$ is an isomorphism. 
		Note that since $D \se \Hb^\times$  is a closed immersion, we obtain that the map $B_2 \rightarrow B_1$ is surjective.
		
		By Corollary \ref{corCompItorUniv}, we may assume that we can identify $\underline{\Ec}[\pf^n]$ in an $A$-linear fashion with the $\pf^n$-torsion of a Drinfeld $A$-module $\Eb$ with trivial underlying vector bundle (as the base $B$ is local). We have the following explicit descriptions of the rings $B_0, B_1$ and $B_2$.
		As $B$-algebras we find
		$$ B_0 = B[P]/\Ic,$$
		where $\Ic$ is the ideal expressing the fact that the map $\iota \colon \pf^{-n}/\Oc_0 \to B[P] = \Eb(B[P])$ defined by $\varpi^{-n} \mapsto P$ is a well-defined $(\pf^{-n}/\Oc_0)$-structure. More precisely, $\Ic$ is generated by $e_{\varpi^n}(P), e_a(P)$, where $(\varpi^n, a) = \pf^n$ in $A$, (this implies that $P \in \underline\Ec[\pf^n](B_0)$ and thus that the map can be extended $A$-linearly to a well-defined map $\iota$)  and the equations defining the condition that $\sum_{\alpha \in \pf^{-1}/\Oc_0} [\iota(\alpha)] = \sum_{\alpha \in \Oc_0/\pf}  [e_{\alpha \varpi^{n-1}} (P)]$ 
		is a subscheme of $\underline\Ec[\pf]$ (this condition is defined by finitely many equations by \cite[Lemma 1.3.4]{Katz1985}). Recall that $B_0$ is a regular local ring with maximal ideal generated by $P, T_1, \ldots, T_{r-1}$ by the proof of Theorem \ref{thmIstructShtReg}. In a similar fashion the rings $B_1$ and $B_2$ are given as $B_0$-algebras as
		$$ B_1 = B_0[Q]/\Jc,$$
		where $\Jc$ is the principal ideal generated by $\prod_{\alpha \in \Oc_0/\pf^{n}} (Q - e_\alpha(P))$, and
		$$ B_2 = B_0[Q]/\Kc,$$
		where $\Kc$ is the ideal expressing the fact that $Q$ defines an $(\pf^{-n}/\Oc_0)$-structure as above and defines the same submodule scheme as $P$, i.e. $\Kc$ is generated by $e_{\varpi^n}(Q), e_a(Q)$, where $(\varpi, a) = \pf$ in $A$, the equations defining the condition that $\sum_{\alpha \in \Oc_0/\pf}  [e_{\alpha \varpi^{n-1}}(Q)]$ is a subscheme of $\underline\Ec[\pf]$ and the coefficients of the polynomial $\prod_{\alpha \in \Oc_0/\pf^n} (t- e_\alpha(Q)) - \prod_{\alpha \in \Oc_0/\pf^{n}} (t- e_\alpha(P))$.
		
		By \cite[Lemma 6.3.4]{Katz1985}, the multiplication by $Q$ on $B_1$ is injective.
		We denote by $K$ the kernel of the map $B_2 \to B_1$. Applying the snake lemma to the diagram
		$$
		\begin{tikzcd}
			0 \arrow[r] & K \arrow[r]\arrow[d] & B_2 \arrow[r]\arrow[d] & B_1 \arrow[r]\arrow[d]& 0 \\
			0 \arrow[r] & K\arrow[r] & B_2 \arrow[r] & B_1 \arrow[r] & 0,
		\end{tikzcd}
		$$
		where the vertical maps are given by multiplication by $Q$,
		yields the short exact sequence (using the injectivity of multiplication by $Q$ on $B_1$)
		$$ 0 \to K/QK \to B_2/QB_2 \to B_1/QB_1 \to 0.$$
		By Nakayama's Lemma $K$ vanishes if and only if $K/QK$ vanishes. It thus suffices to show that 
		$$B_2/QB_2 \to B_1/QB_1$$
		is an isomorphism.
		
		We proceed as in \cite[Lemma 6.3.5]{Katz1985}. From the explicit description of $B_1$ and $B_2$ above we get that
		$$ B_1 /QB_1 = B_0/\overline{\Jc} \qquad \text{and} \qquad B_2/QB_2 = B_0/\overline{\Kc},$$
		where $\overline{\Jc}$ is the ideal generated by $\prod_{\alpha} e_\alpha(P)$ and $\overline{\Kc}$ is the ideal generated by the coefficients of the polynomial $t^{ q^n} - \prod_{\alpha} (t-e_\alpha(P))$, and the reductions of $e_{\varpi^n}(Q), e_a(Q)$, where $(\varpi, a) = \pf$ and the equations defining the condition that $\sum_{\alpha \in \Oc_0/\pf}  [e_{\alpha \varpi^{n-1}}(Q)]$ is a subscheme of $\underline\Ec[\pf]$ modulo $Q$. It suffices to show that $\prod_{\alpha} e_\alpha(P) \in \overline{\Kc}$. We show that it is (up to multiplication by a unit in $B_0$) the coefficient of the term of degree $ q^n -  q^{n-1}({q}-1)$ of $t^{ q^n} - \prod_{\alpha} (t-e_\alpha(P))$. 
		This coefficient is the sum of all $ q^{n-1}({q}-1)$-fold products of distinct elements of the set 
		$\{ e_\alpha(P) \colon \alpha \in \Oc_0/\pf^n\}$. 
		
		By the definition of Drinfeld modules it follows that 
		$e_\alpha(P)$ is of the form $(\text{unit in }B) \cdot P$ for $\alpha \in (\Oc_0/\pf^n)^\times$ and of the form $(\text{elt in } \max(B)) \cdot P$ for $\alpha \in \pf$. 
		Thus, both $\prod_{\alpha} e_\alpha(P)$ and the coefficient of the term of degree $q^n -  q^{n-1}({q}-1)$ of are of the form $(\text{unit in } B) \cdot P^{\phi(q^n)}$. 
	\end{proof}
	
	It would be desirable to have a similar statement also for other types of level structures.
	
\subsection{$\Gamma_0(\pf^n)$-level structures on Drinfeld shtukas.}

\begin{defn}
	\label{defnGamma0Sht}
	A \emph{$\Gamma_0(\pf^{n})$-level structure} on a Drinfeld shtuka $\underline{\Ec}$ over a scheme $S$ is a chain of $\pf^{n}$-cyclic isogenies 
	$$ \underline\Ec_r = \underline\Ec(\pf^n) \stackrel{f_r}{\rightarrow} \underline\Ec_{r-1} \stackrel{f_{r-1}}{\rightarrow} \underline\Ec_{r-2} \rightarrow \ldots \stackrel{f_1}{\rightarrow} \underline\Ec_0 = \underline{\Ec}  $$
	such that the composition $f_1 \circ \ldots \circ f_r$
	is the inclusion $\underline \Ec(\pf^n) \hookrightarrow \underline\Ec$. We denote the stack of Drinfeld shtukas with $\Gamma_0(\pf^{n})$-level structures by  $\Sht_{r, \Gamma_0(\pf^{n})}$. 
\end{defn}
Using the finite shtuka equivalence, a $\Gamma_0(\pf^{n})$-level structure on $\underline{\Ec}$ is equivalently given by a flag
$$ 0 = \Hb_0 \se \Hb_1 \se \Hb_2 \se \ldots \se \Hb_r = \underline\Ec[\pf^{n}] $$
of finite locally free submodule schemes $\Hb_i \se \underline\Ec[\pf^{n}]$ of rank $n \cdot i$ with strict $\Fq$-action such that $\Hb_i/\Hb_{i-1}$ is $\pf^{n}$-cyclic for all $1 \leq i \leq r$. 
In particular, a $\Gamma_0(\pf^n)$-level structure can fppf-locally on the base be extended to a $\Gamma_1(\pf^n)$-structure by definition. 
By Corollary \ref{corLvlMapResBalSht}, such a level structure can fppf-locally be extended to a $(\pf^{-n}/\Oc_0)^{r-1}$-structure on $\underline{\Ec}$. We call such an extension a  \emph{$(\pf^{-n}/\Oc_0)^{r-1}$-generator} of the $\Gamma_0(\pf^n)$-level structure.

We can now show one of our main theorems, that Drinfeld  $\Gamma_0(\pf^n)$-level structures produce a regular moduli problem.
\begin{thm}
	\label{thmG0Reg}
		The stack $\Sht_{r, \Gamma_0(\pf^{n})}$  is a  regular Deligne-Mumford stacks locally of finite type over $\Fq$. The forgetful map $\Sht_{r, \Gamma_0(\pf^{n})} \to\Sht_r $ is schematic and finite flat. It is finite \'etale away from $0$. Moreover, the forgetful map $\Sht_{r, \Gamma_1(\pf^n)} \to\Sht_{r, \Gamma_0(\pf^{n})}$ is schematic, faithfully flat and locally of finite presentation.
\end{thm}

\begin{proof}	
	We follow the proof of \cite[Theorem 6.6.1]{Katz1985}.
	As cyclicity is a closed condition by Theorem \ref{thmMainThmCyc}, $\Sht_{r, \Gamma_0(\pf^{n})}$ is the closed substack of $\Sht_{r,\mathbf{1}^r-\pf^n\text{-chain}}$ over which the universal isogeny is cyclic. Thus, the forgetful map $\Sht_{r, \Gamma_0(\pf^{n})} \to\Sht_r$ is schematic and finite. It follows that $\Sht_{r, \Gamma_0(\pf^{n})}$ is a Deligne-Mumford stack of finite type over $\Fq$. 
	Moreover, the forgetful map $\Sht_{r, \Gamma_1(\pf^n)} \to\Sht_{r, \Gamma_0(\pf^{n})}$ is representable by the scheme of generators $(\Hb_1/\Hb_0)^\times \times (\Hb_2/\Hb_1)^\times \ldots \times (\Hb_r/\Hb_{r-1})^\times$ and thus in particular finite flat by Theorem \ref{thmMainThmCyc}. 
	Since  $\Sht_{r, \Gamma_1(\pf^n)}$ is regular by Theorem \ref{propShtBalStructReg}, it follows that $\Sht_{r, \Gamma_0(\pf^{n})}$ is also regular by \cite[VII, Theorem 4.8]{Altman1970}. Thus, also the map $\Sht_{r, \Gamma_0(\pf^{n})} \to\Sht_r$ is finite flat by Miracle Flatness \cite[§23]{Matsumura1986}. It is finite \'etale away from $0$ by Proposition \ref{propIGenSht}.
\end{proof}

\begin{remark}
	In a similar fashion we can also show the regularity of the moduli stack of Drinfeld shtukas together with a chain of $\pf^n$-cyclic isogenies of length $r' < r$. In other words,  a chain of $\pf^{n}$-isogenies 
	$$ \underline\Ec_{r'+1} = \underline\Ec_0(\pf^n) \stackrel{f_{r'+1}}{\rightarrow} \underline\Ec_{r'} \stackrel{f_{r'}}{\rightarrow} \underline\Ec_{r'-1} \rightarrow \ldots \stackrel{f_1}{\rightarrow} \underline\Ec_0 = \underline{\Ec}  $$
	such that $f_1, \ldots, f_{r'}$ are $\pf^n$-cyclic. The following corollaries have also obvious analogues in this setting.
	Note that we could generalise the argument to moduli spaces of  other kinds of cyclic isogenies provided we had an analogue of Theorem \ref{thmMainThmCyc}. 
\end{remark}

Using the flatness of our moduli problems, we show that there are well-defined level maps.
\begin{cor}
	\label{corCanSubModG0}
	Let $\underline\Ec \in\Sht_r(S)$ and let $(\Hb_i)_i$ be a $\Gamma_0(\pf^{n})$-level structure on $\underline\Ec$. Let $\underline{n} = (n_1, \ldots, n_{r-1})$ with $0 \leq n_{r-1} \leq \ldots \leq n_1 \leq n$. Then there is a canonical subscheme $\Hb_{\underline{n}} \se \Hb_{r-1}$, which is an $\Oc_0/\pf^n$-module subscheme $\Hb_{\underline{n}} \se \underline{\Ec}[\pf^{\max\{n_i\}}]$.
	Fppf-locally on $S$, the scheme $\Hb_{\underline{n}}$ is defined for any $(\pf^{-n}/\Oc_0)^{r-1}$-generator $\iota \colon (\pf^{-n}/\Oc_0)^{r-1} \to \underline\Ec(S)$ of $(\Hb_i)_i$ by the restriction $\iota|_{(\pf^{-n_1}/\Oc_0) \oplus \ldots \oplus (\pf^{-n_{r-1}}/\Oc_0) }$ as in Proposition \ref{propSchIStruct}. 
\end{cor}
\begin{proof}
	We follow the proof of \cite[Theorem 6.7.2]{Katz1985}.
	It suffices to construct $\Hb_{\underline{n}}$ fppf-locally on $S$. We may thus assume that $(\Hb_i)_i$ admits a $(\pf^{-n}/\Oc_0)^{r-1}$-generator. Let $\iota$ and $\iota'$ be two such $(\pf^{-n}/\Oc_0)^{r-1}$-generators of $(\Hb_i)_i$.
	By Proposition \ref{propSchIStruct} the restrictions to $(\pf^{-n_1}/\Oc_0) \oplus \ldots \oplus (\pf^{-n_{r-1}}/\Oc_0)$ of both $\iota$ and $\iota'$ define closed submodule schemes $\Hb_{\underline{n}}, \Hb'_{\underline{n}} \se \underline\Ec[\pf^m]$.  
	We have to check $\Hb_{\underline{n}} = \Hb'_{\underline{n}}$.
	
	By reduction to the universal case we may assume that $S$ is Noetherian and flat over $X'$, as the moduli space of Drinfeld shtukas with $\Gamma_0(\pf^n)$-level together with two generators is given by 
	$$\Sht_{r, (\pf^{-n}/\Oc_0)^{r-1}} \times_{\Sht_{r, \Gamma_0(\pf^n)}}\Sht_{r, (\pf^{-n}/\Oc_0)^{r-1}},$$ 
	which is flat over $X'$ by Theorem \ref{thmG0Reg} and Corollary \ref{corLvlMapResBalSht}. In this case equality of closed subschemes of $\underline\Ec[\pf^{m}]$ is a closed condition by \cite[Lemma 6.7.3]{Katz1985}. The assertion is clear away from $0$ and thus follows from the flatness of $S$ in the general case.
\end{proof}
Via the finite shtuka equivalence and Proposition \ref{propKerIsog} the submodule scheme $\Hb_{\underline{n}} \se \underline{\Ec}[\pf^n]$ corresponds to a $\pf^n$-isogeny $\underline{\Ec}_{\underline{n}} \hookrightarrow \underline{\Ec}$.
For $1 \leq m \leq n$ we denote by $\underline{m}^{(i)} = (m, \ldots, m, 0, \ldots,0)$ with $i$ non-zero entries.
Then  
$$0 = \Hb_{\underline{m}^{(0)}} \se \Hb_{\underline{m}^{(1)}} \se \ldots \se \Hb_{\underline{m}^{(r-1)}} \se \underline{\Ec}[\pf^m]$$
is a $\Gamma_0(\pf^m)$-level structure on $\underline{\Ec}$. 
This shows that we have a well-defined level map $\Sht_{r, \Gamma_0(\pf^n)} \to\Sht_{r, \Gamma_0(\pf^m)}$ for all $0 \leq m \leq n$ which is automatically finite flat.
We show that we can also construct this level map by taking closures without making explicit reference to the generators.
\begin{cor}
	Let $S$ be a scheme which is flat over $X'$. Let $\underline{\Ec}$ be a Drinfeld shtuka over $S$ and
	let $(\Hb_i)_i$ be a $\Gamma_0(\pf^{n})$-level structure on $\underline{\Ec}$. For every $1 \leq i \leq r$ the canonical submodule scheme $\Hb_{\underline{m}^{(i)}} \se \Hb_i$ is the schematic closure of $\Hb_i|_{S \times_{X'} (X'\setminus\{0\})}[\pf^m]$ in $\underline\Ec[\pf^m]$.
\end{cor}
\begin{proof}
	From the explicit descriptions away from $\pf$ it is clear that $\Hb_{\underline{m}^{(i)}}$ is given by the $\pf^m$-torsion of $\Hb_i$ away from $0$.
	The assertion then follows from the fact that $\Hb_{\underline{m}^{(i)}}$ is flat over $S$ and closed in $\underline{\Ec}{[\pf^m]}$.
\end{proof}

Motivated by the discussion in Section \ref{secBT}, we also construct additional level maps.
Recall that the $\Gamma_0(\pf^n)$-level corresponds to a standard $(r-1)$-simplex $\Omega$ of sidelength $n$ in the standard apartment of the Bruhat-Tits building of $\GL_r$. 
We want to construct level maps corresponding to inclusions of sub-$(r-1)$-simplices (of smaller sidelength).
Recall that we enumerated alcoves in the standard apartment by its basepoint $\underline{m}$ and its orientation given by a permutation $\tau$.
In a similar fashion, a $(r-1)$-subsimplex of $\Omega$ is determined by its basepoint $\underline{m} = (m_1, \ldots, m_{r-1})$ with $m_1 \geq \ldots \geq m_{r-1}$, a sidelength $\tilde n$ and an orientation given by some $\tau \in \Sym_{r-1}$. 
Let us denote by $\underline{\tilde n}_\tau^{(i)} \in \Z^{r-1}$ the vector containing $\tilde n$ in entries $\tau(1), \ldots, \tau(i)$ and 0 otherwise.
Note that the simplex with basepoint $\underline{m}$, sidelength $\tilde n$ and orientation $\tau \in \Sym_{r-1}$ is contained in $\Omega$ if and only if  $\underline{m} + \underline{\tilde n}_\tau^{(i)} \se \Omega$.
\begin{cor}
	\label{corCanSubQuot}
	Let $\underline\Ec \in\Sht_r(S)$ and let $(\Hb_i)_i$ be a $\Gamma_0(\pf^{n})$-level structure on $\underline\Ec$. Let $\underline{m} = (m_1, \ldots, m_{r-1})$ with $0 \leq m_{r-1} \leq \ldots \leq m_1 \leq n$. 
	Let $0 \leq \tilde{n} \leq n$ such that $\underline{m} + \underline{\tilde n}_\tau^{(i)} \se \Omega$ for all $i$. 
	Then the flag of quotients
	$$ 0 \se \Hb_{\underline{m} + \underline{\tilde n}_\tau^{(1)}} / \Hb_{\underline{m}}  \se \Hb_{\underline{m} + \underline{\tilde n}_\tau^{(2)}} / \Hb_{\underline{m}}  \se  \ldots \se \Hb_{\underline{m} + \underline{\tilde n}_\tau^{(r-1)}} / \Hb_{\underline{m}} \se \underline{\Ec}_{\underline{m}}[\pf^{\tilde n}]$$
	defines a $\Gamma_0(\pf^{\tilde n})$-level structure on the Drinfeld shtuka $\underline{\Ec}_{\underline{m}}$, which we denote by $(\Hb_{\underline{m} + \underline{\tilde n}_\tau^{(i)}} / \Hb_{\underline{m}})_i$.
	In case that $(\Hb_i)_i$ admits a $(\pf^{-n}/\Oc)^{r-1} $-generator $\iota$,
	$(\Hb_{\underline{m} + \underline{\tilde n}_\tau^{(i)}} / \Hb_{\underline{m}})_i$ is generated by 
	\begin{align*}
		\iota_{\underline{m}, \tilde{n}, \tau}  \colon (\pf^{-\tilde{n}}/\Oc)^{r-1} & \cong \left({(\pf^{-m_{\tau(1)}-\tilde{n}}/\Oc) \oplus \ldots \oplus (\pf^{-m_{\tau(r-1)}-\tilde{n}}/\Oc) } \right)/ \left( {(\pf^{-m_{\tau(1)}}/\Oc) \oplus \ldots \oplus (\pf^{-m_{\tau(r-1)}}/\Oc) } \right) \\
		& \to  \Hb_{\underline{m} + \underline{\tilde n}^{(r-1)}} / \Hb_{\underline{m}} (S).
	\end{align*}
	Moreover, the canonical subscheme from Corollary \ref{corCanSubModG0} for $\underline{m}' = (m'_1, \ldots , m'_{r-1})$ with $\tilde n \geq m'_1 \geq \ldots \geq m'_{r-1} \geq 0$ is  given by 
	$$(\Hb_{\underline{m} + \underline{\tilde n}_\tau^{(i)}} / \Hb_{\underline{m}})_{\underline{m}'} \cong \Hb_{\underline{m} + \tau(\underline{m}')} / \Hb_{\underline{m}}.$$ 
\end{cor}
\begin{proof}
	We follow the proof of \cite[Theorem 6.7.4]{Katz1985}. The question is fppf-local on $S$, so we can assume that $\Hb$ has a generator. 
	By reduction to the universal case, we may further assume that $S$ is flat over $X'$ and Noetherian by Theorem \ref{thmIstructShtReg}. 
	Note that all assertions are clear away from $0$. It thus suffices to show that the locus, where they are satisfied is closed.
	
	The locus where each $\Hb_{\underline{m} + \underline{\tilde n}_\tau^{(i)}}/\Hb_{\underline{m} + \underline{\tilde n}_\tau^{(i-1)}}$ is ${\pf^{\tilde n}}$-cyclic is closed in $S$ by Theorem \ref{thmMainThmCyc} (\ref{thmMainThmCycClosed}). This shows the first claim.
	For the second claim, the locus where $\iota_{\underline{m}, \tilde n, \tau}$ is a generator of $(\Hb_{\underline{m} + \underline{\tilde n}_\tau^{(i)}} / \Hb_{\underline{m}})_i$ is closed by \cite[Proposition 1.9.1]{Katz1985}.
	Moreover, the condition that $(\Hb_{\underline{m} + \underline{\tilde n}_\tau^{(i)}} / \Hb_{\underline{m}})_{\underline{m}'} \cong \Hb_{\underline{m} + \tau(\underline{m}')} / \Hb_{\underline{m}}$ is closed by \cite[Lemma 6.7.3]{Katz1985}. This shows the last claim.  
\end{proof}
Associating the $\Gamma_0(\pf^{\tilde n})$-level structure  $(\Hb_{\underline{m} + \underline{\tilde n}_\tau^{(i)}} / \Hb_{\underline{m}})_i$ on $\underline{\Ec}_{\underline{m}}$ to $(\Hb_i)_i$ as in the previous corollary defines a map of stacks 
$$ F_{\underline{m},\tilde n, \tau} \colon\Sht_{r,\Gamma_0(\pf^n)} \to\Sht_{r,\Gamma_0(\pf^{\tilde n})}.$$

\begin{prop}
	\label{propLevelMapDrinfeld}
	The level map $F_{\underline{m},\tilde n,\tau}$ is schematic and finite locally free.
\end{prop}
\begin{proof}
	Note that by Theorem \ref{thmG0Reg} the map $F_{\underline{0}, 0}$ is schematic and finite locally free.
	As a first step we show that $F_{\underline{n},0} \colon\Sht_{r,\Gamma_0(\pf^n)} \to\Sht_{r}$ is schematic and finite locally free for all $\underline{m}$.
	In order to show that the map is representable by a finite scheme, we consider the auxiliary moduli problem $\Sht_{r, \underline{m}-\text{isog}, \Gamma_0(\pf^n)}$ parametrising a Drinfeld shtuka $\underline{\Ec}$, a $\pf^n$-isogeny $f \colon \underline{\Ec} \hookrightarrow \underline{\Ec}'$ such that $\coker(f)$ has rank $\sum_{i=1}^{r-1} m_i$ as an $\OS$-module, and a $\Gamma_0(\pf^n)$-level structure $(\Hb_i)_i$ on $\underline{\Ec}'$. 
	The projection to $\underline{\Ec}$ defines then a map of stacks $\Sht_{r, \underline{m}-\text{isog}, \Gamma_0(\pf^n)} \to\Sht_r$ which is schematic and finite by Proposition \ref{propIIsoRepSht} and Theorem \ref{thmG0Reg}.
	We also have a map $\Sht_{r, \Gamma_0(\pf^n)} \to\Sht_{r, \underline{m}-\text{isog}, \Gamma_0(\pf^n)}$ sending $(\underline{\Ec}, (\Hb_i)_i)$ to $(\underline{\Ec}_{\underline{m}}, \underline{\Ec}, (\Hb_i)_i)$, which identifies $\Sht_{r, \Gamma_0(\pf^n)}$ with the substack of  $\Sht_{r, \underline{m}-\text{isog}, \Gamma_0(\pf^n)}$ where $\underline{\Ec} = \underline{\Ec}'_{\underline{m}}$. By \cite[Lemma 6.7.3]{Katz1985}, this is schematic and representable by a closed immersion. 
	The composition of the maps $\Sht_{r, \Gamma_0(\pf^n)} \to\Sht_{r, \underline{m}-\text{isog}, \Gamma_0(\pf^n)} \to\Sht_r$ is clearly given by $F_{\underline{m}, 0}$, which is thus schematic and finite.
	
	Note that we have a commutative diagram
	\begin{center}
	\begin{tikzcd}
		\Sht_{r,\Gamma_0(\pf^n)} \arrow[rr, "F_{\underline{m}, \tilde n}"] \arrow[rd, "F_{\underline{m}, 0}"] & &\Sht_{r,\Gamma_0(\pf^{\tilde n})} \arrow[ld, "F_{\underline{0}, 0}"] \\
		&\Sht_r &
	\end{tikzcd}
	\end{center}
	with vertical arrows that are schematic and finite. In order to see that $ F_{\underline{m}, \tilde n, \tau} $ is schematic we argue as follows.
	We fix a map $S' \to\Sht_{r,\Gamma_0(\pf^{\tilde n})}$ from some $\Fq$-scheme $S'$, in other words a Drinfeld shtuka $(\underline{\Ec}, (\Hb_i)_i)$ together with a $\Gamma_0(\pf^{\tilde n})$-level structure. By composition with $F_{\underline{0},0}$, we get a map $S' \to\Sht_r$. By the discussion above, $S'' = S' \times_{\Sht_r} \Sht_{r,\Gamma_0(\pf^n)}$ is representable by a finite $S'$-scheme. Let $(\underline{\Ec}', (\Hb'_i)_i)$ denote the corresponding $\Gamma_0(\pf^n)$-level structure.
	Then the fibre product $S' \times_{\Sht_{r,\Gamma_0(\pf^{\tilde n})}} \Sht_{r,\Gamma_0(\pf^n)}$ is the locus where the image of $(\underline{\Ec}', (\Hb'_i)_i)$ under $ F_{\underline{m}, \tilde n, \tau} $ is given by  $(\underline{\Ec}, (\Hb_i)_i)$. By \cite[Lemma 6.7.3]{Katz1985}, this is representable by a closed subscheme of $S''$.
	
	As both $F_{\underline{m}, 0}$ and $F_{\underline{0},0}$ are finite, it is immediate that $F_{\underline{m}, \tilde n, \tau}$ is finite as well. As both $\Sht_{r. \Gamma_0(\pf^n)}$ and $\Sht_{r, \Gamma_0(\pf^{\tilde n})}$ are regular and $(2r-1)$-dimensional, the level map is flat by miracle flatness.
\end{proof}

	\section{Comparison with naive level structures and Bruhat-Tits theory}
\label{secComp}

We compare the Drinfeld level structures defined above with naive $\Gamma_0(\pf^n)$-level structures.
The naive $\Gamma_0(\pf^n)$-level structures seem inadequate when $n > 1$ as the fibre above $0$ is missing points (compare Remark \ref{remNaiveLvlBad}).
We construct a map from the stack of Drinfeld shtukas with naive $\Gamma_0(\pf^n)$-level structure to our stack of Drinfeld shtukas with Drinfeld $\Gamma_0(\pf^n)$-level which is  an open immersion and  an isomorphism away from $0$. Moreover, we show that the two notions of level structures agree in the parahoric case.	
In this sense, the Drinfeld level structures provide a compactification of the level maps.

Recall that we defined 
a naive $\Gamma_0(\pf^n)$-level structure on a Drinfeld shtuka $\underline{\Ec} = (\Ec, \ph)$ of rank $r$ as a flag of quotients as $\pf^n$-torsion finite shtukas
$$ \underline{\Ec}|_{D_{n,S}} = \underline\Lc_r \twoheadrightarrow \underline\Lc_{r-1} \twoheadrightarrow \ldots \twoheadrightarrow  \underline\Lc_{1} \se \underline\Lc_0 =  0$$ 
such that $\Lc_i$ is finite locally free of rank $i$ as $\Oc_{D_{n,S}}$-module. 
Equivalently, using Proposition \ref{propKerIsog}, a naive $\Gamma_0(\pf^n)$-level structure is a chain of $\pf^n$-isogenies
$$ \underline{\Ec}(\pf^n) = \underline\Ec_r  \stackrel{f_r}{\rightarrow} \underline\Ec_{r-1} \stackrel{f_{r-1}}{\rightarrow} \underline\Ec_{r-2} \rightarrow \ldots \stackrel{f_1}{\rightarrow} \underline\Ec_0 = \underline{\Ec} $$
such that $\coker(f_i)$ is finite locally free of rank 1 as $\ODnS$-module.

\begin{lem}
	\label{lemLinBunCyc}
	Let $\underline{\Ec} = (\Ec, \ph)$ be a Drinfeld shtuka of rank $r$ over $S$ and let $\Ec|_{D_{n,S}} \twoheadrightarrow \Lc$ be a quotient $\pf^n$-torsion finite shtuka such that $\Lc$ is finite locally free of rank 1 as $\ODnS$-module.
	Then $\Dr_q(\underline{\Lc}) \se \underline{\Ec}[\pf^n]$ is a $\pf^n$-cyclic submodule scheme.
\end{lem}
\begin{proof}
	We denote by $\underline{\Lc}^{(i)} = \underline{\Lc}|_{D_{i, S}}$ for $1 \leq i \leq n$. Then $\underline{\Lc}^{(i)}$ is a locally free $\Oc_{D_{i, S}}$-module of rank 1, and consequently a locally free $\OS$-module of rank $i$.
	Thus,
	$$ \underline{\Lc} = \underline{\Lc}^{(n)} \twoheadrightarrow \underline{\Lc}^{(n-1)} \twoheadrightarrow \ldots \twoheadrightarrow \underline{\Lc}^{(2)} \twoheadrightarrow \underline{\Lc}^{(1)} \twoheadrightarrow 0$$
	corresponds via the finite shtuka equivalence to a flag of finite locally free submodule schemes with strict $\Fq$-action
	$$ 0 \se \Hb^{(1)} \se \ldots \se \Hb^{(n-1)} \se \Hb^{(n)} \se \underline{\Ec}[\pf^n],$$
	where we denote by $\Hb^{(i)} = \Dr_q(\underline{\Lc}^{(i)})$. It is clear that $\Hb^{(i)} \se \Eb[\pf^i]$ by construction.  
	As a next step, we inductively construct a generator of $\Hb^{(i)}$ fppf-locally on $S$. 
	
	We may assume that $S = \Spec(R)$ is affine and that $\Lc$ is a free $\Oc_{D_{n,S}} = R[\varpi]/(\varpi^n)$-module of rank 1. Then, $\Lc^{(i)} \cong R[\varpi]/(\varpi^i)$. 
	We choose the standard basis $1, \varpi, \ldots, \varpi^{i-1}$ of $\Lc^{(i)}$ as $R$-module. 
	As a map of finite free $R[\varpi]/(\varpi^i)$-modules, $\ph$ is given by multiplication by an element $\alpha = \sum_{j = 0}^{i-1} \alpha_j \varpi^j \in R[\varpi]/(\varpi^i)$, and thus, its matrix as an $R$-linear map with respect to the standard basis is given  by 
	$$ \begin{pmatrix}
		\alpha_0 &  & & \\
		\alpha_1 &  \alpha_0 &  &  \\
		\vdots &  & \ddots \\
		\alpha_i & \alpha_{i-1} & \ldots & \alpha_0 
	\end{pmatrix}.$$
	It follows that 
	$$ \Hb^{(i)} = \Dr_q(\underline{\Lc}^{(i)}) =  \Spec\left(  R[t_0, \ldots, t_{i-1}] / (t_0^q - \sum_{j = 0}^{i-1}\alpha_j t_j, t_1^q - \sum_{j = 1}^{i-1}\alpha_{j-1} t_j, \ldots, t_{i-1}^q - \alpha_0 t_{i-1} )\right).$$
	As the question is fppf-local on $R$, we may assume that
	$R$ contains a root $\beta_0$ of the polynomial $t^{q-1} - \alpha_0$, a root $\beta_1$ of the polynomial $t^q - \alpha_0 t - \alpha_1 \beta_0$ and inductively a root $\beta_j$ of the polynomial $t^q - \alpha_0 t - \alpha_{1} \beta_{j-1} - \ldots - \alpha_{j} \beta_0$ for all $0 \leq j \leq i-1$. Then $(\beta_{i-1}, \beta_{i-2}, \ldots, \beta_1, \beta_{0})$ is a section of $\Hb^{(i)}$ over $R$ by construction. We claim that the map
	\begin{align*}
		\iota^{(i)} \colon \pf^{-i}/\Oc_0 & \to \Hb^{(i)}(R) \\
		\varpi^{-i} & \mapsto (\beta_{i-1}, \beta_{i-2}, \ldots, \beta_1, \beta_{0})
	\end{align*}
	is a $\pf^{-i}/\Oc_0$-generator of $\Hb^{(i)}$. We proceed by induction on $i$.
	 
	Let $i =1$. In this case $\Hb^{(1)} = \Spec(R[t]/(t_0^q - \alpha_0 t_0))$. In particular, $\Hb^{(1)}$ can be embedded in $\A^1_R$. Then
	$ \pf^{-1}/\Oc_0 \to \Hb^{(1)}(R)$ given by $\varpi^{-1} \mapsto \beta$ is a generator of $\Hb^{(1)}$, as $\prod_{a \in \Fq} (t - a \beta) = t^q - \beta^{q-1} t = t^q - \alpha t$. 
	
	Let us now assume that the claim is true for $i \geq 1$. Note that the subscheme $\Hb^{(i)} \se \Hb^{(i+1)}$ is given by the locus where $t_i = 0$ by construction.	
	Note that the map $\iota^{(i+1)}|_{\pf^{-i}/\Oc_0} = \iota^{(i)}$ is given by $\varpi^{-i} \mapsto  (\beta_{i-1}, \beta_{i-2}, \ldots, \beta_1, \beta_{0}, 0)$. Thus, it factors through $\Hb^{(i)}$ and is a generator of $\Hb^{(i)}$ by hypothesis. 
	
	Moreover, the quotient $\Hb^{(i+1)} /\Hb^{(i)}$ is then given by the canonical inclusion $$R[t_{i}]/(t_{i}^q - \alpha_{0}t_{i}) \to R[t_0, \ldots, t_{i-1}, t_i] / (  t_0^q - \sum_{j = 0}^{i}\alpha_j t_j, \ldots, t_{i}^q - \alpha_0 t_{i}).$$
	Moreover, the image of the section $(\beta_0, \ldots, \beta_{i-1})$ of $\Hb^{(i+1)}$ in the quotient $\Hb^{(i+1)}/\Hb^{(i)}$ is $\beta_0$. In particular, the map 
	$$\iota^{(i+1)} \! \mod \pf^{(i)} \colon  \pf^{-1}/\Oc_0 \cong (\pf^{-(i+1)}/\Oc_0)/(\pf^{-i}/\Oc_0) \to \left( \Hb^{(i+1)}/\Hb^{(i)}\right)(R)$$ is well-defined and
	sends $\varpi^{-1}$ to $\beta_0$ and is thus a generator of $\Hb^{(i+1)}/\Hb^{(i)}$ by the discussion of the case $i =1$ above. 
	By \cite[Lemma 1.11.3]{Katz1985} it follows that $\iota^{^{(i+1)}}$ is a full set of sections of $\Hb^{(i+1)}$. 
	
	Thus, $(\Hb^{(1)}, \iota)$ is a generator of $\Hb^{(i+1)}$ in the sense of Definition \ref{defnIstructSht} and $\Hb^{(i+1)}$ is $\pf^{i+1}$-cyclic.
	This shows the claim.
\end{proof}

\begin{prop}
	Let $\underline{\Ec}=(\Ec, \ph)$ be a Drinfeld shtuka over $S$ and let
	$$ \Ec|_{D_{n,S}} = \Lc_r \twoheadrightarrow \Lc_{r-1} \twoheadrightarrow \ldots \twoheadrightarrow  \Lc_{1} \twoheadrightarrow \Lc_0 =  0$$
	be a naive $\Gamma_0(\pf^n)$-level structure on $\underline{\Ec}$. Then
	$$ 0 \se \Dr_q(\underline{\Lc}_1) \se \ldots \Dr_q(\underline{\Lc}_{r-1}) \se \underline{\Ec}[\pf^n] $$
	is a Drinfeld $\Gamma_0(\pf^n)$-level structure on $\underline{\Ec}$ in the sense of Definition \ref{defnGamma0Sht}.
\end{prop}
\begin{proof}
	This follows directly from Lemma \ref{lemLinBunCyc}.
\end{proof}

Recall that a Drinfeld shtuka with a naive $\Gamma_0(\pf^n)$-level structure is a bounded global $\GL_{r, \Omega}$-shtuka for the Bruhat-Tits group scheme $\GL_{r, \Omega}$ as defined in Remark \ref{remG0BT}.
In particular, we thus constructed a map of Deligne-Mumford stacks
\begin{equation}
	\Sht_{r, \Omega} \to \Sht_{r, \Gamma_0(\pf^n)}.
	\label{eqnGamma0Map}
\end{equation}
As a next step, we show that the map (\ref{eqnGamma0Map}) is an isomorphism in the case $n =1$.
\begin{prop} 
	\label{propPara}
	Let $\underline{\Ec} \in \Sht_{r}(S)$. 
	Then every $\Gamma_0(\pf)$-level structure on $\underline{\Ec}$ comes from a naive $\Gamma_0(\pf)$-level structure.
\end{prop}
\begin{proof}
	It suffices to show that for a $\pf$-cyclic submodule scheme $\Hb \se \underline{\Ec}[\pf]$, the corresponding $\pf^n$-torsion shtuka $\underline{M}_q(\Hb)$ is finite locally free of rank 1 as an $R \otimes \Oc_0/\pf \cong R$-module. But this is clear by construction. 
\end{proof}

\begin{lem}
	\label{lemEt}
	Let $\underline{\Ec} \in\Sht_{r}(S)$ and assume its characteristic is away from $0$. 
	Then every $\Gamma_0(\pf^n)$-level structure on $\underline{\Ec}$ comes from a naive $\Gamma_0(\pf^n)$-level structure.
\end{lem}
\begin{proof}
	Let $(\Hb_i)_{1 \leq i \leq r}$ be a $\Gamma_0(\pf^n)$-structure on $\underline{\Ec}$. As  the characteristic of $\underline{\Ec}$ is away from $0$, all the $\Hb_i$ are finite \'etale over $S$.  As the claim is fppf-local on the base, we may choose a $(\pf^{-n}/\Oc_0)^{r-1}$-generator of $(\Hb_i)_{1 \leq i \leq r}$. 
	By Proposition \ref{propLevelEtaleSht}, the $\Hb_i$ are then given by 
	$$ 0 \se (\pf^{-n}/\Oc_0)_S \se (\pf^{-n}/\Oc_0)^2_S \se \ldots \se (\pf^{-n}/\Oc_0)^{r-1}_S \se \underline{\Ec}[\pf^n].$$
	By the finite shtuka equivalence, this corresponds to the flag of quotients
	$$ \underline{\Ec}|_{D_{n,S}}  \twoheadrightarrow  \ODnS^{r-1} \twoheadrightarrow  \ldots \twoheadrightarrow \ODnS \twoheadrightarrow 0,$$
	where the map $\ODnS^{i+1} \to \ODnS^{i}$ is given by the projection to the first $i$ components and the Frobenius-linear map on $\ODnS^{i}$ is the trivial one. This is clearly a naive $\Gamma_0(\pf^n)$-level structure. 
\end{proof}

\begin{prop}
	\label{prop1Open}
	The map (\ref{eqnGamma0Map}) is schematic and a quasi-compact open immersion.
\end{prop}
\begin{proof}
	By construction, $\Sht_{r, \Omega}$ is identified with the substack of $\Sht_{r, \Gamma_0(\pf^n)}$ where all the $\Hb_i$ correspond to finite locally free $\ODnS$-modules of rank $i$ via the finite shtuka equivalence, or equivalently the locus where all $\Hb_i/\Hb_{i-1}$ correspond to finite locally free $\ODnS$-modules of rank $1$.
	
	In order to show that this condition is representable by an open subscheme, we work locally and assume that $S = \Spec(R)$ is affine. Let $M$ be an $R[\varpi]/(\varpi^n)$-module such that $M$ is free of rank $n$ as $R$-module. Then $M$ is locally free of rank 1 as $R[\varpi]/(\varpi^n)$-module if and only if it is generated by a single element.
	
	Let $q \se R$ be a prime ideal such that $M \otimes \kappa(q)$ is a one-dimensional vector space over the residue field $\kappa(q)$ of $R[\varpi]/(\varpi^n)$ at $(q, \varpi)$ ($\kappa(q)$ is also the residue field of $R$ at $q$). By Nakayama's Lemma, there exists a $a \in (R[\varpi]/(\varpi^n))\setminus(q, \varpi)$ such that $M[a^{-1}]$ is free of rank 1.
	Let $a_0 = a(0)$ be the constant term of $a$. Then $M[a^{-1}] = M[a_0^{-1}]$. 
	Hence, the principal open $D(a_0) \se \Spec(R)$ is an open neighbourhood of $q$ such that $M[a_0^{-1}]$ is locally free of rank 1 as $R[a_0^{-1}][\varpi]/\varpi^n$-module over $D(a)$. Hence, the condition is representable by an open immersion on the base scheme.
\end{proof}

By Propositions \ref{propLevelMapDrinfeld} and \ref{propPara}, we can interpret the level maps to $\Gamma_0(\pf)$-level structures as maps 
$$F_{\underline{m}, 1, \tau} \colon\Sht_{r, \Gamma_0(\pf^n)} \to\Sht_{r, \ff_{\underline{m}, \tau}},$$ 
where $\ff_{\underline{m}, \tau}$ is the alcove in the Bruhat-Tits building corresponding to $\underline{m}$ and $\tau$. This system of maps is compatible with level maps to parahoric levels given by smaller facets by Corollary \ref{corCanSubModG0} and thus define a map
\begin{equation}
	F_{\Omega}\colon \Sht_{r, \Gamma_0(\pf^n)} \to \varprojlim_{\ff \prec \Omega}\Sht_{r,\ff}.
	\label{eqnGamma0MapJoin}
\end{equation} 
\begin{prop}
	\label{prop2Closed}
	The map $F_{\Omega}$ is a closed immersion.
\end{prop}

\begin{proof}
	As all the maps $\Sht_{r, \Gamma_0(\pf^n)} \to \varprojlim_{\ff \prec \Omega}\Sht_{r,\ff}$ are schematic and finite, so is their limit. Moreover, by the explicit moduli description it is clear that the map $F_{\Omega}$ is a monomorphism.
\end{proof}

\begin{thm}
	\label{thmDrinfeldClosureBT}
	The map $\Sht_{r,\Omega} \to \varprojlim_{\ff \prec \Omega} \Sht_{r,\ff}$ is schematic and representable by a quasi-compact open immersion that is an isomorphism away from $0$. Its schematic image in the sense of \cite{Emerton2021} is 
	$$\overline{\Sht}_{r, \Omega} = \Sht_{r, \Gamma_0(\pf^n)}$$ 
	via the maps 
	$$\Sht_{r,\Omega} \hookrightarrow \Sht_{r, \Gamma_0(\pf^n)} \hookrightarrow  \varprojlim_{\ff \prec \Omega} \Sht_{r,\ff}$$
	constructed above.
	In the parahoric case $n = 1$, the map $\Sht_{r,\Omega} \to \Sht_{r, \Gamma_0(\pf)}$ is an isomorphism.
\end{thm}

\begin{proof}
	The assertion for the parahoric case is Proposition \ref{propPara}. That the inclusion 
	$$\Sht_{r,\Omega} \to \varprojlim_{\ff \prec \Omega}\Sht_{r,\ff}$$ 
	is schematic and representable by a quasi-compact locally closed immersion follows from Propositions \ref{prop1Open} and \ref{prop2Closed}. That the image of  $\Sht_{r,\Omega}$ in $\Sht_{r, \Gamma_0(\pf^n)}$ is dense follows from the fact that the inclusion (\ref{eqnGamma0Map}) is an isomorphism away from $0$ by Lemma \ref{lemEt} together with the flatness of $\Sht_{r, \Gamma_0(\pf^n)}$ over $X'$ from Theorem \ref{thmG0Reg}. 
	
	In order to see that the map $\Sht_{r,\Omega} \to \varprojlim_{\ff \prec \Omega} \Sht_{r,\ff}$ is already open, we follow the proof of Proposition \ref{prop1Open}. One can again check that a point $(\Ec_{\underline{m}})_{\underline{m}} \in \varprojlim_{\ff \prec \Omega}\Sht_{r,\ff}$ comes from $\Sht_\Omega$ if and only if the cokernels of the isogenies $\underline{\Ec}_{(n, \ldots,n,0,0, \ldots,0)} \hookrightarrow \underline{\Ec}_{(n, \ldots,n,n,0, \ldots,0)}$ are locally free of rank 1 as $\ODnS$-modules. 
	By the argument in the proof of Proposition \ref{prop1Open}, this condition is representable by an open subscheme.
%
\end{proof}


	\bibliographystyle{alphaurl} 
	\bibliography{literature.bib}
	
\end{document}